\documentclass{article}[12pt]
\usepackage{amsmath,amssymb,amsthm,txfonts}
 \usepackage{color}

\title{Buchberger-Zacharias Theory of Multivariate Ore Extensions}

\author{
{\small\bf Michela Ceria}\\
{\small Dipartimento di Matematica}\\
{\small Universit\`{a} di Trento}\\
{\small{\tt michela.ceria@unitn.it}}
\and
\and
{\small\bf Teo Mora}\\
{\small DIMA}\\
{\small Universit\`a di Genova}\\
{\small{\tt theomora@disi.unige.it}}
}

\def\Forall{\mbox{\ for each }}
\def\Where{\mbox{\ where }}
\def\With{\mbox{\ with }}

\def\And{\mbox{\ and }}

\def\Bcc#1{\overline{\bf #1}}
\def\Bbb#1{{\mathbb #1}}
\def\Cal#1{{\cal #1}}
\def\Frak #1{{\mathfrak #1}}
\def\Can{\mathop{\rm Can}\nolimits}
\def\GM{{\mathfrak G}{\mathfrak M}}
\def\then{\;\Longrightarrow\;}

\def\Rep{\mathop{\bf Rep}\nolimits}
\def\Syz{\mathop{\rm Syz}\nolimits}
\def\Span{\mathop{\rm Span}\nolimits}
\def\supp{\mathop{\rm supp}\nolimits}
\def\lcm{\mathop{\rm lcm}\nolimits}
\def\lc{\mathop{\rm lc}\nolimits}
\def\lift{\mathop{\rm lift}\nolimits}

\def\NF{\mathop{\rm NF}\nolimits}
\def\op{\mathop{\rm op}\nolimits}
\def\Id{\mathop{\rm Id}\nolimits}
\begin{document}
\maketitle
\renewcommand\labelitemi{\bf --}
%
\def\qed{\ifmmode\squareforqed\else{\unskip\nobreak\hfil
\penalty50\hskip1em\null\nobreak\hfil\squareforqed
\parfillskip=0pt\finalhyphendemerits=0\endgraf}\fi}
\def\squareforqed{\hbox{\rlap{$\sqcap$}$\sqcup$}}
\def\pf{\par\noindent {\it Proof }}
\theoremstyle{plain}
\newtheorem{Theorem}{Theorem}
\newtheorem{Corollary}[Theorem]{Corollary}
\newtheorem{Lemma}[Theorem]{Lemma}
\newtheorem{Proposition}[Theorem]{Proposition}
\newtheorem{Fact}[Theorem]{Fact}
\theoremstyle{definition}
\newtheorem{Definition}[Theorem]{Definition}
\newtheorem{Problem}[Theorem]{Problem}
\newtheorem{Void}{}[section]
\theoremstyle{remark}
\newtheorem{Example}[Theorem]{Example}
\newtheorem{Remark}[Theorem]{Remark}
\newtheorem{Algorithm}[Theorem]{Algorithm}
\newtheorem{Procedure}[Theorem]{Procedure}
\newtheorem{History}[Theorem]{Historical Remark}
\newtheorem{Notation}[Theorem]{Notation}
\newtheorem{teo}{Remark$\times$teo}
        
\newcounter{saveq}
\newcommand{\alpheqn}{\setcounter{saveq}{\value{equation}}
\stepcounter{saveq}\setcounter{equation}{0}
\renewcommand{\theequation}
{\mbox{\arabic{saveq}-\alph{equation}}} 
}
\newcommand{\reseteqn}{\setcounter{equation}{\value{saveq}}%
\renewcommand{\theequation}{\arabic{equation}}}                 
\begin{abstract}  We present Buchberger Theory and Algorithm of Gr\"obner bases for multivariate Ore extensions of rings presented as modules over a principal ideal domain. The algorithms are based on M\"oller Lifting Theorem.
\end{abstract}

In her 1978 Bachelor's thesis \cite{Z} Zacharias discussed how to extend Buchberger Theory \cite{Bu1,Bu2,Bu3} from the case of polynomial rings over a field to that of polynomials over a Noetherian ring with suitable effectiveness conditions. 

In the meantime a similar task was performed in a series of papers --- Kandri-Ro\-dy--Ka\-pur \cite{KrK} merged the rewriting rules behind Euclidean Algorithm with Buchberger's rewriting, proposing a Buchberger Theory for polynomial rings over Euclidean domains; Pan \cite{P} studied Buchberger Theory for polynomial rings over domains introducing the notions of strong/weak Gr\"obner bases --- which culminated with \cite{M}. 

Such unsorpassed paper, reformulating and improving Zacharias' intuition, gave efficient solutions to compute 
both weak and strong Gr\"obner bases for polynomial rings over each Zacharias ring, with particolar attention to the PIR case. Its main contribution is the reformulation of Buchberger test/completion (``a basis $F$ is Gr\"obner if and only if each
S-polynomial  between two elements of $F$, reduces to 0'') in the more flexible {\em lifting theorem} (``a basis $F$ is Gr\"obner if and only if
each element in a minimal basis of the syzygies among the
leading monomials lifts, via Buchberger reduction, 
to a syzygy among the elements of $F$'').
The only further contribution to this ultimate paper is the survey \cite{m5} of M\"oller's results which reformulated them in terms of Szekeres Theory \cite{S}.

The suggestion of extending  Buchberger Theory to non-commuative
rings which satisfy Poincar\'e-Birkhoff-Witt Theorem was put forward by Bergman \cite{Be}, 
effectively applied by Apel--Lassner \cite{AL1,AL2} to Lie algebras and further extended to solvable polynomial rings \cite{KrW,Kr}, skew polynomial rings \cite{Ga1,Ga2,Ga3} and to other 
algebras \cite{A0,BGV,Lev,Lev2}
which satisfy Poincar\'e-Birkhoff-Witt Theorem and thus, under the standard intepretation of Buchberger Theory in terms of filration/graduations \cite{A,m1,MoSw,Sw,Bath},
have the classical polynomial ring  as associated graded rings.

In particular Weispfenning \cite{W} adapted his results to a generalization of the Ore extension \cite{O2} proposed by Tamari \cite{T} and its construction was generalized by his student Pesch \cite{P1,P2} thus introducing the notion of {\em iterative Ore extension with commuting variables}
$${\sf R} :=R[Y_1;\alpha_1,\delta_1][Y_2;\alpha_2,\delta_2]\cdots[Y_n;\alpha_n,\delta_n], \, R  \mbox{ a domain};$$ 
the concept has been called   {\em Ore algebra} in \cite{CS} and is renamed here as
{\em multivariate Ore extension} (for a different and promising approach to Ore algebras see \cite{LSL}).

Bergman's approach and most of all extensions are formulated for rings which are vector spaces over a field; in our knowledge the only instances in which the coefficient ring $R$  is presented as a ${\Bbb D}$-module over a domain ${\Bbb D}$ (or at least as a ${\Bbb Z}$-module) are Pritchard's \cite{Pr1,Pr2} extension of M\"oller Lifting Theory to non-commutative free algebras and Reinert's \cite{R2,R3} deep study of Buchberger Theory on Function Rings.

Following the recent survey on Buchberger-Zacharias Theory 
for monoid rings $R[{\sf S}]$ over a unitary effective ring $R$ and an effective monoid ${\sf S}$ \cite{Zac}, we propose here a M\"oller--Pritchard lifting theorem presentation of Buchberger-Zacharias Theory
and related Gr\"obner basis computation algorithms for multivariate Ore extensions. The twist w.r.t. \cite{Zac} is that there $R[{\sf S}]$ coincides with its associated graded ring; here ${\sf R}$ and its associated graded ring 
$$G({\sf R}) := R[Y_1;\alpha_1][Y_2;\alpha_2]\cdots[Y_n;\alpha_n]$$
 coincide as  sets and as left $R$-modules between themselves and with the commutative polynomial ring $R[Y_1,\cdots,Y_n]$,  but as rings  have different multiplications.

\bigskip

We begin by recalling Ore's original theory \cite{O2} of non-commutative polynomials $R[Y]$, relaxing the original assumption that $R$ is a field to the case in which $R$ is a domain (Section~1.1) and Pesch's constructions of multivariate Ore extensions (Section~1.2) and graded multivariate Ore extensions (Section~1.3), focusing on the arithmetics of the main Example~\ref{GOe}
$${\sf R} := R[Y_1;\alpha_1][Y_2;\alpha_2]\cdots[Y_n;\alpha_n], R:={\Bbb Z}[x]\quad
\alpha_i(x) := c_i x^{e_i}, c_i\in{\Bbb Z}\setminus\{0\}, e_i\in{\Bbb N}\setminus\{0\}.$$

Next we introduce Buchberger Theory in  multivariate Ore extensions recalling the notion of term-orderings (Section~2.1), definition and main properties of left, right, bilateral and restricted Gr\"obner bases (Section~2.2) and Buchberger Algorithm for computing canonical forms in the case in which $R$ is a field (Section~2.3). 

We adapt to our setting Szekeres Theory \cite{S} (Section~3), Zacharias canonical representation with related algorithm (Section~4) and M\"oller Lifting Theorem  (Section~5).

The next Sections are the algorithmic core of the paper:
we reformulate for multivariate Ore extensions over a Zacharias ring
\begin{itemize}
\item M\"oller's algorithm for computing the required Gebauer--M\"o\-ller set ({\em id est} the minimal basis of the module of the syzygies among the
leading monomials) for Buchberger test/com\-ple\-tion of left weak bases (Section~6.1); 
\item his reformulation of it  which requires only l.c.m. computation in $R$ for the  case in which $R$ is a PIR (Section~6.2); 
\item still in the  case in which $R$ is a PIR,  M\"oller's completion of a left weak basis to a left strong one (Section~6.3);
\item Gebauer--M\"oller criteria \cite{GM} for producing a Gebauer--M\"oller set  (Section~6.4);
\item Kan\-dri-Ro\-dy--Weis\-pfen\-ning completion \cite{KrW} of a left weak bases for producing a bilateral one  (Section~7.1);
\item Weis\-pfen\-ning's \cite{W} restricted completion  (Section~7.2);
\item as a technical tool required by Weis\-pfen\-ning's  restricted completion, how to produce right Gebauer--M\"oller sets   (Section~7.3).
\end{itemize}

Finally we reverse to a theoretical survey summarizing the structural theorem for the case in which $R$ is a Zacharias ring over a PID (Section~8), specializing to our setting Spear's Theorem \cite{Sp, Bath} (Section~9) and extending to it Lazard's Structural Theorem \cite{Laz} (Section~10).

In an appendix we discuss, as far as it is possible, how to extend this theory and algorithms to the case in which $R$ is a PIR  (Section~A).

\section{Recalls on Ore Theory} 

\subsection{Ore Extensions} 
Let $R$ be a not necessarily commutative domain; Ore  \cite{O2} investigated 
under which conditions 
the $R$-module
${\sf R} := R[Y]$ 
 of all the {\em formal polynomials}
{\em is made a ring}  under the assumption
{\em that the multiplication of polynomials shall be associative and both-sided distributive} and the limitation imposed by the postulate that
{\em the degree of a product shall be equal to the sum of the degree of the factors.}

It is clear that, due to the distributive property, it suffices to define the product of two monomials  $bY^r\cdot aY^s$ or even more specifically, to define the product
$Y \cdot r, r\in R$;
this necessarily requires the existence of maps  $\alpha,\delta : R \to R$ such that
$$Y\cdot r = \alpha(r) Y + \delta(r) \Forall r\in R;$$
Ore calls $\alpha(r)$ the {\em conjugate}\index{conjugate (Ore)} and
$\delta(r)$ the {\em derivative}\index{derivative (Ore)} of $r$.

Under the required postulate clearly we have
\begin{enumerate} 
\item for each $r\in R$, $\alpha(r) = 0 \then r = 0$, 
\end{enumerate} 
so that $\alpha$ is injective.

It is moreover sufficient to consider,  for each $r,r'\in R$, the relations
$$
\alpha(r+r') Y + \delta(r+r') = Y\cdot(r+r') = Y\cdot r+Y\cdot r' =
\bigl(\alpha(r)+\alpha(r')\bigr) Y + \delta(r)+\delta(r'),$$
$$\alpha(rr') Y + \delta(rr') = Y\cdot(rr') = \left(Y\cdot r\right)\cdot r' =
\alpha(r)\alpha(r') Y + \alpha(r)\delta(r')+\delta(r)r',$$
and, if $R$ is a skew field, and $r\neq 0$,
$$Y =   \left(Y\cdot r\right)\cdot r^{-1} =
\alpha(r)\alpha(r^{-1}) Y + \alpha(r)\delta(r^{-1})+\delta(r)r^{-1},$$
to deduce that
\begin{enumerate} 
\setcounter{enumi}{1}
\item $\alpha$ is a ring endomorphism; 
\item the following conditions are equivalent:
\begin{enumerate}
\item for each $d\in R\setminus\{0\}$ exists $c\in R\setminus\{0\} :
Y\cdot c=dY+\delta(c), \alpha(c)=d$;
\item $\alpha$ is a ring automorphism;
\end{enumerate}
\item $\delta$ is an $\alpha$-derivation of $R$\index{derivation} {\em id est} an additive map satisfying
$$\delta(rr') = \alpha(r)\delta(r')+\delta(r)r' \Forall r,r'\in R;$$
\item if $R$ is a skew field, then  each $r\in R\setminus\{0\}$ satisfies
$$\alpha(r^{-1}) = \left(\alpha(r)\right)^{-1}, \quad \delta(r^{-1}) = -\left(\alpha(r)\right)^{-1}\delta(r)r^{-1};$$
\item $Im(\alpha)\subset R$ is a subring, which is an isomorphism copy of $R$;
\item $R_1 := \{r\in R : r = \alpha(r)\}\subset R$ is a ring, the {\em invariant ring}\index{invariant ring}\index{ring!invariant} of $R$;
\item $R_0 := \{r\in R : \delta(r) = 0\}\subset R$ is a ring, the {\em constant ring}\index{constant!ring}\index{ring!constant} of $R$;
\item $\{r\in R : Y\cdot r = rY\} = R_0\cap R_1$.
\item  If $R$ is a skew field, such are also $Im(\alpha)$,  $R_1$ and $R_0$.
\item Denoting $Z := \{z\in R : zr = rz \Forall r\in R\}$ the {\em center} of $R$\index{center}, we have
$$\{r\in R : f\cdot r = rf \Forall f\in {\sf R}\} = R_0\cap R_1\cap Z.$$
\end{enumerate}

Moreover, if we consider two polynomials
$f(Y),g(Y)\in{\sf R}\setminus\{0\},$
$$f= a Y^m+f_0,  g=b Y^n+g_0, a,b\in R\setminus\{0\}, m,n\in{\Bbb N}, 
f_0,g_0\in{\sf R}, \deg(f_0)<m,  \deg(g_0)<n,$$
we have 
$$f\cdot g = a\alpha^{m}(b)Y^{m+n}+h(Y), \deg(h)<m+n;$$
since $\alpha$ is an endomorphism, $b\neq 0\then \alpha(b)\neq 0\then \alpha^m(b)\neq 0$ and
since $R$ is a domain it holds
$\alpha^m(b)\neq 0\neq a\then a\alpha^{m}(b)\neq 0\then f\cdot g\neq 0$.
As a consequence
\begin{enumerate} 
\setcounter{enumi}{11}
\item ${\sf R}$ is a domain.
\end{enumerate}

\begin{Definition} ${\sf R}$ with the ring structure described above is called an {\em Ore extension}\index{Ore, E.!extension} and is denoted $R[Y;\alpha,\delta]$.
\end{Definition}

\begin{Remark}[Ore] In an Ore extension $R[Y;\alpha,\delta]$, denoting
${\Cal S} = \langle \alpha,\delta \rangle$ the free semigroup over the alphabet $\left\{\alpha,\delta\right\}$ and, for each $d\in{\Bbb N}$ and $i\in{\Bbb N}, 0\leq i\leq d$, ${\Cal S}_{d,i}$ the set of the $\binom{d}{i}$ words in ${\Cal S}$ of length $d$ in which occur $i$ instances of $\alpha$ and $d-i$ instances of $\delta$ in an arbitrary order,  we have  
$$Y^d\cdot r = \sum_{i= 0}^d \sum_{\tau\in{\Cal S}_{d,i}} \tau(r) Y^i$$
for each $d\in{\Bbb N}$;
 for instance
\begin{eqnarray*}
Y^3\cdot r &=& 
\alpha^3(r) Y^3 + \delta^3(r)
\\&+&
\left(\alpha^2\delta(r)+\alpha\delta\alpha(r)+ \delta\alpha^2(r)\right) Y^2
\\&+&
\left(\alpha\delta^2(r)+\delta\alpha\delta(r)+ \delta^2\alpha(r)\right) Y.
\end{eqnarray*}

In particular, for $f(Y)=\sum_{i=0}^n a_i Y^{n-i}$ and $g(Y)=\sum_{i=0}^m b_i Y^{m-i}$ in ${\sf R}$ we have
$$g(Y)f(Y)=\sum_{i=0}^{n+m} c_i Y^{n+m-i} \With c_0=b_0\alpha^m(a_0) \And
c_i=\sum_{{\sf a}=0}^i b_{\sf a}\sum_{{\sf b}=0}^{i-{\sf a}}  \sum_{\tau\in{\Cal S}_{m-{\sf a},i-{\sf a}-{\sf b}}}
\tau(a_{\sf b}).$$
\end{Remark}

\begin{Remark}[Ore] \label{OreRem1} Under the assumption that ({\em cf.} 3.)  $\alpha$ is an automorphism, each polynomial 
 $\sum_{i=1}^n a_i Y^i\in{\sf R}$ can be uniquely represented as
 $\sum_{i=1}^n  Y^i \bar{a}_i$ for proper values $\bar{a}_i\in R$.
 
 In fact we have
 $ax=x\alpha^{-1}(a)-\delta(\alpha^{-1}(a))$
 from which we can deduce inductively   proper expressions
 $$ax^n=x^n\alpha^{-n}(a)+\sum_{i=1}^n (-1)^ix^{n-i} \sigma_{in}(a).$$
\qed \end{Remark}

\subsection{Multivariate Ore Extensions}\label{c46S12}

 \begin{Definition}
 An {\em   iterative Ore extension}\index{iterative Ore extension}\index{Ore, E.!extension!iterative} is a ring (whose multiplication we denote $\star$) defined as
$${\sf R} := R[Y_1;\alpha_1,\delta_1][Y_2;\alpha_2,\delta_2]\cdots[Y_n;\alpha_n,\delta_n]$$
where, for each $i>1$, $\alpha_i$ is an endomorphism and $\delta_i$ 
 an $\alpha_i$-derivation of the iterative Ore extension
$$R_{i-1} := R[Y_1;\alpha_1,\delta_1]\cdots[Y_{i-1};\alpha_{i-1},\delta_{i-1}].$$

A   {\em multivariate Ore extension}\index{Ore, E.!extension!multivariate}
(or: {\em Ore algebra}\index{Ore, E.!algebra}  \cite{CS}; or: {\em iterative Ore extension with commuting variables}  \cite{P1,P2}) is an iterative Ore extension which satisfies
\begin{itemize}
\renewcommand\labelitemi{\bf --}
\item  $\alpha_j\delta_i = \delta_i\alpha_j$,  for each $i,j$, $i\neq j$,
\item $\alpha_i\alpha_j = \alpha_j\alpha_i$,
$\delta_i\delta_j = \delta_j\delta_i$  for $j > i$,
\item$\alpha_j(Y_i) = Y_i, \delta_j(Y_i) = 0$ for $j > i$. 
\end{itemize}
\qed\end{Definition}

\begin{Lemma}[Pesch]\label{OreLePex} In an iterative Ore extension, for each  $i<j$
it holds
$$Y_j\star Y_i = Y_i Y_j \iff \alpha_j(Y_i) = Y_i, \delta_j(Y_i) = 0.$$
\end{Lemma}

\begin{proof} For each  $i<j$, we have
$Y_j\star  Y_i = \alpha_j(Y_i) Y_j+\delta_j(Y_i).$ 
\end{proof}

\begin{Lemma}[Pesch]\label{OreLePe} An iterative Ore extension is a multivariate Ore extension iff  $Y_j\star Y_i = Y_i Y_j$
for each $i<j$.\end{Lemma}

\begin{proof} In fact, using Lemma~\ref{OreLePex}
for each $r\in R$, we have
\begin{eqnarray*}
Y_j\star  Y_i \star  r &=& Y_j\star (\alpha_i(r)Y_i+\delta_i(r)) 
\\&=&
\alpha_j\left(\alpha_i(r)Y_i+\delta_i(r)\right)Y_j+\delta_j\left(\alpha_i(r)Y_i+\delta_i(r)\right)
\\&=&
\alpha_j\alpha_i(r)Y_iY_j+
\alpha_j\delta_i(r)Y_j+\delta_j\left(\alpha_i(r)Y_i\right)+
\delta_j\delta_i(r)
\\&=&
\alpha_j\alpha_i(r)Y_iY_j 
+\alpha_j\delta_i(r)Y_j+\delta_j\alpha_i(r)Y_i+\delta_j\delta_i(r)
\end{eqnarray*}
and (by symmetry)
\begin{eqnarray*}
Y_i Y_j\star   r &=& Y_i\star (\alpha_j(r)Y_j+\delta_i(r)) 
\\&=&
\alpha_i\alpha_j(r)Y_iY_j+\delta_i\alpha_j(r)Y_j+\alpha_i\delta_j(r)Y_i+\delta_i\delta_j(r).
\end{eqnarray*}\end{proof}

Thus the   $R$-module structure of a multivariate Ore extension
can  be identified with that of the polynomial ring $R[Y_1,\ldots,Y_n]$ 
over its natural  $R$-basis 
$${\Cal T} := \{Y_1^{a_1}\cdots Y_n^{a_n} : (a_1,\ldots,a_n)\in{\Bbb N}^n\},\quad
{\sf R}\cong R[{\Cal T}] = \Span_R\{\Cal T\}.$$

We can therefore denote $\alpha_{Y_i} := \alpha_i, \delta_{Y_i} := \delta_i $ for each $i$ and, iteratively,
$$\alpha_{\tau Y_i} := \alpha_\tau\alpha_i, \delta_{\tau Y_i} := \delta_\tau\delta_i, \Forall\tau\in{\Cal T}.$$ 

Remark that a multivariate Ore extension is {\em not} an algebra; in fact,
if we define, for $\tau = Y_1^{d_1}\cdots Y_n^{d_n}$
and $t = Y_1^{e_1}\cdots Y_n^{e_n}$ such that $\tau\mid t$
$$\binom{t}{\tau} :=\binom{e_1}{d_1}\cdots\binom{e_n}{d_n},$$
we have
$$t\star r =\sum\limits_{\substack{\tau\in{\Cal T}\\ \tau\mid t}} \binom{t}{\tau} \delta_{\frac{t}{\tau}}\alpha_\tau(r) \tau, \Forall t\in{\Cal T} \And r\in R.$$

We can define, for each $t\in{\Cal T}$, a map 
$$\theta_t : R\to {\sf R},\quad \theta_t(r) = 
\sum\limits_{\substack{\tau\in{\Cal T}\\ \tau\mid t,  \tau\neq t}} 
\binom{t}{\tau} \delta_{\frac{t}{\tau}}\alpha_\tau(r) \tau,$$
so that
$t\star r = \alpha_t(r) t+\theta_t(r)$  for each $t\in{\Cal T}$ and each $r\in R$. 

Such maps $ \alpha_t$ and $\theta_t$ satisfy properties analogous of those of Ore's 
conjugate  and derivative:
\begin{Lemma} \label{OreLe+} With the present notation, for each  $t\in{\Cal T}$, we have
\begin{enumerate} 
\item for each $r\in R$, $\alpha_t(r) = 0 \then r = 0$, 
\item $\alpha_t$ is a ring endomorphism; 
\item the following conditions are equivalent:
\begin{enumerate}
\item for each $d\in R\setminus\{0\}$ exists $c\in R\setminus\{0\} :
Y\star c=dY+\theta_t(c), \alpha_t(c)=d$;
\item $\alpha_t$ is a ring automorphism;
\end{enumerate}
\item $\theta_t$ is an $\alpha_t$-derivation of ${\sf R}$;
\item if $R$ is a skew field, then  each $r\in R\setminus\{0\}$ satisfies
$$\alpha_t(r^{-1}) = \left(\alpha_t(r)\right)^{-1}, \quad \theta_t(r^{-1}) = -\left(\alpha_t(r)\right)^{-1}\theta_t(r)r^{-1};$$
\item $Im(\alpha_t)\subset R$ is a subring, which is an isomorphism copy of $R$.
\end{enumerate}
We further have
\begin{enumerate} 
\setcounter{enumi}{6}
\item if  each $\alpha_i$ is an automorphism, also each  $\alpha_t$, $t\in{\Cal T}$, is such.
\qed\end{enumerate}\end{Lemma}

\subsection{Associated  graded Ore Extension}\label{c46S12b}

\begin{Definition}\label{itOrex}
A multivariate Ore extension
$$R[Y_1;\alpha_1,\delta_1][Y_2;\alpha_2,\delta_2]\cdots[Y_n;\alpha_n,\delta_n]$$
where each $\delta_i$ is zero, will be called a
{\em graded Ore extension}\index{Ore, E.!extension!graded} (or:
{\em Ore extension with zero derivations} \cite{P1,P2})
and will be denoted
$${\sf R} := R[Y_1;\alpha_1][Y_2;\alpha_2]\cdots[Y_n;\alpha_n].$$
\qed\end{Definition}

\begin{Lemma}  In a multivariate graded Ore extension,
\begin{itemize}
\item since it is an Ore algebra, the $\alpha$s commute,
\item and $t\star r = \alpha_t(r) t \Bigl[=:{\bf M}(t\star r)\Bigr] \Forall t\in{\Cal T} \And r\in R.$
\end{itemize}\end{Lemma}

\begin{Remark}\label{RemVal}
Note that, since multivariate Ore extensions coincide, as left $R$-modules, with  the classical polynomial rings $R[Y_1,\ldots,Y_n]$ and so
have the same $R$-basis, namely ${\Cal T}$, they can share with  the polynomial rings their 
standard  ${\Cal T}$-valuation  \cite{S,m1,A,MoSw} \cite[\S24.4,24.6]{SPES}. 
This justifies the definition below.
\end{Remark}

\begin{Definition}
Given an Ore extension
${\sf R} := R[Y_1;\alpha_1,\delta_1][Y_2;\alpha_2,\delta_2]\cdots[Y_n;\alpha_n,\delta_n]$
the corresponding graded Ore extension
$G({\sf R}) := R[Y_1;\alpha_1][Y_2;\alpha_2]\cdots[Y_n;\alpha_n]$
is called its {\em associated  graded Ore extension}\index{associated graded!Ore extension}.
\qed\end{Definition}

\begin{Example} \label{TamEx}\

\begin{enumerate}
\item The first non obvious example of Ore extension was proposed in 1948 by D.Ta\-ma\-ri \cite{T} in connection with the notion of  ``order of irregularity'' introduced by Ore in  \cite{O1}; it consists of the 
graded Ore extension.
$${\sf R}:=R[Y;\alpha],\, R={\Bbb Q}[x] \Where\alpha : R \to R  : x \mapsto x^2.$$
\item Such construction was generalized by Weispfenning  \cite{W} which introduced the rings
$${\sf R} :=R[Y;\alpha], \, R={\Bbb Q}[x] \Where\alpha : R \to R  : x \mapsto x^e, \, e_i\in{\Bbb N}$$
\item and extended by Pesch  \cite{P1} to his {\em iterated Ore extensions with power substitution}
$${\sf R}:=R[Y_1;\alpha_1][Y_2;\alpha_2]\cdots[Y_n;\alpha_n], \, R={\Bbb Q}[x]$$
where $\alpha_i : R \to R  : x \mapsto x^{e_i}, \, e_i\in{\Bbb N}$.
\item An Ore extension where $\alpha$ is invertible is discussed in  \cite{MS}:
$${\sf R}:=R[S;\alpha], R= {\Bbb Q}[D_1,D_2,D_3]$$
where
$$\alpha : R \to R  : f(D_1,D_2,D_3) \mapsto 
f(D_2+2D_1,D_3,-D_1)$$
whose inverse is
$$\alpha^{-1} : R \to R  : f(D_1,D_2,D_3) \mapsto 
f(-D_3,D_1+2D_3,D_2).$$
\qed\end{enumerate}\end{Example}

Note that, while as $R$-modules ${\sf R}$ and $G({\sf R})$ coincide both with the polynomial ring
${\Cal P} = R[Y_1,\ldots,Y_n]$, the three rings have, in general, different arithmetics; we will denote $\star$ the multiplication of ${\sf R}$ and $\ast$ those of $G({\sf R})$.

\begin{Example}\label{ExWeis2} The ring of Example~\ref{TamEx}, 1.
$${\sf R}:=R[Y;\alpha], \, R={\Bbb Q}[x] \Where\alpha : R \to R  : x \mapsto x^2$$
 is an Ore extension
 which is graded.

The map 
$$\delta : {\Bbb Q}[x] \to {\Bbb Q}[x]  : x^i \mapsto \sum_{h=i}^{2i-1} x^h$$
is an $\alpha$-derivation since
\begin{eqnarray*}
\delta(x^{i}\cdot x^j)&=& \sum_{h=i+j}^{2i+2j-1}  x^h 
\\ &=&
 \sum_{h=2i+j}^{2i+2j-1}  x^h +
\sum_{h=i+j}^{2i+j-1}  x^h 
\\ &=& 
x^{2i} \sum_{h=j}^{2j-1}  x^h + 
x^j \sum_{h=i}^{2i-1}  x^h
\\ &=&
 \alpha(x^{i})\delta(x^j)+\delta(x^{i})x^j;
 \end{eqnarray*}
thus ${\sf S}:=R[Y;\alpha,\delta]$ is an Ore extension of which ${\sf R}$ is the 
associated  graded Ore extension.\qed\end{Example}

\begin{Example}\label{GOe}
\alpheqn

Since in Buchberger-Zacharias Theory,  from an algorithmic point of view, the interest points   are associated graded rings and thus the r\^ole of derivates  is  irrelevant, we illustrate the results for the Ore extensions with the  zero-derivatives
$${\sf R} := R[Y_1;\alpha_1][Y_2;\alpha_2]\cdots[Y_n;\alpha_n], R={\Bbb Z}[x],$$
with $\alpha_i(x) := c_i x^{e_i}, c_i\in{\Bbb Z}\setminus\{0\}, e_i\in{\Bbb N}\setminus\{0\}$.

If we denote $\gamma$ the map
$$\gamma : {\Bbb N} \times {\Bbb N}\setminus\{0\} \to  {\Bbb N},
(a,e)\mapsto \sum_{i=0}^{a-1} e^i = \frac{1-e^a}{1-e}$$
 where the last equality holds for $e \neq 1,$ 
we have $Y_i^a\ast x^b=c_i^{b\gamma(a,e_i)}x^{e_i^ab}Y_i^a$.

Note that 
\begin{equation}\label{Eq1}
\gamma(b,e)+e^b\gamma(a,e)=
 \sum_{i=0}^{b-1} e^i+ \sum_{i=0}^{a-1} e^{b+i} =  \sum_{i=0}^{a+b-1} e^i=\gamma(a+b,e).
\end{equation}

Since 
$\alpha_j(\alpha_i(x))=c_i\alpha_j(x^{e_i})=c_ic_j^{e_i}x^{e_ie_j}$
and 
$\alpha_i(\alpha_j(x))=c_j\alpha_i(x^{e_j})=c_jc_i^{e_j}x^{e_ie_j}$,
then ${\sf R}$ is a graded Ore extension if and only if
$$c_ic_j^{e_i}x^{e_ie_j} = \alpha_j(\alpha_i(x)) = \alpha_i(\alpha_j(x))=c_j\alpha_i(x^{e_j})=c_jc_i^{e_j}x^{e_ie_j}$$
{\em id est}
\begin{equation}\label{eqb}
c_j^{e_i-1}=c_i^{e_j-1}. 
\end{equation}
We thus have $\binom{n}{2}$ relations among the $n$ coefficients $c_i$. In particular we need to partition the indices as
 $$\{1,\ldots,n\}= E\sqcup O\sqcup S,  E=\{i: 2\mid e_i\}, O=\{i: 2\nmid e_i>1\},S=\{i: e_i=1\}.$$ 

If $I:=O\sqcup E=\emptyset$ then each such equations are the trivial equality $1=1$ and thus all $c_i$ are free. The situation is completely different when $I:=O\sqcup E\neq\emptyset$;
in fact, 
\begin{itemize}
\item for $i\in S$ necessarily $c_i=\pm 1$;
\item if a prime $p$ divides at least a $c_j, j\in I$, then  it divides each $c_i, i\in I$.
\end{itemize}

As regards the sign of $c_i$ we can say that
\begin{itemize}
\item if $E\neq\emptyset$ then 
\begin{itemize}
\item $c_i$ is positive for each $i\in S \cup O$, 
\item the sign of $c_i, i \in E$, is undetermined but all the $c_i, i \in E$, have the same sign.
\end{itemize}

\item if $E=\emptyset$ then the sign of $c_i, i \in  S \cup O$ is undetermined.
\end{itemize}

For instance
\begin{itemize}
\item for $e_1=e_4=1, e_2=5, e_3=3$ we have $S=\{1,4\},O=\{2,3\},E=\emptyset$ and
$$c_1^4=c_2^0,c_1^2=c_3^0,c_1^0=c_4^0,c_2^2=c_3^4,c_2^0=c_4^4,c_3^0=c_4^2;$$
whence $c_1=\pm1, c_4=\pm1, c_2=\pm c_3^2$;
\item for $e_1=e_4=1, e_2=2, e_3=3$ we have $S=\{1,4\},O=\{3\},E=\{2\},$ and
$$c_1=c_2^0,c_1^2=c_3^0,c_1^0=c_4^0,c_2^2=c_3,c_2^0=c_4,c_3^0=c_4^2;$$
whence $c_1=c_4=1, c_3=c_2^2>0$;
\item for $e_1=1,\, e_2=2, e_3=3$, $S=\{1\}$ $E=\{2\}$, $O=\{3\}$. Suppose $c_2=6$, so both the primes $2$ and $3$ divide $c_2$. From $c_1=c_2^0$, $c_1^2=c_3^0$, $c_2^2=c_3$ we get $c_1=1$ and $c_3=36$. We  notice that $2 \mid c_3$  and $3 \mid c_3$, but neither $2$ nor $3$ divide $c_1$;
\item for $e_1=e_4=1, e_2=4, e_3=8$ we have $S=\{1,4\},E=\{2,3\},O=\emptyset$ and
$$c_1^3=c_2^0,c_1^7=c_3^0,c_1^0=c_4^0,c_2^7=c_3^3,c_2^0=c_4^3,c_3^0=c_4^7.$$
whence $c_1=c_4=1, c_2=\chi^3,c_3=\chi^7,c_2c_3>0$.
\end{itemize}

As regards the values $\vert c_i \vert, 1\leq i \leq n$, 
setting 
$$\rho:=\sum_{j=1}^n (e_j-1)=\sum_{j\in I} (e_j-1), \chi:= \sqrt[\rho]{\prod_{j=1}^n \vert c_j \vert},$$
 we have 
\begin{equation}\label{eqc}
\vert c_j\vert = \chi^{e_j-1}\Forall j\in \{1,....,n\}. 
\end{equation}

In fact, since if a prime $p$ divides at least a $c_j, j\in I$, then  it divides each $c_i, i\in I$,
we can express each $\vert c_i \vert, i\in I$, as
$\vert c_i \vert=p_{1}^{a_{i,1}}\cdots p_{h}^{a_{i,h}}$
where $p_{1},\cdots, p_{h}$ are the prime factors
of the
squarefree associate
$\sqrt{\chi} = p_{1}\cdots p_{h}$ of $\chi$.

We have
$$\vert c_i \vert^{e_j-1}= \vert c_j \vert^{e_i-1} \then p^{a_i(e_j-1)}=p^{a_j(e_i-1)} \then a_i(e_j-1)=a_j(e_i-1)$$
whence $a_i=a_j\iff e_i=e_j$ and $a_i>a_j\iff e_j<e_i$.

Thus the $c_i$s with minimal $e_i$ minimalize also all ${a_{i,l}}, 1\leq l\leq h$.
 
 We moreover have $a_{j,l}=\frac{a_{i,i}(e_j-1)}{(e_i-1)}, 1\leq l\leq h$. 

Therefore 
$\prod_{j=1}^n \vert c_j\vert = \prod_{j\in I} \vert c_j\vert =\prod_{l=1}^h p_l^{a_{j,l}} =  \prod_{l=1}^h
p_{l}^{\frac{a_{i,l}\sum_{j=1}^n(e_j-1)}{e_i-1}}  
=  \prod_{l=1}^h
p_{l}^{\frac{a_{i,l}\rho}{e_i-1}}
$ 
 whence
 $$\chi:= \sqrt[\rho]{\prod_{j=1}^n \vert c_j \vert} = \prod_{l=1}^h
p_{l}^{\frac{a_{i,l}}{e_i-1}} =
\prod_{l=1}^h
p_{l}^{\frac{a_{j,l}}{e_j-1}}$$
and  (\ref{eqc}).

The formula (\ref{eqc}) allows to reformulate (\ref{eqb}) as
\begin{equation}\label{eqd}
\vert c_j\vert^{e_i-1}  =\vert c_i\vert^{e_j-1}  =  \chi^{(e_i-1)(e_j-1)}.
\end{equation}

Note that we have
\begin{eqnarray*} 
\chi^{(e_i^{a_i}-1)(e_j^{a_j}-1)} =
 \chi^{(e_i-1)\gamma(a_i,e_i)(e_j-1)\gamma(a_j,e_j)} 
&=&
\vert c_i\vert^{(e_j-1)\gamma(a_i,e_i)\gamma(a_j,e_j)}=\vert c_i\vert^{(e_j^{a_j}-1)\gamma(a_i,e_i)} 
\\ &=&
\vert c_j\vert^{(e_i-1)\gamma(a_i,e_i)\gamma(a_j,e_j)} =\vert c_j\vert^{(e_i^{a_i}-1)\gamma(a_j,e_j)}
\end{eqnarray*}
and
\begin{eqnarray}\label{eqe}
\vert c_i\vert^{\gamma(a_i,e_i)}\vert c_j\vert^{e_i^{a_i}\gamma(a_j,e_j)} =
\vert c_i\vert^{e_j^{a_j}\gamma(a_i,e_i)}\vert c_j\vert^{\gamma(a_j,e_j)}  =   \chi^{e_i^{a_i} e_j^{a_j}  -1}.
\end{eqnarray}

\reseteqn\alpheqn

To avoid cumbersome and useless case-to-case studies,
let us simply assume $c_i>0$ for each $i$;
under this restricted assumption, a series of inductive arguments allow to deduce 
\begin{equation}\label{eq2a}
 Y_i \ast x^\alpha = c_i^\alpha x^{\alpha e_i} Y_i
\end{equation}

$$Y_i \ast x^\alpha = (Y_i \ast x^{\alpha -1})x= c_i^{\alpha -1} x^{(\alpha-1) e_i} (Y_i \ast x)=
c_i^{\alpha -1} x^{(\alpha-1) e_i} c_i x^{e_i}Y_i = c_i^\alpha x^{\alpha e_i} Y_i.$$
\begin{equation}\label{eq2b}
 Y_j^{a_j} \ast x^{b_0} = c_j^{b_0 \gamma(a_j, e_j)} x^{b_0 e_j^{a_j}} Y_j^{a_j}
\end{equation}

\begin{eqnarray*}
Y_j^{a_j} \ast x^{b_0} = Y_j\left(Y_j^{a_j-1} \ast x^{b_0}\right)
&=& Y_j\ast \left(c_j^{b_0 \gamma(a_j-1, e_j)} x^{b_0 e_j^{a_j-1}} Y_j^{a_j-1}\right)
 \\ &=& c_j^{b_0 \gamma(a_j-1, e_j)}\left(Y_j\ast x^{b_0 e_j^{{a_j}-1}}\right) Y_j^{a_j-1} 
\\ &=& c_j^{b_0 \gamma(a_j-1, e_j)}\left(c_j^{b_0 e_j^{a_j-1}}  x^{\left(b_0 e_j^{a_j-1}\right) e_j}  \right)Y_j^{a_j}
 \\ &=&
  c_j^{b_0 \gamma(a_j, e_j)} x^{b_0 e_j^{a_j}} Y_j^{a_j}.
  \end{eqnarray*}
\begin{equation}\label{eq2c}
c_j^{b_0 \gamma(a_j,e_j)} c_i^{b_0 e_j^{a_j}\gamma(a_i,e_i)} = \chi^{b_0(e_j^{a_j}e_i^{a_i}-1)}
 \end{equation}  
 Substituing $c_j = \chi^{e_j-1}$ and $c_i= \chi^{e_i-1}$ we get
\begin{eqnarray*}
c_j^{b_0 \gamma(a_j,e_j)} c_i^{b_0 e_j^{a_j}\gamma(a_i,e_i)} 
= \chi^{\left(e_j-1\right)b_0 \gamma(a_j,e_j)} \chi^{\left(e_i-1\right)b_0 e_j^{a_j}\gamma(a_i,e_i)} &=& \chi^{\left(e_j-1\right)b_0  \frac{1-e_j^{a_j}}{1-e_j}+\left(e_i-1\right)b_0 e_j^{a_j}\frac{1-e_i^{a_i}}{1-e_i}   }
\\ &=& \chi^{-b_0\left(1-e_j^{a_j}\right)-b_0 e_j^{a_j}\left(1-e_i^{a_i}\right)}
\\ &=& \chi^{b_0 e_j^{a_j}-b_0 + b_0 e_j^{a_j} e_i^{a_i}-b_0 e_j^{a_j}}
\\ &=&  \chi^{b_0(e_j^{a_j}e_i^{a_i}-1)}.
 \end{eqnarray*}
\begin{equation}\label{eq2d}
 Y_i^{a_i} Y_j^{a_j} \ast x^{b_0} = c_j^{b_0 \gamma(a_j, e_j)}c_i^{b_0 e_j^{a_j} \gamma(a_i, e_i)} x^{b_0 e_j^{a_j}e_i^{a_i}}Y_i^{a_i}Y_j^{a_j}= \chi^{b_0(e_j^{a_j}e_i^{a_i}-1)}x^{b_0 e_j^{a_j}e_i^{a_i}}Y_i^{a_i}Y_j^{a_j}
\end{equation}
\begin{eqnarray*}
 Y_i^{a_i} \left(Y_j^{a_j} \ast x^{b_0}\right) = 
 Y_i^{a_i}\ast \left(c_j^{b_0 \gamma(a_j,e_j)}x^{b_0 e_j^{a_j}}Y_j^{a_j} \right)
&=& c_j^{b_0 \gamma(a_j,e_j)} \left(Y_i^{a_i} \ast x^{b_0 e_j^{a_j}}\right) Y_j^{a_j}
\\ &=& c_j^{b_0 \gamma(a_j,e_j)} \left( c_i^{b_0 e_j^{a_j}\gamma(a_i,e_i)} x^{b_0 e_j^{a_j}e_i^{a_i}} Y_i^{a_i}\right) Y_j^{a_j}
\\ &=& c_j^{b_0 \gamma(a_j,e_j)} c_i^{b_0 e_j^{a_j}\gamma(a_i,e_i)} x^{b_0 e_j^{a_j}e_i^{a_i}} Y_i^{a_i} Y_j^{a_j}.
 \end{eqnarray*}
\begin{equation}\label{eq2e}
 (a x^{a_0} Y_1^{a_1} \cdots Y_n^{a_n})\ast (b x^{b_0} Y_1^{b_1} \cdots Y_n^{b_n})  = 
 ab \chi^{b_0\left(\left(\prod_{i=1}^n e_i^{a_i}\right)-1\right)} x^{a_0+ b_0 \prod_{i=1}^n e_i^{a_i} }     Y_1^{a_1+b_1}\cdots Y_n^{a_n+b_n}
\end{equation}
\begin{eqnarray*}
&& (a x^{a_0} Y_1^{a_1} \cdots Y_n^{a_n})\ast (b x^{b_0} Y_1^{b_1} \cdots Y_n^{b_n})  
\\ &=&  a x^{a_0} Y_1^{a_1}(Y_2^{a_2} \cdots Y_n^{a_n}\ast b x^{b_0}) Y_1^{b_1} \cdots Y_n^{b_n} 
\\ &=& a x^{a_0} Y_1^{a_1}\left(b \chi^{b_0\left(\left(\prod_{i=2}^n e_i^{a_i}\right)-1\right)}x^{b_0 \prod_{i=2}^n e_i^{a_i}}\right) Y_1^{b_1}Y_2^{a_2+b_2}\cdots Y_n^{a_n+b_n} 
\\ &=& ab \chi^{b_0\left(\left(\prod_{i=2}^n e_i^{a_i}\right)-1\right)} x^{a_0}\left(Y_1^{a_1}\ast x^{b_0 \prod_{i=2}^n e_i^{a_i}}\right) Y_1^{b_1}Y_2^{a_2+b_2}\cdots Y_n^{a_n+b_n}
\\ &=& ab \chi^{b_0\left(\left(\prod_{i=2}^n e_i^{a_i}\right)-1\right)} x^{a_0}\left(c_1^{b_0 \prod_{i=2}^n e_i^{a_i} \gamma(a_1,e_1)} x^{b_0 \left(\prod_{i=2}^n e_i^{a_i} \right)e_1^{a_1}} \right) Y_1^{a_1+b_1}Y_2^{a_2+b_2}\cdots Y_n^{a_n+b_n} 
\\ &=& 
ab \chi^{b_0\left(\left(\prod_{i=2}^n e_i^{a_i}\right)-1\right)} x^{a_0}\left(\chi^{(e_1-1)b_0 \prod_{i=2}^n e_i^{a_i} \frac{1-e_1^{a_1}}{1-e_1}} x^{b_0 \prod_{i=1}^n e_i^{a_i} }\right ) Y_1^{a_1+b_1}\cdots Y_n^{a_n+b_n}
\\ &=& 
ab \chi^{b_0\left(\left(\prod_{i=1}^n e_i^{a_i}\right)-1\right)} x^{a_0+ b_0 \prod_{i=1}^n e_i^{a_i} }     Y_1^{a_1+b_1}\cdots Y_n^{a_n+b_n}.
 \end{eqnarray*}
Note that associativity is verified by
\begin{eqnarray*}
&& \left[\left(a x^{a_0} Y_1^{a_1} \cdots Y_n^{a_n}\right)\ast \left(b x^{b_0} Y_1^{b_1} \cdots Y_n^{b_n}\right) \right]\ast 
 \left(d x^{d_0} Y_1^{d_1} \cdots Y_n^{d_n}\right)
\\ &=& 
\left[ab \chi^{b_0 \left(\left(\prod_{i=1}^n e_i^{a_i}\right) -1 \right)} x^{a_0+b_0 \prod_{i=1}^n e_i^{a_i}  } Y_1^{a_1+b_1}\cdots Y_n^{a_n+b_n}\right] \ast 
 \left(d x^{d_0} Y_1^{d_1} \cdots Y_n^{d_n}\right) 
\\ &=& 
abd \chi^{b_0\left(\left(\prod_{i=1}^n e_i^{a_i}\right) -1 \right) +d_0\left(\left(\prod_{i=1}^n e_i^{a_i+b_i}\right) -1 \right) }  x^{a_0+b_0 \prod_{i=1}^n e_i^{a_i} +d_0 \prod_{i=1}^n e_i^{a_i+b_i}} Y_1^{a_1+b_1+d_1}\cdots Y_n^{a_n+b_n+d_n}
\end{eqnarray*}
and
\begin{eqnarray*}
&& \left(a x^{a_0} Y_1^{a_1} \cdots Y_n^{a_n}\right)\ast \left[\left(b x^{b_0} Y_1^{b_1} \cdots Y_n^{b_n}\right) \ast 
 \left(d x^{d_0} Y_1^{d_1} \cdots Y_n^{d_n}\right) \right]
\\ &=& 
\left(a x^{a_0} Y_1^{a_1} \cdots Y_n^{a_n}\right)\ast \left[bd \chi^{d_0
\left(\left(\prod_{i=1}^n e_i^{b_i}\right) -1 \right)}x^{b_0+d_0 \prod_{i=1}^n e_i^{b_i}} Y_1^{b_1+d_1}\cdots Y_n^{b_n+d_n}\right ]
\\ &=& 
abd\chi^{d_0
\left(\left(\prod_{i=1}^n e_i^{b_i}\right) -1 \right)
+\left(b_0+d_0 \left(\prod_{i=1}^n e_i^{b_i}\right)\right)\left(\left(\prod_{i=1}^ne_i^{a_i}\right) -1\right)   } 
x^{a_0+\left(b_0+d_0 \prod_{i=1}^n e_i^{b_i}\right)\prod_{i=1}^n e_i^{a_i}} 
\prod_{i=1}^n Y_i^{a_i+b_i+d_i} 
\\ &=& 
abd \chi^{b_0\left(\left(\prod_{i=1}^n e_i^{a_i}\right) -1 \right) +d_0\left(\left(\prod_{i=1}^n e_i^{a_i+b_i}\right) -1 \right) }  x^{a_0+b_0 \prod_{i=1}^n e_i^{a_i} +d_0 \prod_{i=1}^n e_i^{a_i+b_i}} Y_1^{a_1+b_1+d_1}\cdots Y_n^{a_n+b_n+d_n}
\end{eqnarray*}
\reseteqn\qed\end{Example}
 
\section{Buchberger Theory}\label{c46S12a}

\subsection{Term ordering}

For each $m\in{\Bbb N}$,  the free  ${\sf R}$-module ${\sf R}^m$ -- 
 the canonical basis of which will be denoted  $\{{\bf e}_1,\ldots,{\bf e}_m\}$ -- 
is an $(R,R)$-bimodule  with basis the set of the  {\em terms}
$${\Cal T}^{(m)} := \{t{\bf e}_i : t\in{\Cal T}, 1\leq i\leq m\}.$$

If we impose on ${\Cal T}^{(m)}$ a total ordering $<$, then 
each $f\in {\sf R}^m$  has a unique representation as 
an ordered linear combination of terms $t\in{\Cal T}^{(m)}$ with coefficients in $R$:
$$f = \sum_{i=1}^s c(f,t_i) t_i : c(f,t_i) \in R\setminus\{0\}, t_i \in {\Cal T}^{(m)}, 
t_1 > \cdots > t_s.$$
The {\em support}\index{support} of $f$ is the set $\supp(f)
:= \{t \, \vert \, c(f,t)\neq 0\}.$ 

W.r.t. $<$ we denote ${\bf T}(f) := t_1$
the {\em maximal term}\index{maximal!term} of $f$, $\lc(f)
:= c(f,t_1)$ its {\em leading cofficient}\index{leading!cofficient} and ${\bf M}(f) := c(f,t_1)t_1$ 
its {\em maximal monomial}\index{maximal!monomial}.

If we denote, following  \cite{R2,R3}, 
${\sf M}({\sf R}^m) := \{ct{\bf e}_i \, \vert \, c\in R\setminus\{0\}, t\in{\Cal T}, 1\leq i\leq m\},$
for each $f\in{\sf R}^m\setminus\{0\}$ the unique finite representation above
can be reformulated 
$$f = \sum_{\tau\in\supp(f)} m_\tau, \, m_\tau = c(f,\tau)\tau$$
as a sum of elements of the {\em monomial set} ${\sf M}({\sf R}^m).$

While a multivariate Ore extension does not satisfiy commutativity between terms and coefficients,
$$t\star r = r t \Forall r\in R\setminus\{0\}, t\in{\Cal T},$$
it however satisfies
\begin{equation}\label{c46Eq2}
{\bf M}(t\star r) = \alpha_t(r)t, \Forall r\in R\setminus\{0\}, t\in{\Cal T}^{(m)};
\end{equation}
moreover, while ${\sf R}$ is not a monoid ring under  the multiplication $\star$, so that in particular we cannot claim $\tau\star \omega\in{\Cal T}$ for $\tau,\omega\in{\Cal T}$, however 
$\tau\star \omega$ satisfies 
\begin{equation}\label{c46Eq3}
{\bf T}(\tau\star \omega)= \tau\circ\omega
\end{equation}
where we have denoted
$\circ$ the (commutative) multiplication of ${\Cal T}$;
similarly, for $n\in{\sf M}({\sf R}^m)$ 
and $m_l,m_r\in{\sf M}({\sf R})=\{ct : c\in R\setminus\{0\}, t\in{\Cal T}\}$ we have
${\sf M}(m_l\star m\star m_r)=m_l\ast m\ast m_r$.

In conclusion
w.r.t.  each term  ordering $\prec$ on ${\Cal T}$ and each 
$\prec$-compatible
term  ordering $<$ on  ${\Cal T}^{(m)}$, {\em id est} any well-ordering on ${\Cal T}^{(m)}$ which satisfies
$$\omega_1 \prec\omega_2 \then
\omega_1t< \omega_2 t,
t\omega_1< t \omega_2
\Forall t\in{\Cal T}^{(m)},\omega_1,\omega_2\in{\Cal T},
$$
it holds, for each $ l,r\in{\sf R} \And f\in{\sf R}^{(m)},$
\begin{equation} \label{Eq3}
{\bf T}(l\star f\star r)= {\bf T}(l)\circ{\bf T}(f)\circ{\bf T}(r)
\end{equation}
and
\begin{equation} \label{Eq4}
{\bf M}(l\star   f\star r)= {\bf M}({\bf M}(l)\star {\bf M}(f)\star {\bf M}(r)) =
{\bf M}(l)\ast {\bf M}(f)\ast {\bf M}(r).
\end{equation}

As a consequence we trivially have
\begin{Corollary}\label{c46CoX2X2X2} 
If $\prec$ is a term  ordering on ${\Cal T}$ 
and $<$ is a $\prec$-compatible
term  ordering on ${\Cal T}^{(m)}$,
then, for each $l,r\in{\sf R}$ and $f\in{\sf R}^{(m)}$,
\begin{enumerate}
\item ${\bf M}(l\star  f)= {\bf M}({\bf M}(l)\star  {\bf M}(f))={\bf M}(l)\ast {\bf M}(f)$;
\item ${\bf M}(f\star  r)=  {\bf M}({\bf M}(f)\star  {\bf M}(r))={\bf M}(f)\ast {\bf M}(r)$;
\item ${\bf M}(l\star  f\star  r)=  {\bf M}({\bf M}(l)\star  {\bf M}(f)\star  {\bf M}(r))={\bf M}(l)\ast {\bf M}(f)\ast {\bf M}(r)$.
\item ${\bf T}(l\star    f)= {\bf T}(l)\circ{\bf T}(f)$;
\item ${\bf T}(f\star  r)= {\bf T}(f)\circ{\bf T}(r)$;
\item ${\bf T}(l\star  f\star  r)= {\bf T}(l)\circ{\bf T}(f)\circ{\bf T}(r)$.
\end{enumerate}
\end{Corollary}

\subsection{Gr\"obner Bases}

In function of a term  ordering $<$ on ${\Cal T}^{(m)}$ which is 
compatible with a term  ordering on ${\Cal T}$ which, with a slight abuse of notation, we still denote $<$,  
we denote, for any set $F \subset {\sf R}^m$,
${\Bbb I}_L(F), {\Bbb I}_R(F),{\Bbb I}_2(F)$ the left (resp. right, bilateral) module generated by $F$,
and
\begin{itemize}
\item ${\bf T}\{F\} := \{{\bf T}(f) : f\in F\}\subset {\Cal T}^{(m)};$
\item ${\bf M}\{F\} := \{{\bf M}(f) : f\in F\}\subset {\sf M}({\sf R}^m).$
\item ${\bf T}_L(F) := {\Bbb I}_L({\bf T}\{F\}) =
 \{{\bf T}(\lambda\star f) : \lambda\in {\Cal T}, f\in F\} =
 \{\lambda\circ{\bf T}(f) : \lambda\in {\Cal T}, f\in F\}
 \subset {\Cal T}^{(m)};$  
\item ${\bf M}_L(F) := \{{\bf M}(a\lambda\star f) : a\in R\setminus\{0\},\lambda\in {\Cal T}, f\in F\}
= \{m\ast {\bf M}(f) : m\in {\sf M}({\sf R}), f\in F\} \subset  {\sf M}({\sf R}^m)$;  
\item ${\bf T}_R(F) := {\Bbb I}_R({\bf T}\{F\}) =
\{{\bf T}(f\star\rho)  : \rho\in {\Cal T}, f\in F\}
= 
\{{\bf T}(f)\circ\rho  : \rho\in {\Cal T}, f\in F\}\subset {\Cal T}^{(m)};$ 
\item ${\bf M}_R(F)  := \{{\bf M}(f\star b\rho )  : b\in R\setminus\{0\},\rho\in {\Cal T}, f\in F\}
= \{{\bf M}(f)\ast n   : n\in {\sf M}({\sf R}), f\in F\}
\subset  {\sf M}({\sf R}^m);$
\item ${\bf T}_2(F) := {\Bbb I}_2({\bf T}\{F\}) = 
\{{\bf T}(\lambda\star f\star\rho)  : \lambda,\rho\in {\Cal T}, f\in F\}= 
\{\lambda\circ {\bf T}(f)\circ\rho  : \lambda,\rho\in {\Cal T}, f\in F\}\subset {\Cal T}^{(m)};$ 
\item ${\bf M}_2(F) := \{{\bf M}(a\lambda\star f\star b\rho)   : a,b\in
R\setminus\{0\},\lambda,\rho\in {\Cal T}, f\in F\}
= \{m\ast {\bf M}(f)\ast n : m,n\in {\sf M}({\sf R}), f\in F\} 
\subset  {\sf M}({\sf R}^m).$
\end{itemize}

Following an intutition by Weispfenning \cite{W} we further denote
\begin{itemize}
\item  ${\Bbb I}_W(F)$ the {\em restricted}
module generated by $F$,
$${\Bbb I}_W(F) : = \Span_{R}(af\star \rho   : a\in
R\setminus\{0\}, \rho\in {\Cal T}, f\in F),$$
\item ${\bf T}_W(F) := {\bf T}_R(F),$ 
\item ${\bf M}_W(F) := \{{\bf M}(a f\star \rho)  : a\in
R\setminus\{0\},\rho\in {\Cal T}, f\in F\}
= \{a {\bf M}(f)\ast \rho : a\in
R\setminus\{0\},\rho\in {\Cal T}, f\in F\}
\subset  {\sf M}({\sf R}^m).$
\end{itemize}

If $R$ is a skew field, for each set $F \subset {\sf R}^m$ we have
\begin{equation}\label{c46Eq1}
\begin{array}{rclcl}
{\bf M}_L(F) &=& {\bf M}\{{\Bbb I}_L({\bf M}\{F\})\} &=& {\Bbb I}_L({\bf M}\{F\})\cap{\sf M}({\sf R}^m),\\
{\bf M}_R(F) &=& {\bf M}\{{\Bbb I}_R({\bf M}\{F\})\} &=& {\Bbb I}_R({\bf M}\{F\})\cap{\sf M}({\sf R}^m),\\
{\bf M}_2(F) &=& {\bf M}\{{\Bbb I}_2({\bf M}\{F\})\} &=& {\Bbb I}_2({\bf M}\{F\})\cap{\sf M}({\sf R}^m),\\
{\bf M}_W(F) &=& {\bf M}\{{\Bbb I}_W({\bf M}\{F\})\} &=& {\Bbb I}_W({\bf M}\{F\})\cap{\sf M}({\sf R}^m).
\end{array}
\end{equation}

\begin{Notation} From now on, in order  to avoid cumbersome  notation and boring repetitions,
we will drop the subscripts when it will be clear of which kind of module (left, right, bilateral, restricted) we
are discussing. 
As a consequence, the four statements of (\ref{c46Eq1}) will be summarized as
$${\bf M}(F) = {\bf M}\{{\Bbb I}({\bf M}\{F\})\} = {\Bbb I}({\bf M}\{F\})\cap{\sf M}({\sf R}^m).$$

Similarly, we formulate a (left,right,bilateral,restricted) definition simply for the bilateral case leaving to the reader the task to remove the irrelevant factors.

For instance condition (ii) below is stated for the bilateral case; it would be reformulated:
\begin{itemize}
\item[left case] for each $f\in{\Bbb I}(F)$ there are $g\in F$, $a\in
R\setminus\{0\}, \lambda\in {\Cal T}$ such that
${\bf M}(f) =  a\lambda\ast {\bf M}(g)
=  {\bf M}(a\lambda\star g),$
\item[right case] for each $f\in{\Bbb I}(F)$ there are $g\in F$, $b\in
R\setminus\{0\}, \rho\in {\Cal T}$ such that
${\bf M}(f) =  {\bf M}(g)\ast   b\rho 
=  {\bf M}(g\star  b\rho),$
\item[restricted case] for each $f\in{\Bbb I}(F)$ there are $g\in F$, $a\in
R\setminus\{0\}, \rho\in {\Cal T}$ such that
${\bf M}(f) =  a{\bf M}(g)\ast   \rho 
=  {\bf M}(a g\star  \rho),$
\end{itemize}
\qed\end{Notation}

Thus, if $R$ is a skew field, the following conditions are equivalent and can be naturally chosen as definition of Gr\"obner bases: 
\begin{enumerate}
\item ${\bf M}({\Bbb I}(F)) = {\bf M}\{{\Bbb I}(F)\} =  {\bf M}\{{\Bbb I}({\bf M}\{F\})\} = {\Bbb I}({\bf M}\{F\})\cap{\sf M}({\sf R}^m)$,
\item for each $f\in{\Bbb I}(F)$ there is $g\in F$ such that ${\bf M}(g) \mid {\bf M}(f).$
\end{enumerate}

But in general between these statements there is just the implication $(2) \then (1)$.

Thus  \cite{P}, there are two alternative natural definitions for the concept of Gr\"obner bases:
 \begin{itemize}
 \item a stronger one which satisfies 
the following equivalent conditions:
\begin{enumerate}
\renewcommand\theenumi{{\rm (\roman{enumi})}}
\item for each $f\in{\Bbb I}(F)$ there is $g\in F$ such that
${\bf M}(g) \mid {\bf M}(f),$ 
\item for each $f\in{\Bbb I}(F)$ there are $g\in F$, $a,b\in
R\setminus\{0\}, \lambda,\rho\in {\Cal T}$ such that
${\bf M}(f) =  a\lambda\ast {\bf M}(g)\ast   b\rho 
=  {\bf M}(a\lambda\star g\star  b\rho),$
\item ${\bf M}({\Bbb I}(F)) = {\bf M}\{{\Bbb I}(F)\} =  {\bf  M}(F)$;
\end{enumerate}
\item and a weaker one  which satisfies 
the following equivalent conditions: 
\begin{enumerate}
\renewcommand\theenumi{{\rm (\roman{enumi})}}
\setcounter{enumi}{3}
\item for each $f\in{\Bbb I}(F)$ there are $g_i\in F$, $a_i,b_i\in
R\setminus\{0\},\lambda_i,\rho_i\in {\Cal T}$ for which, denoting
$\tau_i:=\mathbf{T}(g_i)$,
 one has
\begin{itemize}
\item ${\bf T}(f) =  \lambda_i\circ{\bf T}(g_i) \circ  \rho_i$ for each $i$, and 
$\lc(f) = 
\sum_i a_i\alpha_{\lambda_i}(\lc(g_{i}))\alpha_{\lambda_i\tau_i}(b_i)$
\end{itemize}
and, equivalently,
\begin{itemize}
\item ${\bf M}(f) = \sum_i a_i\lambda_i\ast  {\bf M}(g_i)\ast  b_i\rho_i
 = \sum_i{\bf M}(a_i\lambda_i\star  g_i\star   b_i\rho_i);$
\end{itemize}
\item ${\bf M}({\Bbb I}(F)) = {\bf M}\{{\Bbb I}(F)\} =   {\bf M}\{{\Bbb I}({\bf M}\{F\})\} = {\Bbb I}({\bf M}\{F\})\cap{\sf M}({\sf R}^m)$;
\end{enumerate}
 if moreover $R$ is a skew field 
${\bf  M}(F) = {\bf M}\{{\Bbb I}({\bf M}\{F\})\}$
so that conditions (i-v) above are all equivalent and are also equivalent to
\begin{enumerate}
\renewcommand\theenumi{{\rm (\roman{enumi})}}\setcounter{enumi}{5}
\item ${\bf T}(f) =  \lambda\circ  {\bf T}(g)\circ \rho$ for some 
 $g\in F$, $\lambda,\rho\in {\Cal T}$.
 \end{enumerate}
\end{itemize}

\begin{Example}\label{GOeY}
Let us now specialize the ring of Example~\ref{GOe}
to the case
$$n=3, e_1=2,e_2=3,e_3=4,\chi=5,c_1=5,c_2=5^2,c_3=5^3$$
and remark that

$$ax^{a_0}Y_1^{a_1}Y_2^{a_2}Y_3^{a_3} \ast
 bx^{b_0} Y_1^{b_1}Y_2^{b_2}Y_3^{b_3} =
ab5^{b_0\left(2^{a_1}3^{a_2}4^{a_3}-1\right)}
x^{a_0+b_02^{a_1}3^{a_2}4^{a_3}}Y_1^{a_1+b_1}Y_2^{a_2+b_2}Y_3^{a_3+b_3}.$$

As a consequence, for each $(b_0,b_1,b_2,b_3),(j_0,j_1,j_2,j_3)\in{\Bbb N}^4$
$$jx^{j_0}Y_1^{j_1}Y_2^{j_2}Y_3^{j_3}\in{\Bbb I}_L(bx^{b_0}Y_1^{b_1}Y_2^{b_2}Y_3^{b_3})$$
if and only if
\begin{equation}\label{SyzEq1}
a_1:=j_1-b_1\geq 0, 
a_2:=j_2-b_2\geq 0,
a_3:=j_3-b_3\geq 0,
a_0:=j_0-b_02^{a_1}3^{a_2}4^{a_3}\geq 0
\end{equation}
and $b5^{b_0\left(2^{a_1}3^{a_2}4^{a_3}-1\right)} \mid j.$ 

Note that if we set $y:=5x$ then for each  $(b_1,b_2,b_3),(j_1,j_2,j_3)\in{\Bbb N}^3$ and
$b(y),j(y)\in {\Bbb Z}[y]\subset R$
$$j(y)Y_1^{j_1}Y_2^{j_2}Y_3^{j_3}\in{\Bbb I}_L(b(y)Y_1^{b_1}Y_2^{b_2}Y_3^{b_3})$$
if and only if, not only (\ref{SyzEq1})
but also \fbox{$b(y^{2^{a_1}3^{a_2}4^{a_3}}) \mid j(y)$.}
\end{Example}

\begin{Definition}\label{49D1} Let ${\sf I} \subset {\sf R}^m$ be a (left, right, bilateral, restricted) module and
$G \subset {\sf I}$.
\begin{itemize}
\item $G$ will be called 
\begin{itemize}
\item a (left, right, bilateral, restricted) {\em weak Gr\"obner basis} ({\em Gr\"obner basis} for short) of  ${\sf I}$ if  
$${\bf M}\{{\sf I}\}= {\bf M}({\sf I}) ={\bf M}\{{\Bbb I}({\bf M}\{G\})\} ={\Bbb I}({\bf M}\{G\})\cap{\sf M}({\sf R}^m),$$
{\em id est} if $G$ satisfies conditions (iv-v) w.r.t. the module ${\sf I} = {\Bbb I}(G);$
in particular ${\bf M}\{G\}$ 
generates the (left, right, bilateral, restricted) module
${\bf M}({\sf I}) \subset{\sf R}^m;$
\item a 
(left, right, bilateral, restricted) {\em strong Gr\"obner
basis}\index{strong!Gr\"obner basis}\index{Gr\"obner!basis!strong} of ${\sf I}$ if 
for each 
$f\in{\sf I}$ there is $g\in G$ such that ${\bf M}(g) \mid {\bf M}(f),$
{\em id est} if $G$ satisfies conditions (i-iii) w.r.t. the module ${\sf I} = {\Bbb I}(G).$
\end{itemize}
\item
We say that $f\in {\sf R}^m\setminus\{0\}$ has 
\begin{itemize}
\item a left {\em Gr\"obner  representation}\index{representation!
Gr\"obner!(of a polynomial)}\index{Gr\"obner!representation (of a polynomial)}
in terms of $G$ if it can be written as
$f = \sum_{i=1}^u l_i  \star g_i,$ 
with $l_i\in {\sf R}, g_i \in G$ and 
${\bf T}(l_i)\circ {\bf T}(g_i) \leq {\bf T}(f) \Forall i;$    
\item a left {\em (weak) Gr\"obner representation}\index{Gr\"obner!representation (of a polynomial)!weak}
in terms of $G$ if it can be written as
$f = \sum_{i=1}^\mu a_i \lambda_i  \star   g_i,$ 
with $a_i\in R\setminus\{0\},  \lambda_i\in {\Cal T}, g_i \in G$ and 
${\bf T}(\lambda_i  \star   g_i) \leq {\bf T}(f) \Forall i;$
\item a left {\em (strong) Gr\"obner representation}\index{representation!
Gr\"obner!(of a polynomial)}\index{Gr\"obner!representation (of a polynomial)!strong} in terms
of $G$ if it can be written as 
$f = \sum_{i=1}^\mu a_i \lambda_i  \star   g_i,$ 
with $a_i\in R\setminus\{0\},  \lambda_i\in {\Cal T}, g_i \in G$ and 
$${\bf T}(f) =  \lambda_1\circ {\bf T}(g_1) >  \lambda_i\circ {\bf T}(g_i)  \Forall i;$$ 
\item a right {\em Gr\"obner representation} 
in terms of $G$ if it can be written as
$f = \sum_{i=1}^u g_i  \star  r_i,$ 
with $r_i\in {\sf R}, g_i \in G$ and 
${\bf T}(g_i)\circ {\bf T}(r_i) \leq {\bf T}(f) \Forall i;$    
\item a right {\em (weak) Gr\"obner representation} 
in terms of $G$ if it can be written as
$f = \sum_{i=1}^\mu  g_i  \star  b_i \rho_i ,$
with $b_i\in R\setminus\{0\}, \rho_i\in {\Cal T}, g_i \in G$ and
${\bf T}(g_i   \star  \rho_i)  \leq {\bf T}(f)$ for each $i;$
\item a right {\em (strong) Gr\"obner representation} in terms
of $G$ if it can be written as 
$f = \sum_{i=1}^\mu g_i  \star  b_i \rho_i,$
with $ b_i\in R\setminus\{0\}, \rho_i\in {\Cal T}, g_i \in G$ and 
$${\bf T}(f) = {\bf T}(g_1)\circ\rho_1  >{\bf T}(g_i)\circ \rho_i   \Forall i;$$ 
\item a bilateral {\em(weak) Gr\"obner representation} 
in terms of $G$ if it can be written as
$f = \sum_{i=1}^\mu a_i \lambda_i \star   g_i  \star b_i \rho_i ,$
with $ a_i, b_i\in R\setminus\{0\}, \lambda_i, \rho_i\in {\Cal T}, g_i \in G$ and 
${\bf T}(\lambda_i \star g_i \star\rho_i)  \leq {\bf T}(f)$ for each $i;$    
\item a bilateral {\em (strong) Gr\"obner representation} in terms
of $G$ if it can be written as 
$f = \sum_{i=1}^\mu a_i \lambda_i  \star  g_i  \star  b_i \rho_i ,$
with $ a_i, b_i\in R\setminus\{0\}, \lambda_i, \rho_i\in{\Cal T}, g_i \in G$ and 
${\bf T}(f) =  \lambda_1\circ{\bf T}(g_1)\circ\rho_1   > \lambda_i\circ{\bf T}(g_i)\circ \rho_i   \Forall i.$
\item a restricted {\em(weak) Gr\"obner representation} 
in terms of $G$ if it can be written as
$f = \sum_{i=1}^\mu a_i   g_i\star  \rho_i ,$
with $ a_i\in R\setminus\{0\}, \rho_i\in {\Cal T}, g_i \in G$ and 
${\bf T}(g_i\star\rho_i)  \leq {\bf T}(f)$ for each $i;$    
\item a restricted {\em (strong) Gr\"obner representation} in terms
of $G$ if it can be written as 
$f = \sum_{i=1}^\mu a_i    g_i  \star  \rho_i ,$
with $ a_i\in R\setminus\{0\}, \rho_i\in{\Cal T}, g_i \in G$ and 
${\bf T}(f) =  {\bf T}(g_1)\circ\rho_1   >{\bf T}(g_i)\circ \rho_i   \Forall i.$
\end{itemize} 
\item For $f\in{\sf R}^m\setminus\{0\}, F \subset{\sf R}^m$, an element 
$h:=\NF(f, F)\in {\sf R}^m$ is called a \index{form!normal}
\begin{itemize}
\item (left, right, bilateral, restricted) {\em (weak) normal form}\index{normal form!(weak)}\index{form!normal}
of $f$ w.r.t. $F$, if
\begin{description}
\item[] $f - h\in{\Bbb I}(F)$ has a weak Gr\"obner representation in
terms of $F$, and 
\item[] $h \neq 0 \then {\bf M}(h) \notin{\bf M}\{{\Bbb I}({\bf M}\{F\})\};$
\end{description} 
\item (left, right, bilateral, restricted) {\em strong normal form}\index{normal form!strong}
of $f$ w.r.t. $F$, if
\begin{description}
\item[] $f - h\in{\Bbb I}(F)$ has a strong Gr\"obner representation in
terms of $F$, and 
\item[] $h \neq 0 \then {\bf M}(f) \notin{\bf M}(F).$
\qed\end{description}
\end{itemize}
\end{itemize}
\end{Definition}

\begin{Proposition}({\em cf.}  \cite{R2,R3})\label{57P1} 
For any set $F\subset {\sf R}^m\setminus\{0\}$, among the
following
conditions:
\begin{enumerate}
\item $f\in{\Bbb I}(F) \iff $ it has a (left, right, bilateral, restricted) strong Gr\"obner representation 
$f = \sum_{i=1}^\mu a_i \lambda_i  \star    g_i  \star  b_i\rho_i$  
in terms of $F$ which further satisfies
$${\bf T}(f) =  {\bf T}(\lambda_1  \star    g_1  \star   \rho_1)  > \cdots > {\bf T}(\lambda_i  \star   g_i  \star   \rho_i)   > \cdots;$$
\item $f\in{\Bbb I}(F) \iff $ it has a (left, right, bilateral, restricted) strong Gr\"obner representation in terms of $F$;
\item $F$ is a (left, right, bilateral, restricted) strong Gr\"obner basis of ${\Bbb I}(F)$; 
\item $f\in{\Bbb I}(F) \iff $ it has a (left, right, bilateral, restricted) weak Gr\"obner representation in terms of $F$;
\item $F$ is a (left, right, bilateral, restricted)  Gr\"obner basis of ${\Bbb I}(F)$;
\item $f\in{\Bbb I}(F) \iff $ it has a (left, right) Gr\"obner representation in terms of $F$;
\item for each $f\in{\sf R}^m\setminus\{0\}$ and any (left, right, bilateral, restricted)  strong normal form $h$ of $f$ w.r.t. $F$ we have
$f\in{\Bbb I}(F) \iff h = 0;$ 
\item for each $f\in{\sf R}^m\setminus\{0\}$ and any (left, right, bilateral, restricted)  weak normal form $h$ of $f$ w.r.t. $F$ we have
$f\in{\Bbb I}(F) \iff h = 0;$ 
\end{enumerate}
there are the implications
$$
\begin{array}{rcccccl}
(1) &\Leftrightarrow&(2) &\Rightarrow& (4) &\Leftrightarrow& (6)\\
&\Nearrow&\Updownarrow&&\Updownarrow&\Nwarrow&\cr
(7)&\Leftarrow&(3) &\Rightarrow& (5)&\Rightarrow&(8)\cr\end{array}$$
If 
$R$ is a skew field 
 we have also the implication $(4) \then (2)$ and as a consequence also $(5) \then (3)$.
\end{Proposition}

\begin{proof} The implications  $(1) \then (2) \then (4) \iff (6)$, $(3) \then (5)$, $(2) \then (3)$ and $(4) \then (5)$
 are trivial. 
 
 Ad $(3) \then (1)$:
for each $f\in{\Bbb I}_2(F)$ by assumption there are 
elements
$g\in F,
m=a\lambda, n=b\rho\in{\sf M}({\sf R})$ such that
${\bf M}(f) = {\bf M}(m  \star g  \star n)$.
Thus ${\bf T}(f) = {\bf T}(m  \star g  \star n)=\lambda\circ{\bf T}(g)\circ\rho$ and,
denoting 
$f_1 := f - m  \star g  \star n$,
we have
${\bf T}(f_1) < {\bf T}(f)$ so the claim follows 
by induction, since $<$ is a well ordering.

Ad $(5) \then (4)$: similarly, for each $f\in{\Bbb I}_2(F)$ by assumption there
are  elements
$g_i\in F, {\bf T}(g_i):= \tau_i{\bf e}_{l_i}$, 
$m_i=a_i\lambda_i, n_i=b_i\rho_i\in{\sf M}({\sf R})$ such that
\begin{itemize}
\item ${\bf T}(f) =  {\bf T}(\lambda_i  \star g_i  \star  \rho_i)=\lambda_i\circ\tau_i\circ\rho_i{\bf e}_{l_i}$ for each $i$,
\item $\lc(f) 
 = \sum_i a_i\alpha_{\lambda_i}(\lc(g_{i}))\alpha_{\lambda_i\tau_i}(b_i).$
\end{itemize}
It is then sufficient to denote 
$f_1 := f - \sum_i m_i  \star g_i  \star n_i$ 
in order to deduce the claim 
by induction, since
${\bf T}(f_1) < {\bf T}(f)$ and $<$ is a well ordering.

Ad $(4) \then (2)$: let $f\in{\Bbb I}_2(F)\setminus\{0\};$ (4) implies the existence of $g\in F, \lambda,\rho\in{\Cal T}$, such that
${\bf T}(f) = \lambda\circ{\bf T}(g)\circ\rho$.
Then  setting 
$f_1 := f -  \lc(f)\Bigl(\alpha_{\lambda}\left(\lc(g)\right)\Bigr)^{-1}\lambda  \star  g  \star \rho$ 
we deduce the claim 
by induction, since
${\bf T}(f_1) < {\bf T}(f)$ and $<$ is a well ordering.

Ad $(3) \then (7)$ and $(5) \then (8)$: either
\begin{itemize}
\item $h=0$ and $f=f-h\in{\Bbb I}(F)$ or
\item $h\neq 0$, ${\bf M}(h) \notin {\bf M}({\Bbb I}(F))$, $h\notin{\Bbb I}(F)$ and $f\notin{\Bbb I}(F)$.
\end{itemize}

Ad $(7) \then (2)$ and $(8) \then (4)$: for each $f\in{\Bbb I}(F)$, its normal form is  $h=0$ and
$f = f-h$ has a strong (resp.: weak) Gr\"obner representation in
terms of $F$.

\end{proof}

\begin{Proposition} (Compare  \cite[Proposition~22.2.10]{SPES})
 If $F$ is a (weak, strong) Gr\"obner basis of ${\sf I} := {\Bbb I}(F)$, then the following holds:
\begin{enumerate}
\item Let $g\in{\sf R}^m$ be a (weak, strong) normal form of $f$ w.r.t. $F$. If $g \neq 0$, then
$${\bf T}(g) = \min\{{\bf T}(h) : h - f \in{\Bbb I}(F)\}.$$
\item Let $f, f' \in  {\sf R}^m\setminus{\sf I}$ be such that $f-f'\in{\sf I}$. Let $g$
be a (weak, strong) normal form of $f$ w.r.t. $F$ and $g'$ be a (weak, strong) normal form of $f'$ w.r.t. $F$.
Then 
\begin{itemize}
\item ${\bf T}(g) = {\bf T}(g') =:\tau$ and
\item  $\lc(g)-\lc(g')\in{\sf I}_\tau:= \{\lc(f) : f\in{\sf I}, {\bf T}(f) = \tau\}\cup\{0\}\subset R$. 
\end{itemize}\end{enumerate}
\end{Proposition}

\begin{proof}\
\begin{enumerate}
\item Let $h\in  {\sf R}^m$ be such that $h-f\in{\sf I}$; then  $h-g\in{\sf I}$ and 
${\bf M}(h-g)\in {\bf M}\{{\sf I}\}$.
If ${\bf T}(g) > {\bf T}(h)$ then ${\bf M}(h-g) = {\bf M}(g)\not\in {\bf M}\{{\sf I}\}$, giving
a contradiction. 
\item The assumption implies that $f-g'\in {\sf I}$ so that, by the previous result, 
${\bf T}(g) \leq
{\bf T}(g').$
Symmetrically, $f'-g\in {\sf I}$ and  ${\bf T}(g') \leq {\bf T}(g).$
Therefore ${\bf T}(g) = {\bf T}(g') =\tau$; morevoer, either
\begin{itemize}
\item ${\bf T}(g-g') < \tau$ and ${\bf M}(g) = {\bf M}(g')$ so that $\lc(g)=\lc(g')$ or
\item ${\bf T}(g-g') = \tau$ and 
${\bf M}(g-g') = {\bf M}(g) - {\bf M}(g') = \Bigl(\lc(g)-\lc(g')\Bigr)\tau;$ thus, since $g-g'\in{\sf I}$, 
$\lc(g)-\lc(g')\in{\sf I}_\tau$.
\end{itemize}
\end{enumerate} 
\end{proof}

\subsection{Canonical forms (skew field case)} 

If $R := {\Bbb K}$ is  a skew field,
for any set $F \subset{\sf R}^m$
we denote ${\bf N}(F)$ the  (left, right, bilateral, restricted) order module
${\bf N}(F) := {\Cal T}^{(m)} \setminus {\bf T}(F)$
and 
$ {\Bbb K}[{\bf N}(F)]$ the (left, right, bilateral, restricted) ${\Bbb K}$-module ${\Bbb K}[{\bf N}(F)] :=\Span_{\Bbb K}({\bf N}(F)).$

\begin{Definition}  
For  any (left, right, bilateral, restricted)  module  ${\sf I} \subset {\sf R}^m$, 
the order module ${\bf N}({\sf I}) :={\Cal T}^{(m)} \setminus {\bf T}\{{\sf I}\} $ is called
the {\em escalier}\index{escalier} of ${\sf I} $.
\end{Definition}

We easily obtain the notion, the properties and
the computational algorithm (Figure~\ref{BuCan}) of (left, right, bilateral, restricted)
canonical forms:
\begin{figure}\index{Buchberger's!reduction}
\caption{Canonical Form Algorithms}
\label{BuCan}
\hrule
\begin{list}{}{}
\item $(g, \sum_{i=1}^\mu c_i \lambda_i  \star  g_i)$ := {\bf LeftCanonicalForm}$(f,G)$
\begin{list}{}{}
\item {\bf where}
\item \begin{list}{}{}
      \item $G$ is the left Gr\"obner basis of the left module ${\sf I} \subset  {\sf R}^m$,
      \item $f \in {\sf R}^m$, $g \in  {\Bbb K}[{\bf N}({\sf I})]$,
            $c_i \in  {\Bbb K}\setminus\{0\}, \lambda_i\in{\Cal T}, g_i \in G$,
      \item $f-g = \sum_{i=1}^\mu c_i \lambda_i  \star  g_i$ is a left strong 
            Gr\"obner representation in terms of $G$,
      \item ${\bf T}(f-g) = \lambda_1\circ {\bf T}(g_1) > \lambda_2\circ{\bf T}(g_2) 
            > \cdots > \lambda_\mu\circ {\bf T}(g_\mu)$.
      \end{list}
\item $h := f, i := 0, g := 0,$
\item {\bf While} $h \neq 0$ {\bf do}
\item \%\% $f = g + \sum_{j=1}^i c_j \lambda_j  \star  g_j + h,$
\item \%\% ${\bf T}(f-g) \geq {\bf T}(h);$
\item \%\% $i > 0 \then {\bf T}(f-g) = \lambda_1\circ{\bf T}(g_1) > \lambda_2\circ{\bf T}(g_2) > \cdots > \lambda_i\circ{\bf T}(g_i) > {\bf T}(h);$
\item \begin{list}{}{}
      \item {\bf If} ${\bf T}(h)\in{\bf T}_L(G)$ {\bf do}
						\item \begin{list}{}{}
            \item {\bf Let} $\lambda  \in{\Cal T}, \gamma \in G : \lambda\circ {\bf T}(\gamma) = {\bf T}(h)$
            \item $i := i+1, c_i := \lc(h)\alpha_{\lambda}\left(\lc(\gamma)\right)^{-1},$ 
                     $ \lambda_i := \lambda, g_i := \gamma,$ $h := h - c_i \lambda_i g_i.$
            \end{list}
      \item {\bf Else}
						\item \begin{list}{}{}
            \item \%\% ${\bf T}(h)\in {\bf N}({\sf I})$
            \item  $g := g + {\bf M}(h),\, h := h - {\bf M}(h)$
            \end{list}
      \end{list}
\end{list}
\hrule
\item $(g, \sum_{i=1}^\mu  g_i  \star  d_i \rho_i)$ := {\bf RightCanonicalForm}$(f,G)$
\begin{list}{}{}
\item {\bf where}
\item \begin{list}{}{}
      \item $G$ is the right Gr\"obner basis of the right module ${\sf I} \subset  {\sf R}^m$,
      \item $f \in {\sf R}^m$,
            $g \in   {\Bbb K}[{\bf N}({\sf I})]$,
            $d_i \in  {\Bbb K}\setminus\{0\}, \rho_i\in{\Cal T}, g_i \in G$,
      \item $f-g = \sum_{i=1}^\mu g_i   \star  d_i\rho_i $ is a right strong 
            Gr\"obner representation in terms of $G$,
      \item ${\bf T}(f-g) = {\bf T}(g_1)\circ\rho_1 > {\bf T}(g_2)\circ\rho_2 
            > \cdots > {\bf T}(g_\mu)\circ\rho_\mu $.
      \end{list}
\item $h := f, i := 0, g := 0,$
\item {\bf While} $h \neq 0$ {\bf do}
\item \begin{list}{}{}
      \item {\bf If} ${\bf T}(h)\in{\bf T}_R(G)$ {\bf do}
						\item \begin{list}{}{}
            \item {\bf Let} $\rho \in{\Cal T}, \gamma \in G :  {\bf T}(\gamma)\circ\rho = {\bf T}(h)$
            \item $i := i+1, d_i := \alpha^{-1}_{{\bf T}(\gamma)}\left(\lc(h)\lc(\gamma)^{-1}\right), 
            \rho_i := \rho, g_i := \gamma,$ 
            \item $h := h - g_i   \star  d_i\rho_i .$
            \end{list}
      \item {\bf Else}
						\item \begin{list}{}{}
            \item $g := g + {\bf M}(h),\, h := h - {\bf M}(h)$
            \end{list}
      \end{list}
\hrule
\end{list}
\end{list}
\end{figure}
\renewcommand{\thefigure}{\arabic{figure} (cont.)}\setcounter{figure}{0}
\begin{figure}\index{Buchberger's!reduction}
\caption{Canonical Form Algorithms}
\hrule
\begin{list}{}{}
\item $(g, \sum_{i=1}^\mu c_i \lambda_i  \star   g_i   \star  \rho_i )$ := {\bf BilateralCanonicalForm}$(f,G)$
\begin{list}{}{}
\item {\bf where}
\item \begin{list}{}{}
      \item $G$ is the bilateral Gr\"obner basis of the bilateral module ${\sf I} \subset  {\sf R}^m$,
      \item $f \in {\sf R}^m$,
            $g \in   {\Bbb K}[{\bf N}({\sf I})]$,
            $c_i \in  {\Bbb K}\setminus\{0\}, \lambda_i,\rho_i\in{\Cal T}, g_i \in G$,
      \item $f-g = \sum_{i=1}^\mu c_i \lambda_i  \star   g_i   \star \rho_i  $ is a bilateral strong 
            Gr\"obner representation in terms of $G$,
      \item ${\bf T}(f-g) = \lambda_1\circ {\bf T}(g_1)\circ\rho_1 > \lambda_2\circ {\bf T}(g_2)\circ\rho_2
            > \cdots > \lambda_\mu\circ {\bf T}(g_\mu)\circ\rho_\mu$.
      \end{list}
\item $h := f, i := 0, g := 0,$
\item {\bf While} $h \neq 0$ {\bf do}
\item \begin{list}{}{}
      \item {\bf If} ${\bf T}(h)\in{\bf T}_2(G)$ {\bf do}
						\item \begin{list}{}{}
            \item {\bf Let} $\lambda,\rho \in{\Cal T}, \gamma \in G :  
                 \lambda\circ{\bf T}(\gamma)\circ\rho = {\bf T}(h)$
            \item $i := i+1, c_i := \lc(h)\alpha_{\lambda}\left(\lc(\gamma)\right)^{-1}, \lambda_i := \lambda,
                  \rho_i := \rho, g_i := \gamma,$
            \item $h := h -c_i \lambda_i  \star   g_i   \star \rho_i .$
            \end{list}
      \item {\bf Else}
						\item \begin{list}{}{}
            \item  $g := g + {\bf M}(h),\, h := h - {\bf M}(h)$
            \end{list}
      \end{list}
\end{list}
\hrule
\item $(g, \sum_{i=1}^\mu c_i g_i   \star  \rho_i )$ := {\bf RestrictedCanonicalForm}$(f,G)$
\begin{list}{}{}
\item {\bf where}
\item \begin{list}{}{}
      \item $G$ is the restricted Gr\"obner basis of the  restricted module ${\sf I} \subset  {\sf R}^m$,
      \item $f \in {\sf R}^m$,
            $g \in   {\Bbb K}[{\bf N}({\sf I})]$,
            $c_i \in  {\Bbb K}\setminus\{0\}, \rho_i\in{\Cal T}, g_i \in G$,
      \item $f-g = \sum_{i=1}^\mu c_i  g_i   \star \rho_i  $ is a  restricted  strong 
            Gr\"obner representation in terms of $G$,
      \item ${\bf T}(f-g) =   {\bf T}(g_1)\circ\rho_1 >   {\bf T}(g_2)\circ\rho_2
            > \cdots >   {\bf T}(g_\mu)\circ\rho_\mu$.
      \end{list}
\item $h := f, i := 0, g := 0,$
\item {\bf While} $h \neq 0$ {\bf do}
\item \begin{list}{}{}
      \item {\bf If} ${\bf T}(h)\in{\bf T}_W(G)$ {\bf do}
						\item \begin{list}{}{}
            \item {\bf Let} $\rho \in{\Cal T}, \gamma \in G :  
                 {\bf T}(\gamma)\circ\rho = {\bf T}(h)$
            \item $i := i+1, c_i := \lc(h)\lc(\gamma)^{-1},  
                  \rho_i := \rho, g_i := \gamma,$
            \item $h := h -c_i   g_i   \star \rho_i .$
            \end{list}
      \item {\bf Else}
						\item \begin{list}{}{}
            \item  $g := g + {\bf M}(h),\, h := h - {\bf M}(h)$
            \end{list}
      \end{list}
\hrule
\end{list}
\end{list}
\end{figure}
%
\begin{Lemma}\label{LeCan} {\rm (cf.  \cite[Lemma~22.2.12]{SPES})} 
Let  ${\sf I} \subset {\sf R}^m$ be a (left, right, bilateral, restricted)  module.
If $R = {\Bbb K}$ is a skew field
and denoting
${\sf A}$ the (left, right,bi\-la\-teral, restricted) module ${\sf A} :={\sf R}^m/{\sf I}$
it holds
\begin{enumerate}
\item ${\sf R}^m \cong {\sf I} \oplus  {\Bbb K}[{\bf N}({\sf I})];$
\item ${\sf A} \cong   {\Bbb K}[{\bf N}({\sf I})];$
\item for each $f\in{\sf R}^m,$ there is a unique 
$$g := \Can(f,{\sf I})  = \sum_{t\in{\bf N}({\sf I})} \gamma(f,t,<) t 
\in   {\Bbb K}[{\bf N}({\sf I})]$$ such that $f - g \in {\sf I}$.

\noindent
Moreover:
\begin{enumerate}
\item $\Can(f_1,{\sf I}) = \Can(f_2,{\sf I}) \iff f_1-f_2 \in {\sf I};$
\item $\Can(f,{\sf I}) = 0 \iff f \in {\sf I}.$
\end{enumerate}
\item For each $f\in{\sf R}^m, f- \Can(f,{\sf I})$ has a (left, right, bilateral, restricted) strong
Gr\"obner representation in terms of any Gr\"obner basis.
\end{enumerate}
\end{Lemma}

\begin{Definition} {\rm (cf.   \cite[Definition~22.2.13]{SPES})}
For each $f\in {\sf R}^m$ the unique element $$g := \Can(f,{\sf I})
\in  {\Bbb K}[{\bf N}({\sf I})]$$ such that $f - g \in {\sf I}$ will be called the 
(left, right, bilateral, restricted) {\em canonical
form}\index{canonical!form}\index{form!canonical} of $f$  w.r.t. ${\sf I}$.
\qed\end{Definition}

\begin{Corollary}\label{C46x1} {\rm (cf.   \cite[Corollary~22.3.14]{SPES})}    If $R = {\Bbb K}$ is a skew field,
there is a unique set $G
\subset {\sf I}$ such that 
\begin{itemize}
\item ${\bf T}\{G\}$ is an irredundant basis of ${\bf T}({\sf I})$;
\item for each $g\in G, \lc(g) = 1$;
\item for each $g\in G, g = {\bf T}(g) - \Can({\bf T}(g),{\sf I}).$
\end{itemize}
$G$ is called the (left, right, bilateral, restricted) {\em reduced
Gr\"obner basis}\index{Gr\"obner!basis!reduced} of ${\sf I}$. 
\qed\end{Corollary}

Note that the algorithm described for right canonical forms is assuming that each $\alpha_i$ is an automorphism; alternatively we can assume that ${\sf R}$ is given as a right $R$-module in which case the theory can be developped symmetrically.

\section{Szekeres Theory}


Let 
${\sf I} \subset {\sf R}^m$ be a (left-bilateral) 
module; if 
we denote
for each 
$\tau\in{\Cal T}^{(m)}$, ${\sf I}_\tau$ the additive group 
$${\sf I}_\tau := \{\lc(f) : f\in{\sf I}, {\bf T}(f) = \tau\}\cup\{0\}\subset R,$$
${\Frak I} := \{{\sf I}_\tau : \tau\in {\Cal T}^{(m)}\}$ and,
for each ideal ${\Frak a}\subset R$ , 
$T_{\Frak a}$ and $L_{\Frak a}$ the sets
$$T_{\Frak a} := \{\tau\in{\Cal T}^{(m)} : {\sf I}_\tau \supseteq{\Frak a}\}\subset{\Cal T}^{(m)}
\And L_{\Frak a} := \{\tau\in{\Cal T}^{(m)} : {\sf I}_\tau ={\Frak a}\}\subset{\Cal T}^{(m)},$$
 we have
\begin{enumerate}
\item  for each $\tau\in{\Cal T}^{(m)}$, ${\sf I}_\tau\subset R$ is a left ideal;
\item for each ideals ${\Frak a},{\Frak b}\subset R$,
${\Frak a}\subset{\Frak b} \then T_{\Frak a} \supset T_{\Frak b};$
\item $T_{\Frak a} = \bigsqcup\limits_{{\Frak b}\supseteq{\Frak a}}  L_{\Frak b}$,
$L_{\Frak a} = T_{\Frak a}\setminus\bigcup\limits_{{\Frak b}\supsetneqq{\Frak a}}  T_{\Frak b}$;
\item for terms $\tau,\omega\in{\Cal T}^{(m)}$, $\tau\mid\omega \then {\sf I}_\tau\subset{\sf I}_\omega$;
\item for each ideal ${\Frak a}\subset R$, $T_{\Frak a}\subset{\Cal T}^{(m)}$ is a right semigroup module.
\qed\end{enumerate}

 If $R$ is a skew field, the situation is quite trivial: for any ideal ${\sf I}$ we have
$${\Frak I} = \{(0), R\}, T_{R} = L_{R} = {\bf T}({\sf I}),
T_{(0)} = {\Cal T}^{(m)}, L_{(0)} =  {\Cal T}^{(m)}\setminus{\bf T}({\sf I}).$$

Szekeres notation is related with a pre-Buch\-ber\-ger construction of ``canonical'' ideals for the case of polynomial rings $R[Y_1,\ldots,Y_n]$ over a PID $R$.

In connection  recall that  \cite{C1,C2}
a  not necessarily commutative  ring $R$ is called a
(left, right, bilateral)  {\em B\'ezout ring} if every finitely generated (left, right, bilateral) ideal is principal 
and is called a {\em B\'ezout domain} if it is both a B\'ezout ring and is a domain, 
and  remark that, if $R$ is a noetherian (left, bilateral) B\'ezout ring, then  for each 
$\tau\in{\Cal T}^{(m)}$,
there is  a value $c_\tau\in R$ satisfying ${\sf I}_\tau = {\Bbb I}(c_\tau)$.

\begin{Definition}\label{SzeDef} With the present notation, we call {\em Szekeres ideal}\index{Szekeres!ideal}\index{ideal!Szekeres} each ideal  ${\sf I}_\tau\subset R$  and {\em Szekeres level}\index{Szekeres!level}
each set $L_{\Frak a}\subset {\Cal T}^{(m)}$,  {\em Szekeres semigroup}\index{Szekeres!semigroup}
each semigroup $T_{\Frak a}\subset {\Cal T}^{(m)}$.

Finally, if $R$ is a noetherian left  B\'ezout ring we call
 {\em Szekeres generator}\index{Szekeres!principal generator} each 
value $c_\tau\in R$ satisfying ${\sf I}_\tau =  {\Bbb I}_L(c_\tau)$.
\end{Definition} 

Note that if $R$ is  a 
noetherian  B\'ezout ring, we have, 
$$\omega\mid\tau \then c_\tau \mid_L \alpha_\lambda(c_\omega)  \Forall 
\lambda,\rho\in{\Cal T} \textrm{ s.t. } \tau=\lambda\circ\omega\circ\rho.$$

\begin{Proposition}[Szekeres] \cite{S}\label{SzePro} Let $R$ be  a 
noetherian left B\'ezout ring and ${\sf I} \subset {\sf R}^m$ be a (left,bilateral) module.
Denote
$${\sf T} := \left\{\tau\in{\Cal T}^{(m)} \textrm{ s.t. } c_\tau\notin{\Bbb I}(\alpha_\lambda(c_\omega), \,   
\omega\in {\Cal T}^{(m)}, \lambda,\rho\in{\Cal T},\tau=\lambda\circ\omega\circ\rho)\right\}\subset {\Cal T}^{(m)}$$
and fix, for each $\tau\in{\sf T}$, any element $f_\tau\in{\sf I}$ such that\footnote{Of course for the extreme case ${\sf I}_\tau = (0)$ so that $c_\tau = 0$, we have $f_\tau := 0$.
 } ${\bf M}(f_\tau) = c_\tau\tau$.\\
Then the basis 
$S_w := \left\{f_\tau \textrm{ s.t. } \tau\in {\sf T}\right\}$   
is a left/bilateral weak Gr\"obner basis of ${\sf I}$.
\end{Proposition}

\begin{proof} For each $f\in{\sf I}$, denoting $\tau := {\bf T}(f)$ we have $\lc(f)\in{\Bbb I}_L(c_\tau)$ and $\lc(f) = dc_\tau$ for suitable $d\in R\setminus\{0\}$.
Thus if $\tau\in{\sf T}$ we have ${\bf M}(f) = d {\bf M}(f_\tau)$; if, instead, $\tau\not\in{\sf T}$
there are suitable 
$d_i,\in R\setminus\{0\},\omega_i\in{\sf T}\subset{\Cal T}^{(m)},\lambda_i,\rho_i\in{\Cal T}$
for which
$\lambda_i\circ\omega_i\circ\rho_i=\tau$ and $c_\tau = \sum_i d_i\alpha_{\lambda_i}(c_{\omega_i})$
so that
\begin{eqnarray*}
{\bf M}(f) = d c_\tau \tau
&=&
d \left(\sum_i d_i \alpha_{\lambda_i}(c_{\omega_i})
\lambda_i\circ\omega_i\circ\rho_i\right) 
\\&=&
\sum_i \left(d d_i \lambda_i\right)\cdot \left(c_{\omega_i} \omega_i\right)\cdot\rho_i
\\&=&
\sum_i  \left(d d_i\lambda_i\right)\ast{\bf M}(f_{\omega_i})\ast\rho_i.\end{eqnarray*}
\end{proof}

\begin{Remark}\label{RightSzaRem+} Remark that  in the case in which 
each endomorphism $\alpha_\tau,\tau\in{\Cal T}^{(m)},$ is an automorphism, we can consider also {\em right} modules ${\sf I}$ to which we can associate 
$${\sf I}_\tau = \{\lc(f) : f\in{\sf I}, {\bf T}(f) = \tau\}\cup\{0\}$$
which are   {\em right} ideals themselves; in fact if we represent 
 $f\in{\sf R}^m$ as (see 
Remark~ \ref{OreRem1}) $f=\sum_{i=1}^n  Y^i \bar{a}_i$
and we denote $_\tau{\sf I}$ the {\em right} ideal
$$_\tau{\sf I} := \left\{c\in R : \tau c\in{\bf M}\{{\sf I}\}\right\}\cup\{0\}\subset R$$
then ${\sf I}_\tau$ is the right ideal $\alpha_\tau({_\tau{\sf I}})$.

However, in this setting, Szekeres Theory can be built more easily by considering the ideals 
$_\tau{\sf I}$ obtained through the right representation of 
Remark~\ref{OreRem1}
and adapting to them the results reported above.

Remark that  if an endomorphism $\alpha_\tau$ is not invertible, in general ${\sf I}_\tau$ is not an ideal but just an additive group.

Finally note that for {\em restricted} modules, one apply {\em verbatim},  the classical Szekeres theory and subistitute in the results above each instance of 
$\alpha_\lambda(c_\omega),  \tau=\lambda\circ\omega\circ\rho$ with $c_\omega,
\tau=\omega\circ\rho$
\qed\end{Remark}

\begin{Example}  In the  Ore extension
$${\sf R}:=R[Y;\alpha], R={\Bbb Z}_2[x] \Where\alpha : R \to R  : x \mapsto x^2$$
we can consider, as a left module, the two-sided ideal
${\Bbb I}_2(x) = {\Bbb I}_L\{xY^i: i\in{\Bbb N}\}$; we thus have 
$${\sf I}_\tau =  {\Bbb I}(x)\subset R \Forall 
\tau\in\{Y^i, i \geq 0\},$$
so that, setting ${\Frak a} :=  {\Bbb I}(x)\subset R$, it holds
${\Frak I} = \{ {\Frak a}\}, T_{\Frak a} = L_{\Frak a} = \{Y^i :i\in{\Bbb N}\},$
and $S_w=\{xY^i: i\in{\Bbb N}\}$ is both a weak and a strong   
 Gr\"obner basis of ${\Bbb I}_L(x)$.

For the right ideal
${\Bbb I}_R(xY)$ the sets ${\sf I}_\tau$ are {\em not}  ideals;
we have, {\em e.g.}
$${\sf I}_{Y^i} = \{x\phi(x^{e^i}) | \phi(x)\in k[x]\}.$$
  \end{Example}

\section{Zacharias canonical representation}\label{46S1G}

Following Zacharias approach  to Buchberger Theory  \cite{Z},  if each module ${\sf I}\subset{\sf R}^m$ has a groebnerian property, necessarily the same property must be satisfied at least by the modules 
${\sf I}\subset R^m\subset{\sf R}^m$ and thus such property in $R^m$ can be used to device a procedure granting the same property in ${\sf R}^m$. The most elementary application of Zacharias approach is the generalization of the property of canonical forms from the case in which $R={\Bbb K}$ is a skew field to the general case: all we need is an effective notion of canonical forms for modules in $R$:

\begin{Definition}[Zacharias]  \cite{Z} \label{c46D1}
 A   ring $R$ is said to have
{\em  canonical representatives}\index{canonical!representatives} if
 there is an algorithm  which, given an element $c\in R^m$ and a  (left,bilateral, right)  module
 ${\sf J}\subset R^m$,
computes a {\em unique} element $\Rep(c,{\sf J})$ such that
\begin{itemize}
\renewcommand\labelitemi{\bf --}
\item  $c-\Rep(c,{\sf J})\in{\sf J}$,
\item $\Rep(c,{\sf J}) = 0 \iff c\in{\sf J}$.
\end{itemize}
The set 
$${\bf Zach}(R^m/{\sf J}) := \Rep({\sf J}) := \left\{\Rep(c,{\sf J}) :c\in R^m \right\}\cong R^m/{\sf J}$$
is called the {\em canonical Zacharias representation} of the module $R^m/{\sf J}$.
\qed\end{Definition}

Remark that, for each $c,d\in R^m$ and each module ${\sf J}\subset R^m$, we have
$$c-d\in{\sf J} \iff \Rep(c,{\sf J})=\Rep(d,{\sf J}).$$

Using Szekeres notation for a  (left, right, bilateral)  module  ${\sf I} \subset {\sf R}^m$ we obtain 
\begin{itemize}
\item the partition 
${\Cal T}^{(m)} = {\bf L}({\sf I})\sqcup{\bf R}({\sf I})\sqcup{\bf N}({\sf I})$
of ${\Cal T}^{(m)}$
where 
\begin{itemize}
\renewcommand\labelitemi{\mathbf --}
\item ${\bf N}({\sf I}) := {\bf L}_{(0)} = \{\omega\in{\Cal T}^{(m)} : {\sf I}_\omega = (0)\}$,
\item ${\bf L}({\sf I}) := {\bf L}_R = \{\omega\in{\Cal T}^{(m)} : {\sf I}_\omega = R\}$,
\item ${\bf R}({\sf I}) := \left\{\omega\in{\Cal T}^{(m)} : {\sf I}_\omega  \notin\left\{(0),R\right\}\right\};$
\end{itemize}
\item  the {\em canonical Zacharias representation} 
\begin{eqnarray*}
{\bf Zach}({\sf R}^m/{\sf I}) := \Rep({\sf I}) := 
 \Bigl\{\Rep(c,{\sf I}) :c\in {\sf R}^m \Bigr\} &=&
\bigoplus\limits_{{\Frak a}\in{\Frak I}} \bigoplus\limits_{\tau\in L_{\Frak a}}
\Rep({\Frak a})\tau 
\\ &=& 
 \bigoplus\limits_{\tau\in{\Cal T}^{(m)}} \Rep({\sf I}_\tau)\tau
\\ &=& \bigoplus\limits_{\tau\in{\Cal T}^{(m)}}{\bf Zach}(R/{\sf I}_\tau)\tau
\cong {\sf R}^m/{\sf I}
\end{eqnarray*}
of the module ${\sf R}^m/{\sf I}.$
\end{itemize}

If $R$ has canonical representatives and there is an algorithm ({\em cf.} Definition~\ref{c46D3} (c), (e))  which, given an element $c\in R^m$ and a 
 (left, right, bilateral) 
module ${\sf J}\subset R^m$ computes the unique canonical representative
 $\Rep(c,{\sf J})$, an easy adaptation of  Figure~\ref{BuCan} allows to extend, from the field coefficients
 case to the Zacharias ring  \cite{Z,Zac} coefficients case, the notion of canonical forms, the algorithm  (Figure~\ref{BeLCanF}) for computing them and 
their characterizing properties:

\begin{Lemma}\label{c46LCan2} If $R$ has canonical representatives,
also ${\sf R}$ has  canonical representatives\index{Zacharias!canonical representation}.

With the present notation and denoting
${\sf A}$ the (left, right, bilateral) module ${\sf A} :={\sf R}^m/{\sf I}$
it holds:
\begin{enumerate}
\item ${\sf R}^m \cong {\sf I} \oplus  \Rep({\sf I});$
\item ${\sf A} \cong  \Rep({\sf I});$
\item for each  $f\in{\sf R}^m,$ there is a unique 
(left, right, bilateral) {\em canonical
form}\index{canonical!form} of $f$
$$g := \Can(f,{\sf I})  = 
\sum_{{\Frak a}\in{\Frak I}} \sum_{\tau\in L_{\Frak a}}\gamma(f,\tau,{\sf I},<) \tau \in \Rep({\sf I}),
\quad \gamma(f,\tau,{\sf I},<)\in \Rep({\sf I}_\tau),
$$
such that
\begin{itemize}
\item $f - g \in {\sf I}$,
\item  $\gamma(f,\tau,{\sf I},<) = \Rep(\gamma(f,\tau,{\sf I},<),{\sf I}_\tau)\in \Rep({\sf I}_\tau)$,
for each $\tau\in{\Cal T}^{(m)}$.
\end{itemize}

\noindent
Moreover:
\begin{enumerate}
\item $\Can(f_1,{\sf I}) = \Can(f_2,{\sf I}) \iff f_1-f_2 \in {\sf I};$
\item $\Can(f,{\sf I}) = 0 \iff f \in {\sf I};$
\end{enumerate}
\item for each $f\in{\sf R}^m, f- \Can(f,{\sf I})$ has a  
(left, right, bilateral) (weak, strong) Gr\"obner representation in terms of any 
 (weak, strong)  Gr\"obner basis.
\qed\end{enumerate}\end{Lemma}

\renewcommand{\thefigure}{\arabic{figure}}
\begin{figure}{\small  
\caption{Canonical Form Algorithms}
\label{BeLCanF}
\hrule
\begin{list}{}{}
\item $(g, \sum_{i=1}^\mu a_i \lambda_i  \star  g_i)$ := {\bf LeftCanonicalForm}$(f,F)$
\begin{list}{}{}
\item {\bf where}
\item \begin{list}{}{}
          \item ${\sf R} := R[ {\Cal T}]$, $R$ a ring with canonical representatives,
           \item $f \in {\sf R}^m$,
           $F$ is the left Gr\"obner basis of the left module ${\sf I} \subset {\sf R}^m$,
          \item $g := \Can(f,{\sf I}) \in{\bf Rep}({\sf I})$, $a_i \in R\setminus\{0\}, \lambda_i\in {\Cal T}, g_i \in F$,
          \item $f-g = \sum_{i=1}^\mu a_i \lambda_i  \star  g_i$ is a left weak 
           Gr\"obner representation in terms of $F$,
           \end{list}
\item $h := f, \mu := 0, g := 0$
\item {\bf While} $h \neq 0$ {\bf do}
\item \begin{list}{}{}
           \item {\bf Let} $c\tau := {\bf M}(h), \gamma:=\Rep(c,{\sf I}_\tau)$
           \item $h := h -\gamma\tau, g := g+\gamma\tau$,
          \item {\bf If}  $c\neq\gamma$, 
                      {\bf let} $g_i\in F, \lambda_i\in{\Cal T}, a_i\in R\setminus\{0\}:$
                     \item \begin{list}{}{}
                               \item $c -  \gamma=\sum_{i=\mu+1}^\nu a_i\alpha_{\lambda_i}(\lc(g_i))$,
                                  ${\bf T}(g) =\lambda_i\circ {\bf T}(g_i), \mu < i \leq \nu$,     
                                \end{list}
                       \item  $h := h-\sum_{i=\mu+1}^{\nu} a_i \lambda_i  \star  g_i$, $\mu:= \nu$, 
                       \end{list}
\end{list}
\hrule
\item
\item $(g, \sum_{i=1}^\mu  g_i  \star  b_i \rho_i)$ := {\bf RightCanonicalForm}$(f,F)$
\begin{list}{}{}
\item {\bf where}
\item \begin{list}{}{}
          \item ${\sf R} := R[ {\Cal T}]$, $R$ a ring with canonical representatives,
           \item $f \in {\sf R}^m$,
           $F$ is the right Gr\"obner basis of the right module ${\sf I} \subset {\sf R}^m$,
          \item $g := \Can(f,{\sf I}) \in{\bf Rep}({\sf I})$, $b_i \in R\setminus\{0\}, \rho_i\in {\Cal T}, g_i \in F$,
          \item $f-g = \sum_{i=1}^\mu g_i  \star  b_i \rho_i$ is a right weak 
           Gr\"obner representation in terms of $F$,
      \end{list}
\item $h := f, \mu := 0, g := 0$
\item {\bf While} $h \neq 0$ {\bf do}
\item \begin{list}{}{}
       \item{\bf Let} $c\tau := {\bf M}(h), \gamma:=\Rep(c,{\sf I}_\tau)$
           \item $h := h -\gamma\tau, g := g+\gamma\tau$,
          \item {\bf If}  $c\neq\gamma$, 
                      {\bf let} $g_i\in F, \rho_i\in{\Cal T}, b_i\in R\setminus\{0\} :$
                     \item \begin{list}{}{}
                                \item $c -  \gamma=\sum_{i=\mu+1}^\nu \lc(g_i)b_i$,
                                 ${\bf T}(g) = {\bf T}(g_i)\circ\rho_i, \mu < i \leq \nu$,     
                                \end{list}
                       \item   $h := h -\sum_{i=\mu+1}^{\nu} g_i  \star  \alpha^{-1}_{{\bf T}(g_i)}(b_i) \rho_i,$ $\mu := \nu$, 
                       \end{list}
\end{list}
\hrule
\end{list}
}\end{figure}

\renewcommand{\thefigure}{\arabic{figure} (cont.)}\setcounter{figure}{1}
\begin{figure}{ 
\caption{Canonical Form Algorithms}
\hrule
\begin{list}{}{}
\item $(g, \sum_{i=1}^\mu a_i \lambda_i\star g_i\star  \rho_i)$ := {\bf BilateralCanonicalForm}$(f,F)$
\begin{list}{}{}
\item {\bf where}
\item \begin{list}{}{}
          \item ${\sf R} := R[ {\Cal T}]$, $R$ a ring with canonical representatives,
           \item $f \in {\sf R}^m$,
           $F$ is the bilateral Gr\"obner basis of the bilateral module ${\sf I} \subset {\sf R}^m$,
          \item $g := \Can(f,{\sf I}) \in{\bf Rep}({\sf I})$, $a_i\in R\setminus\{0\}, \lambda_i,\rho_i\in {\Cal T}, g_i \in F$,
          \item $f-g = \sum_{i=1}^\mu a_i \lambda_i\star  g_i\star \rho_i$ is a bilateral 
           Gr\"obner representation in terms of $F$,
      \end{list}
\item $h := f, \mu := 0, g := 0$
\item {\bf While} $h \neq 0$ {\bf do}
           \item {\bf Let} $c\tau := {\bf M}(h), \gamma:=\Rep(c,{\sf I}_\tau)$
           \item $h := h -\gamma\tau, g := g+\gamma\tau$,
          \item {\bf If}  $c\neq\gamma$, 
                    \item {\bf let} $g_i\in F, \lambda_i, \rho_i\in{\Cal T}, a_i \in R\setminus\{0\} :$
                     \item \begin{list}{}{}
                                \item $c -  \gamma=\sum_{i=\mu+1}^\nu a_i\alpha_{\lambda_i}(\lc(g_i))$,
                                 ${\bf T}(g) =\lambda_i\circ {\bf T}(g_i)\circ\rho_i, \mu < i \leq \nu$,     
                                \end{list}
                       \item $h := h - \sum_{i=\mu+1}^\nu a_i \lambda_i\star  g_i \star  \rho_i$, $\mu := \nu$, 
\end{list}
\end{list}
\hrule
}\end{figure}

\section{M\"oller's Lifting Theorem}\label{LeLiTh}
\subsection{Left case}
The validity of Eqs. (\ref{Eq3}) and  (\ref{Eq4}) allows  
 to intoduce the groebnerian  terminology and, as in the  standard theory of commutative polynomial rings over a field  \cite[\S~21.1-2]{SPES} or a Zacharias ring  \cite{Z},
 the ability of imposing a ${\Cal T}^{(m)}$-valuation on modules over ${\sf R}$ 
and its associated  graded Ore extension  ${\sf S} := G({\sf R})$ (see Remark~\ref{RemVal}).

The only twist w.r.t. the classical theory  is that there the ring was coinciding with its  associated  graded ring; here they coincide as  sets and as left $R$-modules,  but as rings  have two different multiplications.

Consequently, denoting by $\star$ the one of ${\sf R}$ and by $\ast$  the one of ${\sf S}$,
given a finite basis
$$F := \{g_1,\ldots,g_u\}\subset {\sf R}^m, 
g_i = {\bf M}(g_i)-p_i =: c_i \tau_i {\bf e}_{l_i} - p_i,$$
with respect to  the module ${\sf M} := {\Bbb I}_L(F)\subset {\sf R}^m$ we need to consider 
the morphisms
\begin{eqnarray*}
{\Frak s}_L : {\sf S}^u \to {\sf S}^m 
&:& {\Frak s}_L\left(\sum_{i=1}^u h_i e_i\right) := \sum_{i=1}^u h_i\ast {\bf M}(g_i),
\\
{\Frak S}_L : {\sf R}^u \to {\sf M}\subset{\sf R}^m 
&:& {\Frak S}_L\left(\sum_{i=1}^u h_i e_i\right) := \sum_{i=1}^u h_i\star g_i,
\end{eqnarray*}
where the symbols $\{e_1,\ldots,e_u\}$ denote the common canonical basis of
${\sf S}^u$ and ${\sf R}^u$ which, as $R$-modules, coincide.

We can then consider 
\begin{itemize}
\item the ${\Cal T}^{(m)}$-valuation $v : {\sf R}^u \to {\Cal T}^{(m)}$ 
defined,
 for each $\sigma  := \sum_{i=1}^u h_i e_i\in{\sf R}^u\setminus\{0\}$,
by
$$v(\sigma) := \max_<\{{\bf T}(h_i\star g_i)\} = 
\max_<\{{\bf T}_\prec(h_i)\circ{\bf T}_<(g_i)\} 
= 
\max_<\{{\bf T}_\prec(h_i)\circ\tau_i{\bf e}_{l_i}\} =: \delta\epsilon
$$
under which we further have ${\sf S}^u= G({\sf R}^u)$;
\item  the corresponding {\em leading form}\index{leading!form}\index{form!leading} 
${\Cal L}_L(\sigma) := \sum_{h\in H} {\bf M}(h_h) e_{h}\in {\sf S}^u $ -- which is ${\Cal T}^{(m)}$-homo\-geneous of ${\Cal T}^{(m)}$-degree $v(\sigma)=\delta\epsilon$
-- where
$$H := \left\{j : {\bf T}_<(h_j\star g_j) = {\bf T}_\prec(h_j)\circ\tau_j{\bf e}_{l_j}
=\delta\epsilon = v(\sigma)\right\}.$$
\end{itemize}

\begin{Definition}\label{46D2} 
With the notation above and denoting
for each set $S\subset {\sf R}^u$,
${\Cal L}_L\{S\} := \{{\Cal L}_L(g) : g\in S\}\subset {\sf S}^u$,
\begin{itemize}
\item for a left ${\sf R}$-module ${\sf N}\subset {\sf R}^u$,
a set $B\subset{\sf N}$ is called a left {\em standard basis}\index{standard!basis} 
if  ${\Bbb I}_L({\Cal L}_L\{B\}) = {\Bbb I}_L({\Cal L}_L\{{\sf N}\});$ 
\item for each $h\in {\sf N}$  a representation 
$$h = \sum_i h_i\star g_i : h_i\in {\sf R}, g_i \in B,$$ is called a 
 {\em left standard representation}\index{standard!representation} in ${\sf R}$ in terms
of $B$ iff $$v(h) \geq   v(h_i\star g_i), \mbox{\rm\ for each } i;$$ 
\item if $u\in \ker({\Frak s}_L)$ is ${\Cal T}^{(m)}$-homogeneous
and $U\in \ker({\Frak S}_L)$ is such that $u = {\Cal L}_L(U)$, we say that 
$u$ {\em lifts} to $U$, or $U$ is a {\em lifting} of $u$, or simply $u$ {\em has
a lifting};\index{lift}
\item a {\em left Gebauer--M\"oller set}\index{Gebauer--M\"oller!set} for $F$ is any 
${\Cal T}^{(m)}$-homogeneous basis of
$\ker({\Frak s}_L)$;
\item for each ${\Cal T}^{(m)}$-homogeneous element $\sigma\in{\sf R}^u$,
we say that ${\Frak S}_L(\sigma)$  has a 
left  {\em quasi-Gr\"obner representation}
 \index{quasi-Gr\"obner representation}
 in terms of $F$ if it can be written as
 ${\Frak S}_L(\sigma) = 
  \sum_{i=1}^u l_i \star  g_{i}$
with\\ $v(\sigma) > {\bf T}(l_i \star  g_{i}) = {\bf T}(l_i)\circ {\bf T}(g_i) \Forall i.$
\end{itemize}\end{Definition}

\begin{Remark} Note that 
each Gr\"obner representation of  ${\Frak S}_L(\sigma)$ in terms of $F$ gives also a 
quasi-Gr\"obner representation since ${\bf T}(l_i)\circ {\bf T}(g_i)\leq{\bf T}({\Frak S}_L(\sigma))<v(\sigma)$;
on the other side, a quasi-Gr\"obner representation grants only 
${\bf T}(l_i)\circ {\bf T}(g_i)<v(\sigma)$
but not necessarily 
${\bf T}(l_i)\circ {\bf T}(g_i)\leq{\bf T}({\Frak S}_L(\sigma))$, since in principle we could have 
${\bf T}({\Frak S}_L(\sigma)) < {\bf T}(l_i)\circ {\bf T}(g_i)<v(\sigma)$
so that we don't necessarily obtain a Gr\"obner representation of the S-polynomial
${\Frak S}_L(\sigma)$.

This relaxation was   introduced by
Gebauer-M\"oller in their reformulation of  Buchberger Theory for polynomial rings over a field  \cite{GM};
in that setting, it allowed to better remove useless S-pairs and thus granted  a more efficient reformulation of the algorithm; in the more general setting we are considering now, viz polynomials over {\em rings}, it becomes essential also for a smooth reformulation of the theory.
\qed\end{Remark}
Observe that  if $\sigma := \sum_{j=1}^u h_je_j\in\ker({\Frak S}_L)$ 
then denoting 
$$\delta\epsilon := v(\sigma) \And
H := \left\{j, 1\leq j\leq u : {\bf T}(h_j)\circ{\bf T}(g_j) =\delta\epsilon\right\},$$
its {\em leading form} ${\Cal L}_L(\sigma) := \sum_{j=1}^u d_j\lambda_je_j\in{\sf S}^u$ 
is 
${\Cal T}^{(m)}$-homogeneous of ${\Cal T}^{(m)}$-degree $v(\sigma) := \delta\epsilon\in{\Cal T}^{(m)}$,
satisfies
\begin{itemize}
\item $0 \neq d_j  \iff j\in H\And  {\bf M}(h_j) = d_j \lambda_j $,
\item  $\sum_{j=1}^u d_j\lambda_j\ast{\bf M}(g_j) 
 = \sum_{j\in H} (d_j\lambda_j)\ast(c_j \tau_j {\bf e}_{l_j})
=\left(\sum_{j\in H} \left(d_j\alpha_{\lambda_j}(c_j)\right)\cdot\left(\lambda_j\tau_j\right) \right)\epsilon = 0$,
\item $\sum_{j\in H} d_j \alpha_{\lambda_j}(\lc(g_j)) = 0$ and $\lambda_j\circ {\bf T}(g_j) = \delta\epsilon$ for each $j\in H$ , so that in particular
\item $\epsilon= {\bf e}_{l_j}$  for each $j\in H$,
\end{itemize}
and belongs to $\ker({\Frak s}_L)$.
 
\begin{Theorem}[M\"oller; Left Lifting Theorem] \emph{\cite{M}} \label{SAASSAAS}\

With the present notation and denoting 
$\GM(F)$   any left  Gebauer--M\"oller set for $F$, the following
conditions are equivalent: 
\begin{enumerate}
\item $F$ is a left Gr\"obner basis of ${\sf M}$;
\item $f \in {\sf M}\iff f$  has a left Gr\"obner representation in terms of $F$;
\item for each  
$\sigma\in\GM(F)$, 
the S-polynomial\index{S-polynomial} ${\Frak S}_L(\sigma)$ has a   quasi-Gr\"obner representation 
$${\Frak S}_L(\sigma) = \sum_{i=1}^u  l_i g_i;$$
\item each $\sigma\in\GM(F)$ has a lifting $\lift(\sigma)$;
\item each ${\Cal T}^{(m)}$-homogeneous element 
$u\in\ker({\Frak s}_L)$ has a lifting $\lift(u)$.
\end{enumerate}
\end{Theorem}
 
 \begin{proof} \
\begin{description}
\item[$(1) \then (2)$] Let $f\in{\sf M}$; by assumption $${\bf M}(f) =
\sum_{i=1}^u a_i \lambda_i\ast {\bf M}(g_i)$$ where $(a_1 \lambda_1,\ldots,a_u
\lambda_u)$ is ${\Cal T}^{(m)}$-homogeneous  of ${\Cal T}^{(m)}$-degree ${\bf T}(f)$.

Therefore $g := f - \sum_{i=1}^u  a_i \lambda_i \star g_i \in {\sf M}$ and
${\bf T}(g) < {\bf T}(f)$.

Thus, we can assume by induction the existence of a  
 Gr\"obner representation
$g := \sum_{i=1}^u l_i\star g_i$ of $g$; whence
$f := \sum_{i=1}^u (a_i \lambda_i+l_i)\star g_i$
is the required  Gr\"obner representation  of $f$.

\item[$(2) \then (3)$] ${\Frak S}_L(\sigma)\in{\sf M}$ and ${\bf T}({\Frak S}_L(\sigma)) < v(\sigma).$
\item[$(3) \then (4)$] Let ${\Frak S}_L(\sigma) = \sum_{i=1}^u l_i \star g_i$ be a
quasi-Gr\"obner representation in terms of $F$; then 
${\bf T}(l_i\star g_i) < v(\sigma)$ so that 
$\lift(\sigma) := \sigma -
\sum_{i=1}^u l_i e_i$ is the required lifing of $\sigma$. 
\item[$(4) \then (5)$] Let 
$u := \sum_{i=1}^u a_i \lambda_i e_i, a_i \neq 0 \then  \lambda_i\circ \tau_i {\bf e}_{l_i} = v(u)$,
be a ${\Cal T}^{(m)}$-homogeneous element in $\ker({\Frak s}_L)$ of ${\Cal T}^{(m)}$-degree $v(u)$.

Then there are $c_\sigma\in R, \lambda_\sigma\in{\Cal T}$
for which
$$u  = \sum_{\sigma\in\GM(F)} c_\sigma \lambda_\sigma\ast \sigma,\quad
 \lambda_\sigma \circ v(\sigma) = v(u).$$
 
 For each $\sigma\in\GM(F)$ denote
 $$\bar{\sigma} := \sigma - \lift(\sigma) =  {\Cal L}_L(\lift(\sigma)) -
\lift(\sigma) := \sum_{i=1}^u l_{i\sigma}e_i$$
 and remark that
 ${\bf T}(l_{i\sigma})\circ \tau_i {\bf e}_{l_i} \leq v(\bar{\sigma}) < v(\sigma)$,
 ${\Frak S}_L(\lift(\sigma)) = 0$ and 
 ${\Frak S}_L(\bar{\sigma}) = {\Frak S}_L(\sigma).$
 
  It is sufficient to define
 $$\lift(u)  := \sum_{\sigma\in\GM(F)} c_\sigma \lambda_\sigma\star \lift(\sigma),
\And 
\bar{u}  := \sum_{\sigma\in\GM(F)} c_\sigma \lambda_\sigma\star \bar{\sigma}$$
to obtain
$\lift(u) = u - \bar{u},$ ${\Cal L}_L(\lift(u)) = u,$ ${\Frak S}_L(\bar{u}) = {\Frak S}_L(u),$ 
${\Frak S}_L(\lift(u)) = 0.$

\item[$(5) \then (1)$]
Let $g\in{\sf M}$, so that there are 
$l_i\in {\sf R},$
such that
$\sigma_1 := \sum_{i=1}^u l_i e_i\in {\sf R}^u$
satisfies 
$g = {\Frak S}_L(\sigma_1) = \sum_{i=1}^u l_i\star g_i.$

Denoting $H := \{i : {\bf T}(l_i)\circ \tau_i {\bf e}_{l_i} = v(\sigma_1) \}$, then 
either 
\begin{itemize}
\item 
${\bf T}(g) = v(\sigma_1)$, 
so that
${\bf M}(g) = \sum_{i\in H}  {\bf M}(l_i)\ast {\bf M}(g_{i})\in
{\bf M}\{{\Bbb I}_L({\bf M}\{F\})\}$ and we are through, or
\item ${\bf T}(g) < v(\sigma_1)$, 
$0 = \sum_{i\in H} {\bf M}(l_i)\ast {\bf M}(g_{i}) = {\Frak s}_L({\Cal L}_L(\sigma_1))$ and
the ${\Cal T}^{(m)}$-homoge\-neous element 
${\Cal L}_L(\sigma_1)\in \ker({\Frak s}_L)$
has a  lifting $U := {\Cal L}_L(\sigma_1)-\sum_{i=1}^u l'_i e_i$
with $$\sum_{i=1}^u l'_i\star g_i = 
\sum_{i\in H}  {\bf M}(l_i)\star g_i \And
{\bf T}(l'_i)\circ\tau_i {\bf e}_{l_i} < v(\sigma_1),$$
so that
$g = {\Frak S}_L(\sigma_2)$ and  $v(\sigma_2) < v(\sigma_1)$
for
$$\sigma_2 := 
\sum_{i\not\in H} l_i e_i + 
\sum_{i\in H}  \left(l_i-{\bf M}(l_i)\right) e_i +
\sum_{i=1}^u l'_i e_i
\in{\sf R}^u$$
\end{itemize}
and the claim follows by the well-orderedness of $<$.
\end{description}\end{proof}

\begin{Theorem}[Janet---Schreyer] \cite{J,Sch1,Sch2} \

With the same notation the equivalent conditions (1-5)  imply that 
\begin{enumerate}\setcounter{enumi}{5}
\item $\{\lift(\sigma) : \sigma\in\GM(F)\}$ is a left standard basis of $\ker({\Frak S}_L)$. 
\end{enumerate}
\end{Theorem}
 
\begin{proof}  
\begin{description}

\item[$(4) \then (6)$] Let 
$\sigma_1 := \sum_{i=1}^u l_i e_i\in\ker({\Frak S}_L)\subset{\sf R}^u.$ 

Denoting
$H := \{i : {\bf T}(l_i)\circ \tau_i {\bf e}_{l_i} = v(\sigma_1) \}$, we have
$${\Cal L}_L(\sigma_1) = \sum_{i\in H}  {\bf M}(l_i) e_i\in\ker({\Frak s}_L)$$
and there is a ${\Cal T}^{(m)}$-homogeneous representation
$${\Cal L}_L(\sigma_1) = \sum_{\sigma\in\GM(F)} c_\sigma \lambda_\sigma\ast \sigma,\, 
\lambda_\sigma\circ v(\sigma)=v(\sigma_1),
c_\sigma\in R, \lambda_\sigma\in{\Cal T}.$$
Then

\begin{eqnarray*}
\sigma_2 &:=& \sigma_1 -  \sum_{\sigma\in\GM(F)}c_\sigma
\lambda_\sigma\star \lift(\sigma)
\\ &=&
\sigma_1 -  \sum_{\sigma\in\GM(F)}c_\sigma
\lambda_\sigma\star (\sigma- \bar{\sigma})
\\ &=&
\left(\sigma_1-{\Cal L}_L(\sigma_1)\right)
+ \sum_{\sigma\in\GM(F)}c_\sigma
\lambda_\sigma\star  \bar{\sigma}
\\ &=&
\sum_{i\in H} \left(\left(l_i-{\bf M}(l_i)\right)+ \sum_{\sigma\in\GM(F)} \lambda_\sigma\star l_{i\sigma}\right) e_i
+
\sum_{i\notin H} \left(l_i+ \sum_{\sigma\in\GM(F)} \lambda_\sigma\star l_{i\sigma}\right) e_i
\end{eqnarray*}
satisfies both $\sigma_2\in\ker({\Frak S}_L)$ and
$v(\sigma_2) < v(\sigma_1)$;
thus the claim follows by induction.
\end{description}\end{proof}

\begin{Example}\label{GOeX}
Let us consider the ring of Example~\ref{GOeY} and three elements $f_1,f_2,f_3\in{\sf R}$ with 
$${\bf M}(f_1) = (5x-1)Y_1Y_2^2Y_3^2,
{\bf M}(f_2) =(5x-1)Y_1^2Y_2Y_3^2,
{\bf M}(f_3) =(5x-1)Y_1^2Y_2^2Y_3.$$

Under the natural  ${\Cal T}$-pseudovaluation on ${\sf R}^3$,
an element
\begin{equation}\label{SyzEq2}
\sigma:=\left(\alpha Y_1^{\alpha_1}Y_2^{\alpha_2}Y_3^{\alpha_3},
\beta Y_1^{\beta_1}Y_2^{\beta_2}Y_3^{\beta_3},
\gamma Y_1^{\gamma_1}Y_2^{\gamma_2}Y_3^{\gamma_3}\right)\in{\sf S}^3
\end{equation}
is homogeneous of ${\Cal T}$-degree $Y_1^{a+2}Y_2^{b+2}Y_3^{c+2}$ iff
$$\alpha_1 -1 = \beta_1 = \gamma_1 =: a,
\alpha_2 = \beta_2 -1 = \gamma_2 =: b,
\alpha_3 = \beta_3 = \gamma_3 -1 =: c.$$

Let us now specialize ourselves to the case $a=b=c=0$ and consider the ${\Bbb Z}$-module of the homogeneous syzygies  of ${\Cal T}$-degree $Y_1^{2}Y_2^{2}Y_3^{2}$; (\ref{SyzEq2}) 
is a syzygy in  $\ker({\Frak s}_L)$ iff
\begin{eqnarray*}
0&=&{\Frak s}_L(\sigma)
\\ &=&
\alpha Y_1\ast {\bf M}(f_1)+
\beta Y_2\ast {\bf M}(f_2)+
\gamma Y_3\ast {\bf M}(f_3)
\\ &=&
\left(\alpha (y^2-1)+\beta(y^3-1)+\gamma(y^4-1)\right)
Y_1^{2}Y_2^{2}Y_3^{2}.
\end{eqnarray*}

A minimal Gebauer-M\"oller set consists of
$$\sigma_1 := (-(y^2+y+1)Y_1,(y+1)Y_2,0) \And \sigma_2 := (-(y^2+1)Y_1,0,Y_3).$$

In fact  a generic syzygy (\ref{SyzEq2}) 
satisfies
$$\alpha(y+1)+\beta(y^2+y+1)+\gamma(y^2+1)(y+1)=0$$
so that 
$(y+1)\mid\beta$ and setting $\beta=(y+1)\delta$ we have
$\alpha=-\delta(y^2+y+1)-\gamma(y^2+1)$ whence
$$\sigma := \left(\left(-\delta(y^2+y+1)-\gamma(y^2+1)
\right)Y_1,
(y+1)\delta Y_2,\gamma Y_3\right)
=\delta\sigma_1+\gamma\sigma_2.$$ 
\end{Example}

\begin{Remark} \label{GOeR}
We can consider also the 
homogeneous syzygy of ${\Cal T}$-degree $Y_1^{2}Y_2^{2}Y_3^{2}$
$$\sigma_3 := (0,-(y^2+1)(y+1)Y_2,(y^2+y+1)Y_3) -(y^2+1)\sigma_1+(y^2+y+1)\sigma_2.$$

Moreover, since
$$1=(y^2+y+1)-y(y+1)=(y^3+y^2+y+1)-y(y^2+y+1)$$
setting 
$$\varsigma_A := (yY_1,Y_2,0)\in{\sf S}^3,\varsigma_B := (0,-yY_2,Y_3)\in{\sf S}^3$$
we have 
$${\Frak s}_L(\varsigma_A)={\Frak s}_L(\varsigma_B)=(y-1)Y_1^{2}Y_2^{2}Y_3^{2};$$
note that
$$\varsigma_A-\varsigma_B := (yY_1,(y+1)Y_2,Y_3)=\sigma_1+\sigma_2\in\ker({\Frak s}_L).$$
\qed\end{Remark}
\setcounter{Theorem}{34}
\begin{Example}[cont.]
Setting now $\tau:=Y_1^{a}Y_2^{b}Y_3^{c}$ and $z:=y^{2^{a}3^{b}4^{c}}$, for the syzygy (\ref{SyzEq2}) we have

\begin{eqnarray*}
0&=&{\Frak s}_L(\sigma) 
\\ &=&
\alpha \tau Y_1\ast {\bf M}(f_1)+
\beta \tau Y_2\ast {\bf M}(f_2)+
\gamma \tau Y_3\ast {\bf M}(f_3)
\\ &=&
\left(\alpha\tau\ast (y^2-1)+\beta\tau\ast(y^3-1)+\gamma\tau\ast(y^4-1)\right)
Y_1^{2}Y_2^{2}Y_3^{2}.
\\ &=&
\left(\alpha(z^2-1)+\beta(z^3-1)+\gamma(z^4-1)\right)
Y_1^{2}Y_2^{2}Y_3^{2}\tau
\end{eqnarray*}
whence
$$\alpha=-\delta(z^2+z+1)-\gamma(z^2+1),\quad \beta=(y+1)\delta$$
and
$$\sigma := \beta\tau\ast\sigma_1+\gamma\tau\ast\sigma_2.$$ 

Thus, $\{\sigma_1,\sigma_2\}$ is a minimal basis of $\ker({\Frak s}_L)$.
\end{Example}

\subsection{Bilateral case}\label{46S5B} 

Considering ${\sf R}$ as a left $R$-module, the adaptation of M\"oller lifting theorem to the bilateral case requires a few elementary adaptations; given a finite set
$$F := \{g_1,\ldots,g_u\}\subset {\sf R}^m, \,
g_i = {\bf M}(g_i)-p_i =: c_i \tau_i {\bf e}_{\iota_i} - p_i,$$
and the bilateral module ${\sf M} := {\Bbb I}_2(F)$, 
denote 
$$\hat{R} := \{a\in R: ah = a\star h = h\star a, \Forall h\in{\sf R}\}$$ 
the commutative subring $\hat{R}\subset R$ of $R$ consisting of the elements belonging   to the  center of ${\sf R}$ and remark that \index{center}
the subring of $R$ generated by ${\bf 1}_R$ is a subring of $\hat{R}$ and that
$\hat{R}$ is also a subring of the  center of the associated  graded Ore extension ${\sf S}$ of ${\sf R}$. 
 
Considering both the ${\sf R}$-bimodule
${\sf R}\otimes_{\hat{R}}{\sf R}^{\op}$ and the ${\sf S}$-bimodule
${\sf S}\otimes_{\hat{R}}{\sf S}^{\op}$, which, as sets, coincide, 
we impose on the bilateral ${\sf R}$-module $\left({\sf R}\otimes_{\hat{R}}{\sf R}^{\op}\right)^u$, whose canonical basis is denoted
$\{e_1,\ldots,e_u\}$ and whose generic element has the shape
$$\sum_i a_i \lambda_i e_{\ell_i} b_i \rho_i, \, a_i,b_i\in R\setminus\{0\},
\lambda_i,\rho_i\in{\Cal T}, 1\leq \ell_i \leq u,$$
the ${\Cal T}^{(m)}$-graded structure given by 
the valuation
$v :\left({\sf R}\otimes_{\hat{R}}{\sf R}^{\op}\right)^u \to {\Cal T}$
 as
$$v(\sigma) := 
 \max_<\{{\bf T}(\lambda_i\star  g_{\ell_i} \star \rho_i )\} =
\max_<\{\lambda_i\ast {\bf T}(g_{\ell_i} )\ast \rho_i\} = \max_<\{\lambda_i\circ \tau_{\ell_i} \circ \rho_i{\bf e}_{\iota_i}\}
 =: \delta\epsilon$$ 
for each
$$\sigma  :=\sum_i a_i \lambda_i e_{\ell_i} b_i \rho_i\in\left({\sf R}\otimes_{\hat{R}}{\sf R}^{\op}\right)^u\setminus\{0\}
$$ 
so that
$$G\left(\left({\sf R}\otimes_{\hat{R}}{\sf R}^{\op}\right)^u\right) =
\left(G\left({\sf R}\otimes_{\hat{R}}{\sf R}^{\op}\right)\right)^u=
\left({\sf S}\otimes_{\hat{R}}{\sf S}^{\op}\right)^u$$
and its corresponding
${\Cal T}^{(m)}$-homogeneous {\em leading form}\index{leading!form}\index{form!leading} 
is
$${\Cal L}_2(\sigma) := \sum_{h\in H} a_h \lambda_h e_{\ell_h}b_h\rho_h \in \left({\sf S}\otimes_{\hat{R}}{\sf S}^{\op}\right)^u$$
where
$H := \{ j :\lambda_j\circ \tau_j\circ\rho_j{\bf e}_{\iota_{\ell_j}} = v(\sigma) = \delta\epsilon\};$ we also denote, for each set $S\subset \left({\sf R}\otimes_{\hat{R}}{\sf R}^{\op}\right)^u$,
$${\Cal L}_2\{S\} := \{{\Cal L}_2(g) : g\in S\}\subset \left({\sf S}\otimes_{\hat{R}}{\sf S}^{\op}\right)^u.$$ 

We can therefore consider the morphisms
\begin{eqnarray*}
{\Frak s}_2 :  \left({\sf S}\otimes_{\hat{R}}{\sf S}^{\op}\right)^u \to {\sf S}^m
&:& {\Frak s}_2\left(\sum_i a_i \lambda_i e_{\ell_i} b_i   \rho_i\right) :=  
\sum_i a_i \lambda_i \ast {\bf M}(g_{\ell_i})\ast  b_i \rho_i,
\\
{\Frak S}_2 :  \left({\sf R}\otimes_{\hat{R}}{\sf R}^{\op}\right)^u \to {\sf R}^m
&:& {\Frak S}_2\left(\sum_i a_i \lambda_i e_{\ell_i} b_i  \rho_i\right) := 
\sum_i a_i \lambda_i \star g_{\ell_i}  \star  b_i \rho_i.
\end{eqnarray*}
\begin{Definition}
With the notation above 
\begin{itemize}
\item for a bilateral  ${\sf R}$-module ${\sf N}$,
a set $F\subset{\sf N}$ is called a bilateral  {\em standard basis}\index{standard!basis} 
if  
$${\Bbb I}_2({\Cal L}_2\{F\}) = {\Bbb I}_2({\Cal L}_2\{{\sf N}\});$$
\item for each $h\in {\sf N}$  a representation 
$$h =\sum_i a_i \lambda_i\star   g_{\ell_i}\star b_i \rho_i :
a_i,b_i\in R\setminus\{0\},
\lambda_i,\rho_i\in{\Cal T}, g_{\ell_i}\in F,$$ is called a 
 {\em standard representation}\index{standard!representation} in ${\sf R}$ in terms
of $F$ iff $$v(h)\geq v(\lambda_i\star g_{\ell_i}\star  \rho_i) = \lambda_i \circ v(g_{\ell_i})\circ  \rho_i, \mbox{\rm\ for each } i;$$ 
\item if $u\in \ker({\Frak s}_2)$ is ${\Cal T}^{(m)}$-homogeneous
and $U\in \ker({\Frak S}_2)$ is such that $u = {\Cal L}_2(U)$, we say that 
$u$ {\em lifts} to $U$, or $U$ is a {\em lifting} of $u$, or simply $u$ {\em has
a lifting};\index{lift}
\item a bilateral  {\em Gebauer--M\"oller set}\index{Gebauer--M\"oller!set} for $F$ is any 
${\Cal T}^{(m)}$-homogeneous basis of
$\ker({\Frak s}_2)$;
\item for each ${\Cal T}^{(m)}$-homogeneous element 
$\sigma\in \left({\sf R}\otimes_{\hat{R}}{\sf R}^{\op}\right)^u$,
we say that ${\Frak S}_2(\sigma)$  has a bilateral 
 {\em quasi-Gr\"obner representation}\index{Gr\"obner!representation!quasi}
 in terms of $G$ if it can be written as
$${\Frak S}_2(\sigma) = 
\sum_i a_i \lambda_i\star  g_{\ell_i}\star b_i \rho_i : 
a_i,b_i \in R\setminus\{0\}, \lambda_i,\rho_i\in{\Cal T}, g_{\ell_i}\in F$$
with 
$ \lambda_i\circ {\bf T}(g_{\ell_i})\circ  \rho_i < v(\sigma)$ for each $i$.
\qed\end{itemize}\end{Definition}

\begin{Theorem}[M\"oller--Pritchard]\label{50BiLiTh} 
With the present notation and denoting 
$\GM(F)$   any  bilateral Gebauer--M\"oller set for $F$, the following
conditions are equivalent: 
\begin{enumerate}
\item $F$ is a bilateral Gr\"obner basis of ${\sf M}$;
\item $f \in {\sf M}\iff f$  has a bilateral  Gr\"obner representation in terms of $F$;
\item for each $\sigma\in\GM(F)$, 
the bilateral S-polynomial ${\Frak S}_2(\sigma)$ has a bilateral quasi-Gr\"obner representation
${\Frak S}_2(\sigma) = \sum_{l=1}^\mu a_l \lambda_l \star g_{\ell_l}\star  b_l  \rho_l, $  in terms of $F$;
\item each $\sigma\in\GM(F)$ has a lifting $\lift(\sigma)$;
\item each ${\Cal T}^{(m)}$-homogeneous element 
$u\in\ker({\Frak s}_2)$ has a lifting $\lift(u)$.
\end{enumerate}
\end{Theorem}

\begin{proof} \
\begin{description}
\item[$(1) \then (2)$] Let $f\in{\sf M}$; by assumption 
$${\bf M}(f) =
\sum_{i=1}^\mu a_i \lambda_i \ast{\bf M}(g_{\ell_i})\ast b_i\rho_i$$ where 
$\sum_{i=1}^\mu a_i \lambda_i e_{\ell_i}b_i \rho_i\in\left({\sf S}\otimes_{\hat{R}}{\sf S}^{\op}\right)^u$ is ${\Cal T}^{(m)}$-homogeneous  of ${\Cal T}^{(m)}$-degree ${\bf T}(f)$.

Therefore $g := f - \sum_{i=1}^\mu  a_i \lambda_i\star g_{\ell_i}\star b_i\rho_i \in {\sf M}$ and
${\bf T}(g) < {\bf T}(f)$.

Thus, the claim follows by induction since $<$ is a well-ordering.
\item[$(2) \then (3)$] ${\Frak S}_2(\sigma)\in{\sf M}$ and ${\bf T}({\Frak S}_2(\sigma)) < v(\sigma)$.
\item[$(3) \then (4)$] Let 
$${\Frak S}_2(\sigma) = 
\sum_{i=1}^\mu a_i \lambda_i \star g_{\ell_i}\star  b_i\rho_i,
v(\sigma)> \lambda_i\circ\tau_{\ell_i}\circ \rho_i {\bf e}_{\iota_{\ell_i}}$$
be a bilateral
quasi-Gr\"obner representation in terms of $F$; then $$\lift(\sigma) := \sigma -
\sum_{i=1}^\mu a_i \lambda_i e_{\ell_i}b_i   \rho_i$$ is the required lifting of $\sigma$. 
\item[$(4) \then (5)$] Let 
$u := \sum_i a_i \lambda_i e_{\ell_i} b_i  \rho_i\in\left({\sf S}\otimes_{\hat{R}}{\sf S}^{\op}\right)^u,\,
 \lambda_i\circ\tau_{\ell_i}\circ\rho_i  {\bf e}_{\iota_{\ell_i}}= v(u),$
be a ${\Cal T}^{(m)}$-homogeneous element in $\ker({\Frak s}_2)$ of ${\Cal T}^{(m)}$-degree $v(u)$.

Then there are $\lambda_\sigma,\rho_\sigma\in{\Cal T},a_\sigma,b_\sigma\in R\setminus\{0\},$
for which
$$u  = \sum_{\sigma\in\GM(F)} 
a_\sigma \lambda_\sigma \ast\sigma\ast  b_\sigma\rho_\sigma,
 \lambda_\sigma\circ v(\sigma)\circ\rho_\sigma = v(u).$$

For each $\sigma\in\GM(F)$ denote
{\small $$\bar{\sigma} := \sigma - \lift(\sigma) =  {\Cal L}_2(\lift(\sigma)) -
\lift(\sigma) := 
\sum_{i=1}^{\mu_\sigma} a_{i\sigma} \lambda_{i\sigma} e_{\ell_{i\sigma}}  b_{i\sigma} \rho_{i\sigma}
\in\left({\sf R}\otimes_{\hat{R}}{\sf R}^{\op}\right)^u$$}
and remark that
 $\lambda_{i\sigma}\circ\tau_{\ell_{i\sigma}} \circ\rho_{i\sigma}{\bf e}_{\iota_{\ell_{i\sigma}}}\leq v(\bar{\sigma}) < v(\sigma)$ and 
 ${\Frak S}_2(\bar{\sigma}) = {\Frak S}_2(\sigma).$
 
It is sufficient to define
 $$\lift(u)  := \sum_{\sigma\in\GM(F)} a_\sigma \lambda_\sigma\star \lift(\sigma)\star b_\sigma\rho_\sigma ,
\And 
\bar{u}  := \sum_{\sigma\in\GM(F)} a_\sigma \lambda_\sigma\star \bar{\sigma}\star b_\sigma\rho_\sigma$$
to obtain
$$\lift(u) = u - \bar{u},{\Cal L}_2(\lift(u)) = u,{\Frak S}_2(\bar{u}) = {\Frak S}_2(u), 
{\Frak S}_2(\lift(u)) = 0.$$

\item[$(5) \then (1)$] Let
$g\in{\sf M}$, so that there are 
$\lambda_i,\rho_i\in{\Cal T}, a_i,b_i\in R\setminus\{0\}, 1\leq \ell_i \leq u,$ 
such that
$\sigma_1 := \sum_{i=1}^\mu a_i \lambda_i e_{\ell_i} b_i \rho_i\in  \left({\sf R}\otimes_{\hat{R}}{\sf R}^{\op}\right)^u$
satisfies 
$$g = {\Frak S}_2(\sigma_1) =  \sum_{i=1}^\mu a_i \lambda_i \star g_{\ell_i}\star  b_i\rho_i.$$
Denoting $H := \{i : \lambda_i\circ{\bf T}(g_{\ell_i})\circ \rho_i = 
\lambda_i\circ\tau_{\ell_i}\circ \rho_i{\bf e}_{\iota_{\ell_i}} = 
v(\sigma_1) \}$, then 
either 
\begin{itemize}
\item $v(\sigma_1) = {\bf T}(g)$ so that, for each $i\in H$,
$ {\bf M}(a_i \lambda_i \star {\bf M}(g_{\ell_i})\star  b_i\rho_i) =  a_i \lambda_i\ast{\bf M}(g_{\ell_i})\ast  b_i\rho_i$
and
$${\bf M}(g) = \sum_{i\in H}  a_i \lambda_i\ast{\bf M}(g_{\ell_i})\ast  b_i\rho_i\in
{\bf M}\{{\Bbb I}_2({\bf M}\{F\})\},$$ and we are through, or
\item ${\bf T}(g) < v(\sigma_1)$, in which case 
$0 = \sum_{i\in H}  a_i \lambda_i\ast{\bf M}(g_{\ell_i})\ast b_i\rho_i = {\Frak s}_2({\Cal L}_2(\sigma_1))$
and 
the ${\Cal T}^{(m)}$-homoge\-neous element 
${\Cal L}_2(\sigma_1)\in \ker({\Frak s}_2)$
has a  lifting $$U := {\Cal L}_2(\sigma_1)-
\sum_{j=1}^\nu a_j \lambda_j e_{\ell_j}   b_j\rho_j\in \left({\sf R}\otimes_{\hat{R}}{\sf R}^{\op}\right)^u$$
with $$
\sum_{j=1}^\nu a_j \lambda_j \star g_{\ell_j}\star   b_j\rho_j = 
\sum_{i\in H} a_i \lambda_i\star g_{\ell_i}\star  b_i  \rho_i\And
\lambda_j\circ\tau_{l_j}\circ \rho_j{\bf e}_{\iota_{\ell_j}} < v(\sigma_1)$$
so that
$g = {\Frak S}_2(\sigma_2)$ and  $v(\sigma_2)< v(\sigma_1)$ holds
for
$$\sigma_2 := 
\sum_{i\not\in H} a_i \lambda_i e_{\ell_i}b_i  \rho_i + 
\sum_{j=1}^\nu a_j \lambda_j e_{\ell_j}  b_j \rho_j
\in \left({\sf R}\otimes_{\hat{R}}{\sf R}^{\op}\right)^u$$
\end{itemize}
and the claim follows by the well-orderedness of $<$.
\end{description}\end{proof}

\begin{Theorem}[Janet---Schreyer]  \

With the same notation the equivalent conditions (1-5)  imply that 
\begin{enumerate}\setcounter{enumi}{5}
\item $\{\lift(\sigma) : \sigma\in\GM(F)\}$ is a bilateral standard basis of $\ker({\Frak S}_2)$. 
\end{enumerate}
\end{Theorem}
 
\begin{proof}  
\begin{description}

\item[$(4) \then (6)$] Let 
$\sigma_1 := 
 \sum_{i=1}^\mu a_i \lambda_i e_{\ell_i} b_i  \rho_i\in\ker({\Frak S}_2)\subset \left({\sf R}\otimes_{\hat{R}}{\sf R}^{\op}\right)^u.$ 

Denoting
 $H := \{i : 
\lambda_i\circ\tau_{\ell_i}\circ \rho_i{\bf e}_{\iota_{\ell_i}} = 
v(\sigma_1) \}$, we have
$${\Cal L}_2(\sigma_1) = \sum_{i\in H}  
a_i \lambda_i e_{\ell_i}  b_i  \rho_i\in\ker({\Frak s}_2)$$
and there is a ${\Cal T}^{(m)}$-homogeneous representation
$${\Cal L}_2(\sigma_1) = \sum_{\sigma\in\GM(F)} a_\sigma \lambda_\sigma \ast\sigma\ast
b_\sigma  \rho_\sigma,   \lambda_\sigma\circ v(\sigma)\circ\rho = v(\sigma_1)$$
with $\lambda_\sigma,\rho_\sigma\in{\Cal T}, a_\sigma,b_\sigma\in R\setminus\{0\}$.

Then
\begin{eqnarray*}
\sigma_2
&:=& 
\sigma_1 -  \sum_{\sigma\in\GM(F)}
a_\sigma \lambda_\sigma \star\lift(\sigma)\star  b_\sigma\rho_\sigma
\\&=& 
\sigma_1 -  \sum_{\sigma\in\GM(F)}
a_\sigma \lambda_\sigma \star\left(\sigma-\bar\sigma\right)\star  b_\sigma \rho_\sigma
\\&=&
\sigma_1-{\Cal L}_2(\sigma_1)
+ \sum_{\sigma\in\GM(F)}
a_\sigma \lambda_\sigma \star\bar{\sigma}\star   b_\sigma\rho_\sigma
\\&=& 
\sum_{i\not\in H} a_i \lambda_i e_{\ell_i} b_i  \rho_i
\\&+&
\sum\limits_{\sigma\in\GM(F)} \sum_{i=1}^{\mu_\sigma}
\Bigl(
\left(a_\sigma\alpha_{\lambda_\sigma}(a_{i\sigma})\right)\cdot\left(\lambda_\sigma\circ\lambda_{i\sigma}\right)\Bigr)
e_{\ell_{i\sigma}}
\Bigl(
\left(b_{i\sigma}\alpha_{\rho_\sigma}(b_\sigma)\right)\cdot\left(\rho_{i\sigma}\circ\rho_{\sigma}\right)\Bigr)
\end{eqnarray*}
satisfies both $\sigma_2\in\ker({\Frak S}_2)$ and
$v(\sigma_2)< v(\sigma_1)$;
thus the claim follows by induction.
\end{description}\end{proof}

\subsection{Restricted case} 

In order to deal with restricted modules, we need simply to adapt and simplify the bilateral case.

Thus, we consider both the left $R$-modules
$R\otimes{\sf R}^{\op}$ and 
$R\otimes{\sf S}^{\op}$, which, as sets, coincide, 
we impose on the bilateral $R$-module $\left(R\otimes{\sf R}^{\op}\right)^u$, whose canonical basis is denoted
$\{e_1,\ldots,e_u\}$ and whose generic element has the shape
$$\sum_i a_i   e_{\ell_i}  \rho_i, a_i\in R\setminus\{0\},
\rho_i\in{\Cal T}, 1\leq \ell_i \leq u,$$
the ${\Cal T}^{(m)}$-graded structure given by 
the valuation
$v :\left(R\otimes{\sf R}^{\op}\right)^u \to {\Cal T}$
 as
$$v(\sigma) := 
 \max_<\{{\bf T}(g_{\ell_i} \star \rho_i )\} =
\max_<\{{\bf T}(g_{\ell_i} )\ast \rho_i\} = \max_<\{\tau_{\ell_i} \circ \rho_i{\bf e}_{\iota_i}\}
 =: \delta\epsilon$$ 
for each
$$\sigma  :=\sum_i a_i   e_{\ell_i}  \rho_i\in\left(R\otimes_{\hat{R}}{\sf R}^{\op}\right)^u\setminus\{0\}
$$ 
so that
$$G\left(\left(R\otimes{\sf R}^{\op}\right)^u\right) =
\left(G\left(R\otimes{\sf R}^{\op}\right)\right)^u=
\left(R\otimes{\sf S}^{\op}\right)^u$$
and its corresponding
${\Cal T}^{(m)}$-homogeneous {\em leading form}\index{leading!form}\index{form!leading} 
is
$${\Cal L}_W(\sigma) := \sum_{h\in H} a_h  e_{\ell_h}    \rho_h \in \left(R\otimes{\sf S}^{\op}\right)^u$$
where
$H := \{ j : \tau_j\circ\rho_j{\bf e}_{\iota_{\ell_j}} = v(\sigma) = \delta\epsilon\};$ we also denote, for each set $S\subset \left(R\otimes{\sf R}^{\op}\right)^u$,
$${\Cal L}_W\{S\} := \{{\Cal L}_W(g) : g\in S\}\subset \left(R\otimes{\sf S}^{\op}\right)^u.$$ 

We can therefore consider the morphisms
\begin{eqnarray*}
{\Frak s}_W :  \left(R\otimes{\sf S}^{\op}\right)^u \to {\sf S}^m
&:& {\Frak s}_W\left(\sum_i a_i  e_{\ell_i}    \rho_i\right) :=  
\sum_i a_i   {\bf M}(g_{\ell_i})\ast   \rho_i,
\\
{\Frak S}_W :  \left(R\otimes{\sf R}^{\op}\right)^u \to {\sf R}^m
&:& {\Frak S}_W\left(\sum_i a_i  e_{\ell_i}   \rho_i\right) := 
\sum_i a_i   g_{\ell_i}  \star   \rho_i.
\end{eqnarray*}
\begin{Definition}
With the notation above 
\begin{itemize}
\item for a restricted   module ${\sf N}$,
a set $F\subset{\sf N}$ is called a restricted  {\em standard basis}\index{standard!basis} 
if  
$${\Bbb I}_W({\Cal L}_W\{F\}) = {\Bbb I}_W({\Cal L}_W\{{\sf N}\});$$
\item for each $h\in {\sf N}$  a representation 
$$h =\sum_i a_i    g_{\ell_i}\star \rho_i :
a_i\in R\setminus\{0\},
\rho_i\in{\Cal T}, g_{\ell_i}\in F,$$ is called a 
 {\em standard representation}  in ${\sf R}$ in terms
of $F$ iff $$v(h)\geq v( g_{\ell_i}\star  \rho_i) =  v(g_{\ell_i})\circ  \rho_i, \mbox{\rm\ for each } i;$$ 
\item if $u\in \ker({\Frak s}_W)$ is ${\Cal T}^{(m)}$-homogeneous
and $U\in \ker({\Frak S}_W)$ is such that $u = {\Cal L}_W(U)$, we say that 
$u$ {\em lifts} to $U$, or $U$ is a {\em lifting} of $u$, or simply $u$ {\em has
a lifting};\index{lift}
\item a restricted  {\em Gebauer--M\"oller set}\index{Gebauer--M\"oller!set} for $F$ is any 
${\Cal T}^{(m)}$-homogeneous basis of
$\ker({\Frak s}_W)$;
\item for each ${\Cal T}^{(m)}$-homogeneous element 
$\sigma\in \left(R\otimes_{\hat{R}}{\sf R}^{\op}\right)^u$,
we say that ${\Frak S}_W(\sigma)$  has a restricted
 {\em quasi-Gr\"obner representation}\index{Gr\"obner!representation!quasi}
 in terms of $G$ if it can be written as
$${\Frak S}_W(\sigma) = 
\sum_i a_i   g_{\ell_i}\star  \rho_i : 
a_i \in R\setminus\{0\}, \rho_i\in{\Cal T}, g_{\ell_i}\in F$$
with 
$  {\bf T}(g_{\ell_i})\circ  \rho_i < v(\sigma)$ for each $i$.
\qed\end{itemize}\end{Definition}

\begin{Theorem}[M\"oller--Pritchard]\label{50ReLiTh} 
With the present notation and denoting 
$\GM(F)$   any  restricted Gebauer--M\"oller set for $F$, the following
conditions are equivalent: 
\begin{enumerate}
\item $F$ is a restricted Gr\"obner basis of ${\Bbb I}_W(F)$;
\item $f \in {\Bbb I}_W(F) \iff f$  has a restricted  Gr\"obner representation in terms of $F$;
\item for each $\sigma\in\GM(F)$, 
the restricted S-polynomial ${\Frak S}_W(\sigma)$ has a restricted quasi-Gr\"obner representation
${\Frak S}_W(\sigma) = \sum_{l=1}^\mu a_l   g_{\ell_l}\star    \rho_l, $ 
in terms of $F$;
\item each $\sigma\in\GM(F)$ has a lifting $\lift(\sigma)$;
\item each ${\Cal T}^{(m)}$-homogeneous element 
$u\in\ker({\Frak s}_W)$ has a lifting $\lift(u)$.
\end{enumerate}
\end{Theorem}

\begin{Theorem}[Janet---Schreyer]  \

With the same notation the equivalent conditions (1-5)  imply that 
\begin{enumerate}\setcounter{enumi}{5}
\item $\{\lift(\sigma) : \sigma\in\GM(F)\}$ is a restricted standard basis of $\ker({\Frak S}_W)$. 
\end{enumerate}
\end{Theorem}

\section{Gr\"obner basis Computation for Multivariate Ore Extensions of Zacharias Domains}

We recall the definition of Zacharias ring  \cite{Z},  \cite[\S 26.1]{SPES}, \cite{Zac}.

\begin{Definition}   \label{c46D3}
\renewcommand\theenumi{{\rm (\alph{enumi})}}
 A ring $R$ with identity is called 
 a (left) {\em Zacharias ring}\index{Zacharias!ring}\index{ring!Zacharias}
if it satisfies the following properties:
\begin{enumerate}
\item R is a noetherian ring;
\item  there is an algorithm which, for each 
$c \in R^m, \, C := \{c_1,\ldots c_t\} \subset R^m\setminus\{0\},$
 allows to decide whether 
$c\in {\Bbb I}_L(C)$  in which case it produces elements 
$d_i\in R :  c = \sum_{i=1}^t d_i c_i;$
\item there is an algorithm  which, given
$\{c_1,\ldots c_t\} \subset R^m\setminus\{0\},$ computes a finite set of
generators for 
the left syzygy R-module  $\left\{(d_1,\cdots,d_t) \in R^t :
\sum_{i=1}^t d_i c_i = 0\right\}$.
\end{enumerate}

Note that
 \cite{M}  for a  ring $R$ with identity which satisfies (a) and (b), (c) is equivalent to
\begin{enumerate}
\setcounter{enumi}{3}
\item there is an algorithm  which, given
$\{c_1,\ldots c_s\} \subset R^m\setminus\{0\},$ computes a finite basis of the
ideal
$${\Bbb I}_L(\{c_i : 1\leq i < s\}) : {\Bbb I}_L(c_s).$$
\end{enumerate}

If $R$ has canonical representatives,
we improve the computational assumptions
of Zacharias rings, requiring also  the following property:
\begin{enumerate}\setcounter{enumi}{4}
\renewcommand\theenumi{{\rm (\alph{enumi})}}
\item there is an algorithm  which, given an element $c\in R^m$ and a 
 left
module ${\sf J}\subset R^m$, computes the unique canonical representative
 $\Rep(c,{\sf J})$.
 \qed\end{enumerate}
\end{Definition}

If  $R$ is a left Zacharias domain, the three algorithms proposed by M\"oller  \cite{M} for computing Gr\"obner bases in the polynomial ring over $R$ can be easily adapted to
 multivariate Ore extensions of Zacharias domains, provided that
each $\alpha_i$, and therefore each $\alpha_\tau$, is an automorphism.

\subsection{First algorithm}\label{1Alg}         
 
 Still considering a  finite basis
$$F := \{g_1,\ldots,g_u\}\subset {\sf R}^m, \,
g_i = {\bf M}(g_i)-p_i =: c_i \tau_i {\bf e}_{l_i} - p_i,$$
of the module ${\sf M} := {\Bbb I}_L(F)$
and denoting
\begin{itemize} 
\renewcommand\labelitemi{\bf --}
\item ${\Frak H}(F): = \{\{i_1,i_2,\ldots,i_r\}\subseteq\{1,\ldots,u\} : l_{i_1} = \cdots  = l_{i_r}\};$
\item for each $H := \{i_1,i_2,\ldots,i_r\}\in{\Frak H}(F)$,
\begin{itemize}
\item $\varepsilon_H :={\bf e}_{l_{i_1}} = \cdots 
={\bf e}_{l_{i_1}},$ 
\item $\tau_H := \lcm\left(\tau_i : i\in H\right),$
\item  for each $I\subset H$,
\begin{itemize}
\renewcommand\labelitemiii{\bf --}
\item $\tau_{H,I}  := \frac{\tau_{H}}{\tau_I}$,
\item $\alpha_{H,I} : R\to  R$ the morphism 
$\alpha_{\tau_{H,I}}$;
\end{itemize}
\item ${\bf T}(H) := \tau_H\varepsilon_H,$
\end{itemize}
and, if $R$ is a PID,
\begin{itemize}
\item $c_H := \lcm(\alpha_{H,i}(c_i) :i\in H)$,
\item $\mu(H) := c_H\tau_H$ and
\item ${\bf M}(H) = c_H{\bf T}(H) =c_H\tau_H\varepsilon_H = \mu(H)\varepsilon_H$;
\end{itemize}
\item ${\sf T} := \{{\bf T}(H):   H\in{\Frak H}(F)\}$;
\item for any ${\sf m} =\delta\epsilon\in {\sf T}$, 
\begin{itemize} 
\item for each $i, 1 \leq i \leq u, t_i({\sf m}) :=
\begin{cases}\frac{\delta}{ \tau_i} &\mbox{  if $ {\bf T}(g_i) \mid {\sf m}$,}\cr
1&\mbox{   otherwise;}\cr\end{cases}$
\item $v({\sf m}) = (v({\sf m})_1,\ldots,v({\sf m})_u)\in R^u$ the vector such
that  
$$v({\sf m})_i :=  \begin{cases}\alpha_{t_i({\sf m})}(\lc(g_i)) &\mbox{  if $ {\bf T}(g_i) \mid {\sf m}$,}\cr
0&\mbox{   otherwise;}\cr\end{cases}$$
\item $C({\sf m}) \subset R^u$ a finite basis of the syzygy module
$$\Syz_L\left(v({\sf m})_1,\ldots,v({\sf m})_u\right) := \left\{(d_1,\ldots,d_u)\in R^u : 
\sum_{i=1}^u d_i  v({\sf m})_i = 0\right\};$$ 
\item $S({\sf m}) := \{(d_1 t_1({\sf m}), \ldots, d_u
t_u({\sf m})) : (d_1,\ldots,d_u)\in C({\sf m})\};$
\end{itemize}
\item ${\Cal S}(F) := \bigcup_{{\sf m}\in {\sf T}} S({\sf m});$
\item ${\Cal S}'(F)\subset{\Cal S}(F)$ any subset satisfying
\begin{itemize}
\renewcommand\labelitemi{\bf --}
\item for each $\sigma\in {\Cal S}(F)\setminus{\Cal S}'(F)$
exist  $\sigma_j\in{\Cal S}'(F), d_j\in R, \tau_j\in{\Cal T},$ 
such that
$\sigma = \sum_j d_j\tau_j\ast\sigma_j$;
\end{itemize}
\item ${\Cal R}(F) := \left\{\sum_i m_i\star g_i : (m_1,\ldots,m_u)\in
{\Cal S}'(F)\right\},$  
\end{itemize}
we have that ({\em cf.} \cite{M}, \cite[Theorem~26.1.4]{SPES})

\begin{Lemma}\label{47L1}  ${\Cal S}(F)$ is a Gebauer--M\"oller set for $F$.
\end{Lemma}

\begin{proof}
Let us consider a generic ${\Cal T}^{(m)}$-homogeneus element
$$\sigma :=  \sum_{i=1}^u a_i \lambda_i e_i\in{\sf R}^u\setminus\{0\},$$
with  
$a_i\in R, \lambda_i\in{\Cal T}, v(\sigma) := \tau\epsilon$,  and
$a_i\neq 0 \then  \lambda_i \tau_i=\tau, \epsilon = {\bf e}_{l_i}$,
and assume that it is a left syzygy in $\ker({\Frak s}_L)$.

Denoting $I:= \{i\leq u : a_i\neq 0\}$ and  setting 
${\sf m} :=\delta\epsilon := \lcm\{  {\bf T}(g_i) : i\in H\}\mid v(\sigma)$,
there is $\upsilon\in {\Cal T} : \upsilon\delta = \tau$.
With  the present notation we also have
$\delta = t_i({\sf m})\tau_i$;
thus $\upsilon t_i({\sf m})\tau_i = \upsilon\delta = \lambda_i \tau_i$
and
$\lambda_i = \upsilon t_i({\sf m})$.
We also  have
$$0 = \sum_{i=1}^u a_i \alpha_{\lambda_i}(\lc(g_i))$$
so that
$$0=\alpha_\upsilon^{-1}\left(\sum_{i=1}^u a_i \alpha_{\lambda_i}(\lc(g_i))\right) =
\sum_{i=1}^u \alpha_\upsilon^{-1}(a_i) \alpha_{t_i({\sf m})}(\lc(g_i)),$$
so that
$(\alpha_\upsilon^{-1}(a_1),\ldots, \alpha_\upsilon^{-1}(a_u))\in\Syz_L(v({\sf m})_1,\ldots,v({\sf m})_u)$.

Therefore, if we enumerate as
$$(d_{11},\ldots,d_{1u}),\ldots, (d_{v1},\ldots,d_{vu})$$ a basis of $C({\sf m})$ 
and we denote ${\sf s}_j :=\sum_{i=1}^u d_{ji} t_i({\sf m}) e_i, 1\leq j \leq v$, the elements of $S({\sf m})$, 
we have
$(\alpha_\upsilon^{-1}(a_1),\ldots, \alpha_\upsilon^{-1}(a_u)) = \sum_{j=1}^v b_j (d_{j1},\ldots,d_{ju})$ for suitable $b_j\in R$
and
\begin{eqnarray*}
\sigma &=&  \sum_{i=1}^u a_i \lambda_i e_i 
\\  &=& 
\sum_{i=1}^u a_i \upsilon t_i({\sf m}) e_i
\\  &=& 
\sum_{i=1}^u \upsilon\ast\alpha_\upsilon^{-1}(a_i)  t_i({\sf m}) e_i
\\  &=& 
 \upsilon\ast\sum_{i=1}^u \sum_{j=1}^v b_jd_{ji} t_i({\sf m}) e_i
\\  &=& 
\sum_{j=1}^v\alpha_\upsilon(b_j) \upsilon \ast  \left(\sum_{i=1}^u d_{ji} t_i({\sf m})e_i\right) 
\\  &=& 
\sum_{j=1}^v\alpha_\upsilon(b_j) \upsilon\ast{\sf s}_j.
\end{eqnarray*}
\end{proof}

\begin{Corollary} The following holds:
\begin{enumerate}
\item ${\Cal S}'(F)$ is a Gebauer--M\"oller set for $F$.
\item $F$ is a left Gr\"obner basis of the module it generates iff 
each $h\in{\Cal R}(F)$ has a left Gr\"obner representation in terms of $F$.
\qed\end{enumerate}\end{Corollary}

\begin{Example} 
If we consider the ring of Example~\ref{GOeY}  as a left ${\Bbb Z}[x]$-module endowed 
with the $\Gamma$-pseudovaluation, 
$\Gamma =  \{Y_1^{a_1}Y_2^{a_2}Y_3^{a_3} : (a_1,a_2,a_1)\in{\Bbb N}^3\},$
we obtain a  similar solution as the one described  in Example~\ref{GOeX} .

Expressing each ${\bf M}(f_i)$ as ${\bf M}(f_i)=\lc(f_i){\bf T}(f_i)$, according Zacharias approach we need to compute a syzygy bases in ${\Bbb Z}[x]$ among
$\alpha_{Y_1}(\lc(f_1)) = (y^2-1)$, $\alpha_{Y_2}(\lc(f_2)) = (y^3-1)$ and
$\alpha_{Y_3}(\lc(f_3))  = (y^4-1)$; the natural solutions  
$(-(y^2+y+1),(y+1),0)$, $(-(y^2-1),0,1)$
produce $\sigma_1$ and  $\sigma_2$.
\end{Example}
\begin{Example}\label{Azzurro}
Let us now specialize the ring of Example~\ref{GOe}
to the case
$$n=3, e_1=e_2=e_3=1,c_1=20,c_2=6,c_3=15,$$
and let us consider four elements $f_1,f_2,f_3,f_4\in{\sf R}$ with 
$${\bf M}(f_1) = xY_1Y_2^3Y_3^2,
{\bf M}(f_2) =x^2Y_1^2Y_2Y_3^2,
{\bf M}(f_3) = xY_1^2Y_2^3Y_3,
{\bf M}(f_4) =xY_1^2Y_2^2Y_3^2.$$

We have
$${\sf T} =\{Y_1Y_2^3Y_3^2,
Y_1^2Y_2Y_3^2,
Y_1^2Y_2^2Y_3^2,
Y_1^2Y_2^3Y_3,
Y_1^2Y_2^3Y_3^2\},$$
and
$$\begin{array}{lll}
{\sf m}&(t_1({\sf m}),\ldots,t_4({\sf m}))&v({\sf m})\\
\hline
Y_1Y_2^3Y_3^2&(1,0,0,0)&(x,0,0,0) \\
Y_1^2Y_2Y_3^2&(0,1,0,0)&(0,x^2,0,0)\\
Y_1^2Y_2^2Y_3^2&(0,Y_2,0,1)&(0,6^2x^2,0,x)\\
Y_1^2Y_2^3Y_3&(0,0,1,0)&(0,0,x,0)\\
Y_1^2Y_2^3Y_3^2&(Y_1,Y_2^2,Y_3,Y_2)&(20x,6^4x^2,15x,6x),\\
\end{array}$$
Denoting
\begin{eqnarray*}
b(1,3)&:=&(-3Y_1,0,4Y_3,0),\\
b(2,4)&:&(0,-Y_2,0,6^2x),\\
b(3,4)&:=&(0,0,-2Y_3,0,5Y_2) 
\end{eqnarray*}
we have $S(Y_1^2Y_2^2Y_3^2) =\{b(2,4)\}$
and, since $$Y_2\ast b(2,4) =(0,-Y_2^2,0,6^2Y_2\ast x )=(0,-Y_2^2,0,6^3xY_2)$$
 we can take
$S(Y_1^2Y_3^2Y_3^2) =\{b(1,3),Y_2\ast b(2,4),b(3,4)\}$;
thus $${\Cal S}' := \{b(1,3), b(2,4),b(3,4)\}$$  is the required Gebauer--M\"oller set.
\qed\end{Example}

\subsection{Second algorithm}\label{2Alg} 

M\"oller proposes an (essentially) equivalent alternative computation:
for any $s, 1\leq s \leq u$, let us consider the 
syzygy module 
$${\Cal S}_s := \left\{(h_1,\ldots,h_s) : \sum_{i=1}^s h_i\star{\bf M}(g_i) = 0\right\}\subset
{\sf R}^s$$
and let us compute  ${\Cal S}(F) = {\Cal S}_u$ by inductively
extending ${\Cal S}_{s-1}$ to  ${\Cal S}_{s}$, the inductive seed being
${\Cal S}_{1}=\emptyset$.

A direct application of the property (d)  of  a Zacharias ring allows to compute a Gebauer-M\"oller set via
\begin{Definition}\label{c46D4} A subset $H\subset\{1,\ldots,s\}\cap{\Frak H}(F), s \leq u$, is
said to be 
\begin{description}
\item[] {\em maximal} for a term $\delta\epsilon\in{\Cal T}^{(m)}$ if 
$H = \{i, 1\leq i\leq s : \tau_i\mid\delta, 
{\bf e}_{l_i} = \epsilon\}$,
\item[] {\em basic} if $s\in H$ and $H$ is maximal for ${\bf T}(H)$.
\end{description}

For a basic subset $H\subset\{1,\ldots,s\}\cap{\Frak H}(F)$, denote $H^{\times} :=
H\setminus\{s\}$.

For any 
$$d_s\in{\Bbb I}_L(\{\alpha_{H,i}(c_i) : i\in H^{\times}\}) : {\Bbb I}_L(\alpha_{H,s}(c_s)),$$
a {\em syzygy associated to $H$ and $d_s$} is any
${\Cal T}^{(m)}$-homogeneous syzygy
$$\sum_{i\in H^{\times}} d_i \frac{\tau_H}{\tau_i} e_i 
+ d_s \frac{ \tau_H}{\tau_s} e_s\in{\Cal S}_s$$
where $d_i\in R$ are suitable elements for which 
$d_s\alpha_{H,s}(c_s) = - \sum_{i\in H^{\times}} d_i \alpha_{H,i}(c_i).$
\qed\end{Definition}

\begin{Theorem}[M\"oller]\label{48Tx1} \cite{M}
 With the present notation, denoting
\begin{itemize}
\item $\{A_1,\ldots,A_\mu\}$ a ${\Cal T}^{(m)}$-homogeneous basis of ${\Cal S}_{s-1},$ 
\item ${\Cal H}$ the set of all basic subsets $H\subset\{1,\ldots,s\}\cap{\Frak H}(F),$ 
\item $\{d_{1H},\ldots,d_{r_HH}\}$ a basis of the ideal ${\Bbb I}_L(\{\alpha_{H,i}(c_i) \textrm{ s.t. } i\in H^{\times}\}) : {\Bbb I}_L(\alpha_{H,s}(c_s))$ for each basic subset $H\in{\Cal H},$ 
\item $D_{jH}\in{\sf R}^s$ a syzygy associated to $H$ and $d_{jH}$, for each  basic subset $H\in{\Cal H}$ and each $j, 1 \leq j \leq r_H$ 
\end{itemize}
the set 
$\{A_1,\ldots,A_\mu\} \cup
 \{D_{jH} : H\in{\Cal H},1 \leq j \leq r_H\}$
is a  ${\Cal T}^{(m)}$-homogeneous basis of ${\Cal S}_{s}$.
\end{Theorem}

\begin{proof} Let 
$S := (d_1\lambda_1,\ldots,d_s\lambda_s)\in{\Cal S}_{s}, d_s\neq 0$,
be a 
 ${\Cal T}^{(m)}$-homogeneous element of ${\Cal T}^{(m)}$-degree $\delta\epsilon$
and let 
$$K := \{i, 1\leq i \leq s : d_i \neq 0\};$$ since by ${\Cal T}^{(m)}$-homogeneity, 
$\tau_i\mid\delta$ and ${\bf e}_{\iota_i} = \epsilon$ for each $i\in K$, we have ${\bf T}(K)\mid\delta\epsilon$; 
we also have 
$d_i = 0  \Forall i\notin K
\And \lambda_i\tau_i = \delta, {\bf e}_{\iota_i} = \epsilon \Forall i\in K.$

For the set $H := \{i, 1\leq i \leq s : 
\tau_i  \mid \tau_K, {\bf e}_{\iota_i} = \varepsilon_K\}$
clearly we have 
$\tau_H  \mid \tau_K$ and $K \subseteq H$ so that
$\tau_H  \mid \tau_K \mid \delta$; we also have $\varepsilon_H = \varepsilon_K = \epsilon.$
Moreover $d_s\neq 0$ implies $s\in K \subseteq H$ so that $H$ is basic.
Since $(d_1\lambda_1,\ldots,d_s\lambda_s)\in{\Cal S}_{s}$, setting $\upsilon :=  
\frac{\delta}{\tau_H}$,
we have
$$0 = \sum_{i=1}^s d_i\lambda_i\ast {\bf M}(g_i) = 
\sum_{i\in H} d_i\frac{\delta}{\tau_i}\ast c_i \tau_i\epsilon
=\left(\sum_{i\in H} d_i\alpha_{\lambda_i}(c_i)\right)\delta\epsilon$$
so that $\sum_{i\in H} d_i\alpha_{\lambda_i}(c_i) = 0$,
$\sum_{i\in H} \alpha_{\upsilon}^{-1}(d_i)\alpha_{H,i}(c_i) = 0$, whence 
$$ \alpha_{\upsilon}^{-1}(d_s)\alpha_{H,s}(c_s)\in{\Bbb I}_L(\alpha_{H,i}(c_i) : i\in H^{\times}) \And
 \alpha_{\upsilon}^{-1}(d_s)\in{\Bbb I}_L(\alpha_{H,i}(c_i) : i\in H^{\times}) : {\Bbb I}_L(\alpha_{H,s}(c_s)).$$

Therefore $ \alpha_{\upsilon}^{-1}(d_s) =\sum_{j=1}^{r_H} u_j d_{jH}$
and $S -  \sum_{j=1}^{r_H} \alpha_{\upsilon}(u_j) \upsilon\ast D_{jH}\in{\Cal S}_{s-1}.$
\end{proof}
\begin{Example} If we consider the ring of Example~\ref{GOeY}  as a left ${\Bbb Z}[x]$-module endowed 
with the $\Gamma$-pseudovaluation, 
$\Gamma =  \{Y_1^{a_1}Y_2^{a_2}Y_3^{a_3} : (a_1,a_2,a_1)\in{\Bbb N}^3\},$
we obtain a  similar solution as the one described  in Example~\ref{GOeX} .

Expressing each ${\bf M}(f_i)$ as ${\bf M}(f_i)=\lc(f_i){\bf T}(f_i)$, according Zacharias approach we need to compute a syzygy bases in ${\Bbb Z}[x]$ among
$\alpha_{Y_1}(\lc(f_1)) = (y^2-1)$, $\alpha_{Y_2}(\lc(f_2)) = (y^3-1)$ and
$\alpha_{Y_3}(\lc(f_3))  = (y^4-1)$; the natural solutions  
$(-(y^2+y+1),(y^2+1),0)$, $(-(y^2-1),0,1)$
produce $\sigma_1$ and  $\sigma_3$.
\end{Example}
   
   \begin{Example}\label{Azzurro3}
In  Examples~\ref{Azzurro}, the basic elements are the following:
$$\begin{array}{l|ll|l}
&H&{\bf T}(H)&f_H\\
\hline
s=1&\{1\}&Y_1Y_2^3Y_3^2& f_1 \\
\hline
s=2&\{1\}&Y_1Y_2^3Y_3^2& f_1 \\
&\{2\}&Y_1^2Y_2Y_3^2& f_2\\
&\{1,2\}&Y_1^2Y_2^3Y_3^2& f_{\{1,2\}}=4xY_1^2Y_2^3Y_3^2,\\
\hline
s=3&\{1\}&Y_1Y_2^3Y_3^2& f_1 \\
&\{2\}&Y_1^2Y_2Y_3^2& f_2\\
&\{3\}&Y_1^2Y_2^3Y_3& f_3\\
&\{1,2,3\}&Y_1^2Y_2^3Y_3^2& f_{\{1,2,3\}}=xY_1^2Y_2^3Y_3^2,\\
\hline
s=4&\{1\}&Y_1Y_2^3Y_3^2& f_1 \\
&\{2\}&Y_1^2Y_2Y_3^2& f_2\\
&\{3\}&Y_1^2Y_2^3Y_3& f_3\\
&\{2,4\}&Y_1^2Y_2^2Y_3^2& f_{\{2,4\}}=f_4\\
&\{1,2,3,4\}&Y_1^2Y_2^3Y_3^2& f_{\{1,2,3,4\}}=f_{\{1,2,3\}}.\\
\end{array}$$
\qed\end{Example}

\begin{Corollary}\label{48Cz1} Assuming that the Zacharias domain $R$ is a principal ideal domain and denoting\footnote{Remember that $\alpha_{\{i,j\},j}=\alpha_\tau$ for $\tau= \frac{\lcm(\tau_i,\tau_j)}{\tau_j}$.}, for each $i,j, 1 \leq i < j \leq u$, ${\bf e}_{\iota_i} ={\bf e}_{\iota_j}$,
\begin{eqnarray*}
b(i,j) &:=& 
\frac{\lcm\left(
\alpha_{\{i,j\},i}(c_i),
\alpha_{\{i,j\},j}(c_j)\right)}
{\alpha_{\{i,j\},j}(c_j)} \frac{\lcm(\tau_i,\tau_j)}{\tau_j} e_j
\\ &-&
\frac{\lcm\left(
\alpha_{\{i,j\},i}(c_i),
\alpha_{\{i,j\},j}(c_j)\right)}{\alpha_{\{i,j\},i}(c_i)} \frac{\lcm(\tau_i,\tau_j)}{\tau_i}e_i,
\\B(i,j) &:=&  
\frac{\lcm\left(
\alpha_{\{i,j\},i}(c_i),
\alpha_{\{i,j\},j}(c_j)\right)}
{\alpha_{\{i,j\},j}(c_j)} \frac{\lcm(\tau_i,\tau_j)}{\tau_j}\star g_j 
\\ &-&
\frac{\lcm\left(
\alpha_{\{i,j\},i}(c_i),
\alpha_{\{i,j\},j}(c_j)\right)}{\alpha_{\{i,j\},i}(c_i)}
\frac{\lcm(\tau_i,\tau_j)}{\tau_i}\star g_i
\end{eqnarray*}
we have that $\{b(i,j) : 1 \leq i < j \leq u, {\bf e}_{\iota_i} ={\bf e}_{\iota_j}\}$ is a Gebauer--M\"oller set for
$F$,
so that 
 $F$ is a   Gr\"obner basis of ${\sf M}$, iff
each $B(i,j)$, $1 \leq i < j \leq u, {\bf e}_{\iota_i} ={\bf e}_{\iota_j},$  has a weak  Gr\"obner
representation in terms of $F$.
\qed\end{Corollary}

\begin{proof} Since, for any  basic subset 
$H\subset\{1,\ldots,s\}\cap{\Frak H}(F)$
we have
\begin{eqnarray*}
{\Bbb I}(\{\alpha_{\{i,s\},i}(c_i) : i\in H^{\times}\}) : {\Bbb I}(\alpha_{\{i,s\},s}(c_s)) &=&
 \bigoplus ({\Bbb I}(\alpha_{\{i,s\},i}(c_i)) : {\Bbb I}(\alpha_{\{i,s\},s}(c_s))) 
 \\ &=& 
 {\Bbb I}\left(\frac{\lcm(\alpha_{\{i,s\},i}(c_i),\alpha_{\{i,s\},s}(c_s))}{\alpha_{\{i,s\},s}(c_s)}\right)
 \end{eqnarray*}
and $b(i,s)$ is the syzygy associated to $\{i,s\}$ and 
$\frac{\lcm(\alpha_{\{i,s\},i}(c_i),\alpha_{\{i,s\},s}(c_s))}{\alpha_{\{i,s\},s}(c_s)}$.
\end{proof}
\begin{Example}\label{Azzurro1}
In  Examples~\ref{Azzurro}, we obtain the following redundant Gebauer--M\"oller set (see Examples~\ref{Azzurro2})

$$\begin{array}{rllrrrr}
(i,j)&
\lcm\left(
\alpha_{\{i,j\},i}(c_i),
\alpha_{\{i,j\},j}(c_j)\right)&
\lcm(\tau_i,\tau_j)&
b(i,j)&&&\\
\hline
(1,2)&6^4\cdot5x^2&Y_1^2Y_2^3Y_3^2&(-2^23^4xY_1,&5Y_2^2,&0,&0)\\
(1,3)&60x&Y_1^2Y_2^3Y_3^2&(-3Y_1,&0,&4Y_3,&0)\\
(2,3)&6^4\cdot5x^2&Y_1^2Y_2^3Y_3^2&(0,&-5Y_2^2,&3^32^4xY_3,&0)\\
(1,4)&60x&Y_1^2Y_2^3Y_3^2&(-3Y_1,&0,&0,&10Y_2)\\
(2,4)&6^2x^2&Y_1^2Y_2^2Y_3^2&(0,&-Y_2,&0,&6^2x)\\
(3,4)&30x&Y_1^2Y_2^3Y_3^2&(0,&0,&-2Y_3,&5Y_2)\\
\end{array}$$
\qed\end{Example}

\begin{Corollary}\label{48Czr1} Assuming that the Zacharias domain $R$ is a principal ideal domain and 
that each $\alpha_i$  is an automorphism
denoting, for each $i,j, 1 \leq i < j \leq u$, ${\bf e}_{\iota_i} ={\bf e}_{\iota_j}$,
\begin{eqnarray*}
b(i,j) &:=& 
e_j\alpha^{-1}_{\tau_j}\left(\frac{\lcm(c_i,c_j)}{c_j}\right) \frac{\lcm(\tau_i,\tau_j)}{\tau_j}  -
e_i\alpha^{-1}_{\tau_i}\left(\frac{\lcm(c_i,c_j)}{c_i}\right) \frac{\lcm(\tau_i,\tau_j)}{\tau_i} 
\\B(i,j) &:=&  
g_j\star\alpha^{-1}_{\tau_j}\left(\frac{\lcm(c_i,c_j)}{c_j}\right) \frac{\lcm(\tau_i,\tau_j)}{\tau_j}  -
g_i\star\alpha^{-1}_{\tau_i}\left(\frac{\lcm(c_i,c_j)}{c_i}\right) \frac{\lcm(\tau_i,\tau_j)}{\tau_i} 
\end{eqnarray*}
we have that $\{b(i,j) : 1 \leq i < j \leq u, {\bf e}_{\iota_i} ={\bf e}_{\iota_j}\}$ is a Gebauer--M\"oller set for
$F$,
so that 
 $F$ is a  right Gr\"obner basis of ${\sf M}$, iff
each $B(i,j)$, $1 \leq i < j \leq u, {\bf e}_{l_i} ={\bf e}_{l_j},$  has a right weak  Gr\"obner
representation in terms of $F$.
\qed\end{Corollary}

\subsection{Third algorithm: from weak to strong Gr\"obner basis}\label{c46S11E} \label{3Alg}

As regards {\em strong} Gr\"obner bases, we have

\begin{Definition} A set $C\subset{\sf R}^m$ is called a
{\em completion}\index{completion} of $F$, if, for each subset
$H\subset{\Frak H}(F)$ which is maximal for ${\bf T}(H)$, it contains an element
$f_H\in {\Bbb I}(F)$ 
which satisfies
\begin{enumerate}
\item ${\bf T}(f_H) = {\bf T}(H) = \tau_H\varepsilon_H$,
\item $\lc(f_H) = c_H = \gcd\left(\alpha_{H,i}(\lc(g_i)): i\in H\right)$,
\item $f_H$ has a  Gr\"obner representation in terms of $F$.
\end{enumerate}\end{Definition}

\begin{Algorithm}[M\"oller]\label{ShMo2} A completion of $F$ can be inductively computed by
mimicking the construction of Theorem~\ref{48Tx1}
as follows: the result being trivial if $\#F = 1$, we can assume to have
already obtained a completion $C(F^\times)$ of $F^\times = \{g_1,\ldots,g_{s-1}\},
s \leq u$; for each maximal subset $H \subset\{1,\ldots,s\}$, if $s\notin H$ we
can take as $f_H$ the corresponding element in $C(F^\times)$.
If instead $s\in H$, then $H^\times$ is maximal in $F^\times$ for 
${\bf T}(H^\times)$ and $\tau_{H^\times} \mid \tau_{H}$; thus there is a corresponding element
$f_{H^\times}$ in $C(F^\times)$;
let us compute the values $s,t,d\in R$ such that
$$\alpha_{H,H^\times}(\lc(f_{H^\times})) s + \alpha_{H,s}(\lc(g_s)) t = 
\gcd(\alpha_{H,H^\times}(\lc(f_{H^\times})),\alpha_{H,s}(\lc(g_s))) = d$$ 
and define
$f_H := s \frac{\tau_H}{\tau_{H^\times}}\star f_{H^\times} + 
t\frac{\tau_H}{\tau_s} \star g_s$
which satisfies ${\bf M}(f_H) = d {\bf T}(H) = d\tau_H\epsilon_H$ so that
\begin{enumerate}
\item ${\bf T}(f_H) = {\bf T}(H) = \tau_H\epsilon_H$,
\item $\lc(f_H) = \gcd(\alpha_{H,H^\times}(\lc(f_{H^\times})),\alpha_{H,s}(\lc(g_s))) = \ \gcd\left(\alpha_{H,i}(\lc(g_i)): i\in H\right) = d;$.
\item it is sufficient to substitute $f_{H^\times}$ with its  Gr\"obner
representation, to obtain the required  Gr\"obner representation of $f_H$.
\qed\end{enumerate}
\end{Algorithm}

\begin{Proposition}[M\"oller]\label{48xP1}  With the present notation 
and under the  assumption that $R$ is a principal ideal domain, 
the following conditions are equivalent: 
\begin{enumerate}
\item $F$ is a  Gr\"obner basis of ${\sf M}$;
\item a completion of $F$ is a strong   Gr\"obner basis of ${\sf M}$.
\end{enumerate}
\end{Proposition}

\begin{proof}\

\begin{description}
\item[$(1) \then (2)$] 

Let $f\in{\sf M}$ and let
$f = \sum_{i=1}^u h_i\star g_i$ be a  Gr\"obner representation;
denoting $H := \{j : {\bf T}(h_j\star g_j) = {\bf T}(f) =: \tau\epsilon\}$ we have
$\tau_H \mid\tau, \epsilon_H = \epsilon$. 
Thus, setting 
$\upsilon_j :=\frac{\tau}{\tau_j},\omega_j :=\frac{\tau_H}{\tau_j}$ for each $j$ and
$\lambda:=\frac{\tau}{\tau_H}$ we have
$$\begin{array}{rclcl}
\lc(f) &=& \sum_{j\in H} \lc(h_j)\alpha_{\upsilon_j}(\lc(g_j)) &&\\
&=& \sum_{j\in H} \lc(h_j)\alpha_{\lambda}\alpha_{\omega_j}(\lc(g_j))
&\in&
{\Bbb I}\left(\alpha_{\lambda}\alpha_{\omega_j}(\lc(g_j)) : j\in H\right) \\
&& &=& 
\alpha_{\lambda}\left({\Bbb I}\left(\alpha_{\omega_j}(\lc(g_j)) : j\in H\right)\right)  \\
&&&=& \alpha_{\lambda}({\Bbb I}(c_H))
\end{array}$$
so that
$ \alpha_{\lambda}(\lc(f_H)) = \alpha_{\lambda}(c_H)\mid \lc(f)$
and $\lc(f)=d\alpha_{\lambda}(\lc(f_H))$ with $d\in R.$ In conclusion we have
${\bf M}(f) =d\lambda\ast  {\bf M}(f_H)$.
\item[$(2) \then (1)$]   Let $f\in{\sf M}$ and let
$f = \sum\limits_{K\subset{\Frak H}(F)} c_K \tau_K f_K$ be a strong  Gr\"obner
representation of it in terms of 
a completion of $F$; it is sufficient to substitute
each $f_K$ with a  Gr\"obner
representation of it in terms of $F$ to obtain the required representation.
\end{description}\end{proof}

\begin{Example}
In the ring of Examples~\ref{GOeY} and~\ref{GOeX}, we finally have (see Remark~\ref{GOeR})
$$f_{\{1,2\}}=f_{\{1,3\}}=f_{\{1,2,3\}}={\Frak s}_L(\varsigma_A)={\Frak s}_L(\varsigma_B)=(y-1)Y_1^{2}Y_2^{2}Y_3^{2}.$$
\qed\end{Example}
\begin{Example}\label{Azzurro4}
The strong Gr\o"bner basis (see Examples~\ref{Azzurro3}) is 
$$\{f_1,f_2,f_3,f_4,f_{\{1,2,3,4\}}\}$$
since
$$
\gcd(\alpha_{\{1,2\},\{1\}}(\lc(f_{1})),\alpha_{\{1,2\},2}(\lc(f_2)))=
\gcd(\alpha_{Y_1}(x)),\alpha_{Y_2^2}(x))=
\gcd(20x)),36x)=4x$$
$$
        \gcd(\alpha_{\{1,2,3\},\{1,2\}}(\lc(f_{\{1,2\}})),\alpha_{\{1,2,3\},3}(\lc(f_3)))=
        \gcd(\alpha_{1}(4x)),\alpha_{Y_3}(x))=
         \gcd(4x)),15x)=x.$$
Similarly, $f_{\{2,4\}}=f_4$ follows trivially from
$$
        \gcd(\alpha_{\{2,4\},\{2\}}(\lc(f_{2})),\alpha_{\{2,4\},4}(\lc(f_4)))=
        \gcd(\alpha_{Y_2}(x^2)),\alpha_{1}(x))=
         \gcd(36x^2)),x)=x.$$
\qed\end{Example}

\subsection{Useless  S-pairs and Gebauer-M\"oller sets}\label{c46S12g}  

Let us still assume that  the Zacharias domain $R$ is a principal ideal domain
and we will use freely notations as 
${\bf M}(i), {\bf M}(i,j), {\bf M}(i,j,k), 1 \leq i,j,k \leq u$, instead of
${\bf M}(\{i\}),$ ${\bf M}(\{i,j\}),$ ${\bf M}(\{i,j,k\})$; we can then easily apply to the present setting the reformulation and improvement by Gebauer--M\"oller  \cite{GM} of 
Buchberger Criteria  \cite{BCrit}.
However we must be aware that in this context, there is no chance of  reformulating 
Buchberger's First Criterion.

\begin{Remark} In fact we should at least require that
$${\bf M}(i)\ast {\bf M}(j) = {\bf M}(i)\ast {\bf M}(j)={\bf M}(i,j)$$
{\em id est} not only $\lcm(\tau_i,\tau_j)=\tau_i\circ\tau_j=\tau_j\circ\tau_i$ which is trivially true but also
$$\lcm(
\alpha_{\tau_j}(c_i),
\alpha_{\tau_i}(c_j)) = c_j\alpha_{\tau_j}(c_i)=c_i\alpha_{\tau_i}(c_j).$$

This essentially requires $c_i\mid\alpha_{\tau_j}(c_i)$ and $c_j\mid\alpha_{\tau_i}(c_j)$ whence $\alpha_{\tau_j} = \Id$; this suggests that  Buchberger's First Criterion hardly can be applied except for the case of the commutative ring ${\Cal P}=R[Y_1,\ldots,Y_n], R$ a PIR, where it is stated as
\begin{quote}
If $F\subset{\Cal P}$ and ${\Bbb I}(F)$ is an ideal of ${\Cal P}$, it holds
\begin{eqnarray*}
{\bf M}(i) {\bf M}(j) = {\bf M}(i,j) &\iff&
\lcm(\tau_i,\tau_j)=\tau_i\tau_j,
\lcm(c_i,c_j) = c_ic_j 
\\ &\then& {\mathrm{NF}}(B(i,j), F) = 0.
\end{eqnarray*} 
\end{quote}

Note that the proof which considers the trival sysygies $g_ig_j-g_jg_i=0$ holds only to the classical polynomial ring case.
\qed\end{Remark}

\begin{Definition} 
A {\em useful S-pair set}
for $F$ is any 
subset 
$${\GM}\subset{\Cal S}(u) = \left\{
\{i,j\},1 \leq i < j \leq u, {\bf e}_{l_i} ={\bf e}_{l_j}\right\}$$ such that
$\left\{b(i,j) : \{i,j\}\in{\GM}\right\}$ is a Gebauer--M\"oller set for $F$.
\end{Definition}

\begin{Corollary} \label{c46C2} With the present notation, under the assumption that
 $R$ is a principal ideal domain, 
  $F$ is a    Gr\"obner basis of the left module ${\sf M}$ iff, denoting
$\GM$  a useful S-pair set  for $F$,
each S-polynomial $B(i,j), \{i,j\}\in\GM$ 
 has
 a   Gr\"obner representation in terms of $F$.
\qed\end{Corollary}

\begin{proof}  By definition
$\left\{b(i,j) : \{i,j\}\in{\GM}\right\}$ is a Gebauer--M\"oller set for $F$ so that,
by Theorem~\ref{SAASSAAS}, 
$F$ is a    Gr\"obner basis of ${\sf M}$ iff
each S-polynomial $B(i,j), \{i,j\}\in\GM$ 
 has
 a   Gr\"obner representation in terms of $F$.
\end{proof} 

If we moreover define, 
\begin{itemize}
\renewcommand\labelitemi{\bf --}
\item for each $i,j : 1 \leq i , j  \leq u,  {\bf e}_{l_i} ={\bf e}_{l_j},$
\begin{itemize}
\item $c(i,j) :=\lcm(
\alpha_{\{i,j\},i}(c_i),
\alpha_{\{i,j\},j}(c_j)),$
\item $\tau(i,j)=\lcm(\tau_i,\tau_j)$
\item $\mu(i,j) =c(i,j)\tau(i,j)$
\end{itemize}
\item and  for each $i,j,k : 1 \leq i , j , k \leq u,  {\bf e}_{l_i} ={\bf e}_{l_j} ={\bf e}_{l_k},$
\begin{itemize}
\item $c(i,j,k) :=\lcm(
\alpha_{\{i,j,k\},\{i,j\}}(c(i,j)),
\alpha_{\{i,j,k\},\{i,k\}}(c(i,k)),
\alpha_{\{i,j,k\};\{j,k\}}(c(j,k)))$
\item $\tau(i,j,k)=\lcm(\tau_i,\tau_j,\tau_k)$
\item $\mu(i,j,k) =c(i,j,k) \tau(i,j,k)$,
\end{itemize}
and we impose
 on 
the set
$${\Cal S}(u) := \left\{
\{i,j\},1 \leq i < j \leq u, {\bf e}_{l_i} ={\bf e}_{l_j}\right\} $$
the ordering $\prec$ defined by
\begin{equation}\label{eqsyzord}
\{i_1,j_1\} \prec \{i_2,j_2\} \iff \begin{cases}\tau(i_1,j_1) < \tau(i_2,j_2)&\mbox{  or}\cr
\tau(i_1,j_1) = \tau(i_2,j_2), j_1 < j_2&\mbox{  or}\cr
\tau(i_1,j_1) = \tau(i_2,j_2), j_1 = j_2, i_1 < i_2,& \cr\end{cases}
\end{equation}
\end{itemize}
we obtain
\begin{Definition}\label{c46D6} An S-element
$b(i,j), 
1 \leq i < j \leq u, {\bf e}_{l_i} ={\bf e}_{l_j},$ and the related S-pair $\{i,j\}$ are
 called {\em redundant}\index{S-element!redundant}\index{redundant!S-element/pair}
if either
\begin{enumerate}
\renewcommand\theenumi{{\rm (\alph{enumi})}}
\item exists $k > j$, ${\bf e}_{l_k} = {\bf e}_{l_i} ={\bf e}_{l_j}$ such that $$\mu(i,j,k) = \mu(i,j); \,
\mu(i,k) \neq \mu(i,j) \neq \mu(j,k)$$  
\item or exists $k < j,{\bf e}_{l_k} = {\bf e}_{l_i} ={\bf e}_{l_j} : \mu(j,k) \mid \mu(i,j) \neq \mu(k,j)$.
\qed\end{enumerate}
\end{Definition}

\begin{Lemma}[M\"oller]\label{48Lx1} The following holds
\begin{enumerate}
\item for each $i,j,k : 1 \leq i , j , k \leq u,  {\bf e}_{l_i} ={\bf e}_{l_j} ={\bf e}_{l_k},$ it holds
$$\frac{\mu(i,j,k)}{ \mu(i,k)} b(i,k) 
- \frac{\mu(i,j,k)}{ \mu(i,j)} b(i,j) 
+ \frac{\mu(i,j,k)}{ \mu(k,j)} b(k,j) 
= 0.$$ 
\item ${\Frak R} := \bigl\{b(i,j), 1\leq i < j \leq u, {\bf e}_{l_i} ={\bf e}_{l_j} 
\mbox{\ and not redundant}\bigr\}$
is a useful S-element set.
\item Let $G := \{g_1,\ldots,g_s\}, s \leq u$, and let $$\GM_* \subset
\{\{i,j\}, 1\leq i < j < s, {\bf e}_{l_i} ={\bf e}_{l_j}\}$$ be a useful S-pair set for $G_* = \{g_1,\ldots,
g_{s-1}\}$.

Let $\overline{M} := \{\mu(j,s) : 1 \leq j < s,{\bf e}_{l_j} ={\bf e}_{l_s}\}$ and let
$\overline{M}' \subset\overline{M}$ be the set of the elements $\mu := \mu(j,s) \in\overline{M}$ such
that 
there exists $\mu(j',s)\in \overline{M} : \mu(j',s) \mid \mu(j,s) 
\neq \mu(j',s)$.

For each $\mu := {\bf M}(j,s)\in\overline{M}\setminus\overline{M}'$ 
let $i_\mu, 1 \leq i_\mu < s, $ be such that 
$\mu = {\bf M}(i_\mu,s).$ 
Then $$\GM := \GM_* \cup \{\{i_\mu,s\} :
\mu\in\overline{M}\setminus\overline{M}'\}$$ is a useful S-pair set for $G$. 
\end{enumerate}\end{Lemma}

\begin{proof} 
\begin{enumerate}
\item ({\em cf.}  \cite[Lemma~25.1.4]{SPES}) One has
\begin{eqnarray*}
&& \frac{\mu(i,j,k)}{ \mu(i,k)}\ast b(i,k) - \frac{\mu(i,j,k)}{\mu(i,j)}\ast b(i,j) + \frac{\mu(i,j,k)}{ \mu(k,j)}\ast b(k,j) 
\\&=& 
\frac{\mu(i,j,k)}{\mu(i,k)}\ast\left(\frac{{\mu(i,k)}}{{\mu(k)}} e_k - \frac{{\mu(i,k)}}{{\mu(i)}} e_i\right)
\\&-&
\frac{\mu(i,j,k)}{\mu(i,j)}\ast \left(\frac{{\mu(i,j)}}{{\mu(j)}} e_j - \frac{{\mu(i,j)}}{{\mu(i)}} e_i\right)
\\&+&
\frac{\mu(i,j,k)}{ \mu(k,j)} \ast\left(\frac{{\mu(k,j)}}{{\mu(j)}} e_j -\frac{{\mu(k,j)}}{{\mu(k)}} e_k\right)
\\&=&
\left(\frac{{\mu(i,j,k)}}{{\mu(k)}} e_k - \frac{{\mu(i,j,k)}}{{\mu(i)}} e_i\right)
\\&-&
\left(\frac{{\mu(i,j,k)}}{{\mu(j)}} e_j - \frac{{\mu(i,j,k)}}{{\mu(i)}} e_i\right)
\\&+&
\left(\frac{{\mu(i,j,k)}}{{\mu(j)}} e_j - \frac{{\mu(i,j,k)}}{{\mu(k)}} e_k\right)
\\&=&
0 
\end{eqnarray*}
\item ({\em cf.}  \cite[Lemma~25.1.8]{SPES}) In order to prove the claim by induction, it is sufficient to show
that, for each redundant $\{i,j\}, 1\leq i < j \leq u, {\bf e}_{l_i} ={\bf e}_{l_j} =: \epsilon,$ 
there are
\begin{itemize}
\renewcommand\labelitemi{\bf --}
\item  $\{i_1,j_1\}, \ldots, \{i_\rho,j_\rho\}, \ldots, \{i_r,j_r\}, 1\leq i_\rho < j_\rho  \leq u$,  
${\bf e}_{l_{i_\rho}} ={\bf e}_{l_{j_\rho}} = \epsilon$,
\item elements $t_1,\ldots,t_r \in {\Cal T},$ 
\item and coefficients $c_1,\ldots c_r \in R\setminus\{0\}$
\end{itemize}
such that
\begin{itemize}
\item $b(i,j) = \sum_\rho c_\rho t_\rho\ast b(i_\rho,j_\rho);$
\item $\tau(i,j) = t_\rho\circ \tau(i_\rho,j_\rho),$ for each $\rho;$
\item $\{i_\rho,j_\rho\} \prec  \{i,j\}$.
\end{itemize}

In order to show this, we only need to consider the representation 
$$b(i,j) = \frac{\mu(i,j,k)}{ \mu(i,k)}\ast b(i,k) + \frac{\mu(i,j,k)}{ \mu(k,j)}\ast b(k,j)$$
and to prove that 
$$\{i,k\} \prec  \{i,j\} \succ  \{k,j\};$$ 
this happens (according to the two cases of the definition)
because
\begin{enumerate}
\item 
$\tau(i,k) \mid \tau(i,j,k) = \tau(i,j) \neq \tau(i,k)$ implies
$\{i,k\} \prec 
\{i,j\}$  and the same argument proves $\{j,k\} \prec  \{i,j\};$
\item the same argument as that above proves $\{j,k\} \prec  \{i,j\},$  while
$\{i,k\} \prec  \{i,j\}$ because $\tau(i,k) \leq\tau(i,j)$ and $k < j.$
\end{enumerate}
\item ({\em cf.}  \cite[Lemma~25.1.9]{SPES})\index{Moeller's@M\"oller's!Lemma}
 Let $i < s,  {\bf e}_{l_i} ={\bf e}_{l_s} =: \epsilon, \mu := \mu(i,s)$. Then:
\begin{itemize}
\item if there exists $\mu'\in\overline{M}$ such that $\mu(i_{\mu'},s) = \mu' \mid \mu(i,s) \neq
\mu'$, then since $i_{\mu'} < s$, $\{i,s\}$ is redundant;
\item if $i = i_\mu$ 
 then $\{i_m,s\}\in \GM$;
\item if $i \neq i_\mu$ then 
$b(i,s) = \frac{\mu(i,i_\mu,s)}{\mu(i,i_\mu)} b(i,i_\mu) - b(i_\mu,s).$
 \end{itemize}
\end{enumerate}\end{proof}

\begin{Corollary}\label{48Cy1}  With the present notation, under the assumption that
 $R$ is a principal ideal domain,  
  $F$ is a    Gr\"obner basis of ${\sf M}$ iff 
each S-polynomial $B(i,j), \{i,j\}\in\Frak{R}$ 
 has
 a   Gr\"obner representation in terms of $F$.
\qed\end{Corollary}

\begin{Example}\label{Azzurro2}
In connection with Lemma~\ref{48Lx1} we have
$$\begin{array}{rll|l}
(i,j,k)&
c(i,j,k)&
\mu(i,j,k)&
b(i,j,k)\\
\hline
(1,2,3)&6^4\cdot5x^2&Y_1^2Y_2^3Y_3^2&b(1,2)-2^23^3xb(1,3)+b(2,3)=0,\\
(1,2,4)&6^4\cdot5x^2&Y_1^2Y_2^3Y_3^2&b(1,2)-2^23^3xb(1,4)+5Y_2\ast b(2,4)=0,\\
(1,3,4)&60x&Y_1^2Y_2^3Y_3^2&b(1,3)-b(1,4)+2 b(3,4)=0.\\
(2,3,4)&6^4\cdot5x^2&Y_1^2Y_2^3Y_3^2&b(2,3)-5Y_2\ast b(2,4)+6^3 b(3,4)=0.\\
\end{array}.$$

Note that we obviously \cite{J,MM} have also
$$b(1,2,3)-b(1,2,4)+2^23^3xb(1,3,4)-b(2,3,4).$$

Thus the redundant elements are $b(2,3)$ via 1 or 4, $b(1,2)$ via 4 and $b(1,4)$ via 3.

But, as it is well-known, it is more efficient (if else for storing considerations) the algorithm sketched in Lemma~\ref{48Lx1}.3 which  
\begin{description}
\item[for $s=2$] stores $(1,2)$,
\item[for $s=3$] stores $(1,3)$,
\item[for $s=4$] removes
$(1,2)$ and stores $(2,4)$ and $(3,4)$.
\end{description}

Thus the Gebauer M\"oller set is still $$\{b(1,3),b(2,4),b(3,4)\}$$ while
\begin{eqnarray*}
b(1,4)&=&b(1,3)+2 b(3,4),\\
b(2,3)&=&5Y_2\ast b(2,4)+6^3 b(3,4),\\
b(1,2)&=&2^23^3xb(1,4)-5Y_2\ast b(2,4).
\end{eqnarray*}
\qed\end{Example}

\section{Weis\-pfen\-ning Completions for Bilateral Gr\"obner basis  for Multivariate Ore Extensions of Zacharias Domains}

\subsection{Kan\-dri-Ro\-dy--Weis\-pfen\-ning completion}

The most efficient technique for producing bilateral  Gr\"obner bases  $G:={\Bbb I }_2(F)$  in  a noetherian Ore extension is
Kan\-dri-Ro\-dy--Weis\-pfen\-ning completion  \cite{KrW}. Iteratively:
 \begin{itemize}
 \item Repeat
 \begin{itemize}
 \item Compute  a left-Gr\"obner basis $G$ of the ideal ${\Bbb I }_L(F)$;
  \item for each $g\in G, 1\leq i \leq n$, compute the normal form $\NF(g\star Y_i, {\Bbb I }_L(F))$ of 
$g\star Y_i$ w.r.t. $G$;
\item set $H:=\{\NF(g\star Y_i, {\Bbb I }_L(G)), g\in G, 1\leq i \leq n\}$, $F:=G\cup H$
\end{itemize}
until $H=\emptyset$.
\end{itemize}

The {\em rationale} of the algorithm is

\begin{Lemma}[Kandri-Rody--Weispfenning]\label{KR3} 
For $G\subset{\sf R}$ the following conditions are equivalent:
\begin{enumerate}
\item ${\Bbb I}_L(G)={\Bbb I}_2(G)$;
\item for each $\tau\in{\Cal T}$ and each $g\in G$, $g\star\tau\in{\Bbb I}_L(G)$;
\item for each $i, 1 \leq i \leq n$, and each $g\in G$, $g\star Y_i\in{\Bbb I}_L(G)$.
\end{enumerate} \end{Lemma}

\begin{proof}\

\begin{description}
\item[$(1) \then (2) \iff (3)$] is trivial.
\item[$(2) \then (1)$]  
${\Cal B}_2(G)  := \{\lambda\star g\star\rho :
\lambda,\rho\in{\Cal T}, g\in G\}$ 
is an $R$-linear basis of ${\Bbb I}_2(G)$ and satisfies
$$
{\Cal B}_2(G)  
=
 \{\lambda\star\left(g\star\rho\right) :
\lambda,\rho\in{\Cal T}, g\in G\} 
\subseteq
 \{\lambda\star h :
\lambda\in{\Cal T}, h\in {\Bbb I}_L(G)\}\subseteq{\Bbb I}_L(G).
$$
\end{description}\end{proof}


\subsection{Weispfenning: Restricted Representation and Completion}\label{SWRRC}

We can wlog assume that $R$ is effectively given as a quotient $R={\Cal R}/{\sf I}$ of a 
free monoid ring ${\Cal R} := {\Bbb Z}\langle {\Bcc v}\rangle$ (over  ${\Bbb Z}$ and the monoid $\langle {\Bcc v}\rangle$ of all words over the alphabet ${\Bcc v}$) modulo a bilateral ideal ${\sf I}$.

We must restrict ourselves to the case in which $<$ is a sequential term-ordering, {\em id est} for each $\tau\in{\Cal T}$, the set $\{\omega\in{\Cal T} : \omega<\tau\}$ is finite.

 \begin{Lemma} \cite{W} Let 
$$F := \{g_1,\ldots,g_u\}\subset {\sf R}^m, 
g_i = {\bf M}(g_i)-p_i =: c_i \tau_i {\bf e}_{\iota_i} - p_i;$$
set $\Omega := \max_<\{{\bf T}(g_i) : 1 \leq i\leq u\}$.

Let ${\sf M}$ be the bilateral module ${\sf M} := {\Bbb I}_2(F)$ 
and  ${\Bbb I}_W(F)$  the  restricted module 
$${\Bbb I}_W(F) : = \Span_{R}(a f\star \rho   : a\in
R\setminus\{0\}, \rho\in {\Cal T}, f\in F).$$

If every $f\star \alpha_{\upsilon}(v), f\in F, v\in{\Bcc v}, \upsilon\in{\Cal T},\upsilon<\Omega$, has a restricted representation in terms of $F$ w.r.t. a sequential term-ordering $<$, then
every $f\star r, f\in F, r\in{\sf R}$, has a restricted representation in terms of $F$ w.r.t. $<$.
\end{Lemma}

\begin{proof} We can wlog assume $r=\prod_{i=1}^\nu v_i, v_i\in{\Bcc v}$ and prove the claim by induction on $\nu\in{\Bbb N}$.
 
 Thus we have a restricted representation in terms of $F$
 $$f\star\left(\prod_{i=1}^{\nu-1} v_i\right)=\sum_j d_j  g_{i_j}\star \rho_j, \tau_{i_j}\circ\rho_j\leq{\bf T}(f),$$ whence we obtain
 $$f\star \prod_{i=1}^\nu v_i=\left(f\star\prod_{i=1}^{\nu-1} v_i\right)\star v_\nu=\sum_j d_j  g_{i_j}\star \left(\rho_j\star v_\nu\right)
=\sum_j d_j   g_{i_j}\star \alpha_{\rho_j}(v_\nu)\rho_j$$
and since $\rho_j<{\bf T}(f)\leq\Omega$ each element
$g_{i_j}\star \alpha_{\rho_j}(v_\nu)$ can be substituted with its restricted representation whose existence is granted by assumption.
\end{proof}

\begin{Lemma} \cite{W}  Under the same assumption, if,  for each 
$g_j\in F$, both $Y_i\star g_j, 1\leq i \leq n$  and  each 
$g_j\star  \alpha_{\upsilon}(v),  v\in{\Bcc v}, \upsilon\in{\Cal T},\upsilon<\Omega$, have a restricted representation in terms of $F$ w.r.t. $<$, then ${\Bbb I}_W(F)={\sf M}$.
\end{Lemma}

\begin{proof} It is sufficient to show that, for each $f\in {\Bbb I}_W(F)$, both each
$Y_i\star f\in  {\Bbb I}_W(F), 1\leq i \leq n$ and each $f\star r\in {\Bbb I}_W(F),r\in{\sf R}$.

By assumption $f=\sum_j d_j g_{i_j} \star \rho_j, d_j\in R\setminus\{0\}, \rho_j\in{\Cal T},1\leq i_j\leq u$, so that
$$Y_i\star f = \sum_j \alpha_i(d_j)\star(Y_i\star g_{i_j}) \star \rho_j
\And f\star r=\sum_j d_j  (g_{i_j}\star  \alpha_{\rho_j}(r))\star\rho_j;$$
by assumption each $Y_i\star g_{i_j}$ has a restricted representation in terms of $F$;
for the Lemma above, also each $g_{i_j} \star \alpha_{\rho_j}(r)$  has a restricted representation in terms of $F$.
\end{proof}

\begin{Corollary}\label{48cC1} \cite{W}  Let 
$$F := \{g_1,\ldots,g_u\}\subset {\sf R}^m, 
g_i = {\bf M}(g_i)-p_i =: c_i \tau_i {\bf e}_{\iota_i} - p_i.$$
Let ${\sf M}$ be the bilateral module ${\sf M} := {\Bbb I}_2(F)$ 
and  ${\Bbb I}_W(F)$  the  restricted module 
$${\Bbb I}_W(F) : = \Span_{R}(a  f\star \rho   : a\in
R\setminus\{0\}, \rho\in {\Cal T}, f\in F).$$

$F$ is the bilateral Gr\"obner basis of  ${\sf M}$ iff
\begin{enumerate}
\item denoting 
$\GM(F)$   any  restricted Gebauer--M\"oller set for $F$,  each $\sigma\in\GM(F)$ has a restricted quasi-Gr\"obner representation
in terms of $F$;
\item for each 
$g_j\in F$, both $Y_i\star g_j, 1\leq i \leq n$  and  each 
$g_j\star  \alpha_{\upsilon}(v),  v\in{\Bcc v}, \upsilon\in{\Cal T},\upsilon<\Omega$, have a restricted representation in terms of $F$ w.r.t. $<$.
\end{enumerate}
\end{Corollary}
\subsection{Gebauer-M\"oller sets for Restricted Gr\"obner bases}

It is clear from Corollary~\ref{48cC1}  that the computation of a Gr\"obner bases can be obtained via Weispfenning's completion, provided that we are able to produce restricted Gebauer-M\"oller sets; to do so, we need only to properly reformulate the results of Section~\ref{SWRRC}. 

We begin by remarking that for each monomial $c\tau\in{\sf M}({\sf R})$ the function
$g\mapsto cg\star\tau$ distributes, thus we can define a multiplication 
$\diamond : {\sf R}\times{\sf R}\to{\sf R}$ by setting
$$c_i\tau_i\diamond c_j\tau_j := c_ic_j\tau_j\tau_i = c_jc_i\tau_i\tau_j =: c_j\tau_j\diamond c_i\tau_i$$
which of course is commutative and thus, granting the trivial syzygy
$$g_i\diamond g_j = g_j \diamond g_i$$ allows to recover Buchberger First Criterium.

As a consequence, we can define the notion of restricted Gr\"obner representation: 
\begin{itemize}
\item we say that $f\in {\sf R}^m\setminus\{0\}$ has a restricted {\em Gr\"obner  representation}\index{representation!
Gr\"obner!(of a polynomial)}\index{Gr\"obner!representation (of a polynomial)}
in terms of $G$ if it can be written as
$f = \sum_{i=1}^u l_i  \diamond g_i,$ 
with $l_i\in {\sf R}, g_i \in G$ and 
${\bf T}(l_i)\circ {\bf T}(g_i) \leq {\bf T}(f) \Forall i.$    
\end{itemize}

Let us denote, for each $i,j, 1 \leq i < j \leq u$, ${\bf e}_{l_i} ={\bf e}_{l_j}$,
\begin{eqnarray*}
b_W(i,j) &:=& \frac{\lcm(c_i,c_j)}{c_j} e_j\frac{\lcm(\tau_i,\tau_j)}{\tau_j} - 
\frac{\lcm(c_i,c_j)}{c_i}e_i\frac{\lcm(\tau_i,\tau_j)}{\tau_i}
\\ &=&
\frac{{{\bf M}(i,j)}}{{{\bf M}(j)}}\diamond e_j - \frac{{{\bf M}(i,j)}}{{{\bf M}(i)}}\diamond  e_i
\in\ker({\Frak s}_W), 
\\ B_W(i,j) &:=&
\frac{\lcm(c_i,c_j)}{c_j}  g_j\star\frac{\lcm(\tau_i,\tau_j)}{\tau_j} -
\frac{\lcm(c_i,c_j)}{c_i}g_i\star \frac{\lcm(\tau_i,\tau_j)}{\tau_i}
\\ &=&
\frac{{{\bf M}(i,j)}}{{{\bf M}(j)}}\diamond g_j - \frac{{{\bf M}(i,j)}}{{{\bf M}(i)}}\diamond  g_i.
\end{eqnarray*}
and let us explicitly assume that
\begin{itemize}
\item for each 
$g_j\in F$, both $Y_i\star g_j, 1\leq i \leq n$  and  each 
$g_j\star  \alpha_{\upsilon}(v),  v\in{\Bcc v}, \upsilon\in{\Cal T},\upsilon<{\bf T}(g_j)$, have a restricted representation in terms of $F$ w.r.t. $<$.
\end{itemize}

\begin{Lemma}[Buchberger's First Criterion]\index{Buchberger's!First Criterion}\index{criteria!Buchberger's First}
If $m = 1$, {\em id est} $F\subset{\sf R}$ and ${\Bbb I}_W(F)$ is an ideal of ${\sf R}$, 
then
\begin{eqnarray*}
{\bf M}(i)\diamond {\bf M}(j) = {\bf M}(i,j) &\iff&
\lcm(\tau_i,\tau_j)=\tau_i\tau_j \And
\lcm(c_i,c_j) = c_ic_j 
\\ &\then&{\mathrm{NF}_W}(B_W(i,j), F) = 0.
\end{eqnarray*}
\end{Lemma}

\begin{proof}  We will prove that $B_W(i,j)$ has a restricted Gr\"obner representation in terms of $F$; thus the result will follow by the equivalence   
Proposition~\ref{57P1}, (4)$\iff$(8).

Remark that 
$$p_i := g_i - {\bf M}(i) =  \sum_l  c_{il} t_{il} \And p_j := g_j - {\bf M}(j)=\sum_k c_{jk} t_{jk}$$  satisfy
${\bf T}(p_i) < {\bf T}(g_i),$ ${\bf T}(p_j) < {\bf T}(g_j).$

Then it holds:
$$0 = g_i\diamond g_j - g_j\diamond g_i = {\bf M}(i)\diamond g_j + p_i\diamond  g_j - {\bf M}(j)\diamond  g_i - p_j\diamond  g_i,$$
and
$$B_W(i,j) := \frac{{\bf M}(i,j)}{{\bf M}(j)}\diamond  g_j - \frac{{{\bf M}(i,j)}}{{{\bf M}(i)}}\diamond  g_i
= {\bf M}(i)\diamond  g_j - {\bf M}(j)\diamond  g_i = p_j\diamond  g_i - p_i\diamond  g_j.$$

There are then two possibilities: either
\begin{itemize}
\item ${\bf M}( p_j)\diamond{\bf M}( g_i) \neq {\bf M}( p_i)\diamond{\bf M}( g_j)$ in which case 
$${\bf T}(B_W(i,j)) = \max({\bf T}( p_j)\circ{\bf T}(g_i), {\bf T}( p_i)\circ{\bf T}( g_j))$$ and 
$$B_W(i,j)= p_j\diamond g_i - p_i\diamond g_j =
\sum_k c_{jk}g_i\star t_{jk}-\sum_l  c_{il}g_j\star t_{il}$$
is a restricted Gr\"obner representation;
\item or ${\bf M}( p_j)\diamond{\bf M}( g_i) = {\bf M}( p_i)\diamond{\bf M}( g_j)$, 
${\bf T}(B_W(i,j)) < {\bf T}( p_j)\circ{\bf T}( g_i) = {\bf T}( p_i)\circ{\bf T}( g_j)$, in which case
$B_W(i,j)= p_j\diamond g_i - p_i\diamond g_j$ would not be a Gr\"obner representation.
\end{itemize}
But the latter case is impossible: in fact, from
$${\bf T}( p_i)\circ{\bf T}( g_j) = {\bf T}( p_j)\circ{\bf T}( g_i) < {\bf T}(g_j) \circ{\bf T}(g_i)$$ 
we deduce 
$\lcm({\bf T}(g_i), {\bf T}(g_j)) \neq {\bf T}(g_j)\circ {\bf T}(g_i)$ 
and ${\bf T}(i,j) \neq {\bf T}(i)\circ {\bf T}(j)$
contradicting the assumption 
${\bf M}(i,j) = {\bf M}(i)\diamond {\bf M}(j).$
\end{proof}
\begin{Definition} Denote  
{\small$${\Frak C}_u  
:=\begin{cases}\Bigl\{
\{i,j\} :  {\bf M}(i)\diamond {\bf M}(j) = {\bf M}(i,j)\Bigr\}&\mbox{ if
${\sf M}$ is an ideal}\cr
\emptyset &\mbox{ otherwise.}\end{cases}$$}
A {\em useful S-pair set}\index{S-element!useful}\index{useful S-element/pair} for $F$ is any 
subset 
$${\GM}\subset{\Cal S}(u) = \left\{
\{i,j\},1 \leq i < j \leq u, {\bf e}_{l_i} ={\bf e}_{l_j}\right\}$$ such that
$\left\{b(i,j) : \{i,j\}\in{\GM}\cup{\Frak C}_u\right\}$ is a Gebauer--M\"oller set for $F$.
\end{Definition}

\begin{Corollary}  With the present notation, under the assumption that
 $R$ is a principal ideal domain,  
  $F$ is a    Gr\"obner basis of ${\sf M}$ iff, denoting
$\GM$  a useful S-pair set  for $F$,
each S-polynomial $B_W(i,j), \{i,j\}\in\GM$ 
 has
 a   Gr\"obner representation in terms of $F$.
\qed\end{Corollary}

\begin{proof}  By definition
$\left\{b_W(i,j) : \{i,j\}\in{\GM}\cup{\Frak C}_u\right\}$ is a Gebauer--M\"oller set for $F$ so that,
by Theorem~\ref{SAASSAAS}, 
$F$ is a    Gr\"obner basis of ${\sf M}$ iff
each S-polynomial $B_W(i,j), \{i,j\}\in\GM\cup{\Frak C}_u$ 
 has
 a   Gr\"obner representation in terms of $F$.

The claim is a direct consequence of  Buchberger's First Criterion which states that for each $\{i,j\}\in{\Frak C}_u$, $B_W(i,j)$ has a weak Gr\"obner
representation in terms of $F$.
\end{proof}
\begin{Definition}  An S-element
$b(i,j), 
1 \leq i < j \leq u, {\bf e}_{l_i} ={\bf e}_{l_j},$ and the related S-pair $\{i,j\}$ are
 called {\em redundant}\index{S-element!redundant}\index{redundant!S-element/pair}
if either
\begin{enumerate}
\renewcommand\theenumi{{\rm (\alph{enumi})}}
\item exists $k > j$, ${\bf e}_{l_k} = {\bf e}_{l_i} ={\bf e}_{l_j}$ such that $${\bf M}(i,j,k) = {\bf M}(i,j); \,
{\bf M}(i,k) \neq {\bf M}(i,j) \neq {\bf M}(j,k),$$
\item or exists $k < j,{\bf e}_{l_k} = {\bf e}_{l_i} ={\bf e}_{l_j} : 
{\bf M}(j,k) \mid {\bf M}(i,j) \neq {\bf M}(j,k)$.
\qed\end{enumerate}
\end{Definition}

\begin{Lemma}[M\"oller]  The following holds
\begin{enumerate}
\item for each $i,j,k : 1 \leq i , j , k \leq u,  {\bf e}_{l_i} ={\bf e}_{l_j} ={\bf e}_{l_k},$ it holds
$$\frac{c(i,j,k)}{c(i,k)} b(i,k) \ast\frac{\tau(i,j,k)}{ \tau(i,k)}
- \frac{c(i,j,k)}{c(i,j)} b(i,j) \ast\frac{\tau(i,j,k)}{ \tau(i,j)}
+ \frac{c(i,j,k)}{ c(k,j)} b(k,j) \ast \frac{\tau(i,j,k)}{ \tau(k,j)}
= 0.$$ 
\item ${\Frak R} := \bigl\{b(i,j), 1\leq i < j \leq u, {\bf e}_{l_i} ={\bf e}_{l_j} 
\mbox{\ and not redundant}\bigr\}$
is a useful S-element set.
\item Let $G := \{g_1,\ldots,g_s\}, s \leq u$, and let $$\GM_* \subset
\{\{i,j\}, 1\leq i < j < s, {\bf e}_{l_i} ={\bf e}_{l_j}\}$$ be a useful S-pair set for $G_* = \{g_1,\ldots,
g_{s-1}\}$.

Let $\overline{M} := \{{\bf M}(j,s) : 1 \leq j < s,{\bf e}_{l_j} ={\bf e}_{l_s}\}$ and let
$\overline{M}' \subset\overline{M}$ be the set of the elements $\mu := {\bf M}(j,s) \in\overline{M}$ such
that either 
\begin{itemize}
\item there exists ${\bf M}(j',s)\in \overline{M} : {\bf M}(j',s) \mid {\bf M}(j,s) 
\neq {\bf M}(j',s)$ or 
\item (in case ${\sf M}$ is an ideal) there exists $i_\mu, 1 \leq i_\mu < s :$
$${\bf M}(i_\mu)\diamond{\bf M}(s)={\bf M}(i_\mu,s) =\mu.$$
\end{itemize}
For each $\mu := {\bf M}(j,s)\in\overline{M}\setminus\overline{M}'$ 
let $i_\mu, 1 \leq i_\mu < s, $ be such that 
$\mu = {\bf M}(i_\mu,s).$ 
Then $$\GM := \GM_* \cup \{\{i_\mu,s\} :
\mu\in\overline{M}\setminus\overline{M}'\}$$ is a useful S-pair set for $G$. 
\end{enumerate}\end{Lemma}

\begin{proof} 
\begin{enumerate}
\item ({\em cf.}  \cite[Lemma~25.1.4]{SPES}) One has
\begin{eqnarray*}
&& \frac{{\bf M}(i,j,k)}{ {\bf M}(i,k)}\diamond b(i,k) - \frac{{\bf M}(i,j,k)}{{\bf M}(i,j)}\diamond b(i,j) + \frac{{\bf M}(i,j,k)}{ {\bf M}(k,j)}\diamond b(k,j) 
\\&=& 
\frac{{\bf M}(i,j,k)}{{\bf M}(i,k)}\diamond\left(\frac{{{\bf M}(i,k)}}{{{\bf M}(k)}} \diamond e_k - \frac{{{\bf M}(i,k)}}{{{\bf M}(i)}} \diamond e_i\right)
\\&-&
\frac{{\bf M}(i,j,k)}{{\bf M}(i,j)}\diamond \left(\frac{{{\bf M}(i,j)}}{{{\bf M}(j)}} \diamond e_j - \frac{{{\bf M}(i,j)}}{{{\bf M}(i)}} \diamond e_i\right)
\\&+&
\frac{{\bf M}(i,j,k)}{ {\bf M}(k,j)} \diamond\left(\frac{{{\bf M}(k,j)}}{{{\bf M}(j)}} \diamond e_j -\frac{{{\bf M}(k,j)}}{{{\bf M}(k)}} \diamond e_k\right)
\\&=&
\left(\frac{{{\bf M}(i,j,k)}}{{{\bf M}(k)}} \diamond e_k - \frac{{{\bf M}(i,j,k)}}{{{\bf M}(i)}} \diamond e_i\right)
\\&-&
\left(\frac{{{\bf M}(i,j,k)}}{{{\bf M}(j)}} \diamond e_j - \frac{{{\bf M}(i,j,k)}}{{{\bf M}(i)}} \diamond e_i\right)
\\&+&
\left(\frac{{{\bf M}(i,j,k)}}{{{\bf M}(j)}} \diamond e_j - \frac{{{\bf M}(i,j,k)}}{{{\bf M}(k)}} \diamond e_k\right)
\\&=&
0 
\end{eqnarray*}
\item ({\em cf.}  \cite[Lemma~25.1.8]{SPES}) In order to prove the claim by induction, it is sufficient to show
that, for each redundant $\{i,j\}, 1\leq i < j \leq u, {\bf e}_{l_i} ={\bf e}_{l_j} =: \epsilon,$ 
there are
\begin{itemize}
\renewcommand\labelitemi{\bf --}
\item  $\{i_1,j_1\}, \ldots, \{i_\rho,j_\rho\}, \ldots, \{i_r,j_r\}, 1\leq i_\rho < j_\rho  \leq u$,  
${\bf e}_{l_{i_\rho}} ={\bf e}_{l_{j_\rho}} = \epsilon$,
\item elements $t_1,\ldots,t_r \in {\Cal T},$ 
\item and coefficients $c_1,\ldots c_r \in R\setminus\{0\}$
\end{itemize}
such that
\begin{itemize}
\item $b(i,j) = \sum_\rho c_\rho t_\rho\diamond b(i_\rho,j_\rho);$
\item $\tau(i,j) = t_\rho\circ \tau(i_\rho,j_\rho),$ for each $\rho;$
\item $\{i_\rho,j_\rho\} \prec  \{i,j\}$.
\end{itemize}

In order to show this, we only need to consider the representation 
$$b(i,j) = \frac{{\bf M}(i,j,k)}{ {\bf M}(i,k)}\diamond b(i,k) - \frac{{\bf M}(i,j,k)}{ {\bf M}(k,j)}\diamond b(k,j)$$
and to prove that 
$$\{i,k\} \prec  \{i,j\} \succ  \{k,j\};$$ 
this happens (according to the two cases of the definition)
because
\begin{enumerate}
\item 
$\tau(i,k) \mid \tau(i,j,k) = \tau(i,j) \neq \tau(i,k)$ implies
$\{i,k\} \prec 
\{i,j\}$  and the same argument proves $\{j,k\} \prec  \{i,j\};$
\item the same argument as that above proves $\{j,k\} \prec  \{i,j\},$  while
$\{i,k\} \prec  \{i,j\}$ because $\tau(i,k) \leq\tau(i,j)$ and $k < j.$
\end{enumerate}
\item ({\em cf.}  \cite[Lemma~25.1.9]{SPES})\index{Moeller's@M\"oller's!Lemma}
 Let $i < s,  {\bf e}_{l_i} ={\bf e}_{l_s} =: \epsilon, \mu := {\bf M}(i,s)$. Then:
\begin{itemize}
\item if there exists $\mu'\in\overline{M}$ such that ${\bf M}(i_{\mu'},s) = \mu' \mid {\bf M}(i,s) \neq
\mu'$, then since $i_{\mu'} < s$, $\{i,s\}$ is redundant;
\item if $i = i_\mu$ and ${\bf M}(i_\mu)\diamond{\bf M}(s) = {\bf M}(i_\mu,s)$, then  
(${\sf M}$ is an ideal)
$b_W(i_\mu,s)\in{\Frak C}_s$ so that $B_W(i_\mu,s)$ has a restricted Gr\"obner
representation in terms of $G$ by Buchberger's First Criterion;
\item if $i = i_\mu$ and ${\bf M}(i_\mu)\diamond{\bf M}(s) \neq {\bf M}(i_m,s)$ then $\{i_m,s\}\in \GM$;
\item if $i \neq i_\mu$ then 
$b(i,s) = \frac{{\bf M}(i,i_\mu,s)}\diamond{{\bf M}(i,i_\mu)} b(i,i_\mu) - b(i_\mu,s).$
 \end{itemize}
\end{enumerate}\end{proof}

\begin{Corollary}   With the present notation, under the assumption that
 $R$ is a principal ideal domain,  
  $F$ is a restricted Gr\"obner basis of ${\sf M}$ iff 
 \begin{enumerate}
\item each S-polynomial $B_W(i,j), \{i,j\}\in\Frak{R}$, 
has
 a restricted Gr\"obner representation in terms of $F$;
\item for each 
$g_j\in F$, both $Y_i\star g_j, 1\leq i \leq n$  and  each 
$g_j\star  \alpha_{\upsilon}(v),  v\in{\Bcc v}, \upsilon\in{\Cal T},\upsilon<{\bf T}(g_j)$, have a restricted representation in terms of $F$ w.r.t. $<$.
\end{enumerate}
\qed\end{Corollary}
 
\section{Structural Theorem for Multivariate Ore Extensions of Zach\-arias PIDs}

\begin{Theorem}[Structural Theorem] 
Let $R$ be a left Zacharias principal ideal domain, ${\sf R} := R[Y_1,\ldots,Y_n]$ a multivariate Ore extension of $R$,
 $<$ a term-ordering, ${\sf M} \subset{\sf R}^m$ a left module generated 
by a basis 
$F := \{g_1,\ldots,g_u\}\subset{\sf M}$, ${\bf M}(g_i) = c_i \tau_i{\bf e}_{l_i}$,
$C(F)$ a completion of $F$,  
${\Frak R} := \{B(i,j), 1\leq i < j \leq u, {\bf e}_{l_i} = {\bf e}_{l_j} 
\mbox{\ and not redundant}\}.$ 

Then the following conditions are equivalent:
\begin{enumerate}
\renewcommand\theenumi{{\rm (\arabic{enumi})}}
\item $F$ is a left Gr\"obner basis of ${\sf M}$;
\renewcommand\theenumi{{\rm ($\arabic{enumi}_s$)}}
\setcounter{enumi}{0}
\item $C(F)$ is a left strong  Gr\"obner basis of ${\sf M}$;
\renewcommand\theenumi{{\rm (\arabic{enumi})}}
\item ${\Cal B}(F) := \{\lambda g : \lambda\in {\Cal T}, g\in F\}$ is a Gauss generating
set \emph {\cite[Definition 21.2.1]{SPES}};  
\item $f \in {\sf M}\iff$ it has a  left Gr\"obner representation in terms of $F$;
\item $f \in {\sf M}\iff$ it has a left  strong  Gr\"obner representation in terms of $C(F)$; 
\item for each $f\in{\sf R}^m\setminus\{0\}$ and any normal form
$h$ of $f$ w.r.t. $F$, we have
$$f\in{\sf M} \iff h = 0;$$
\renewcommand\theenumi{{\rm ($\arabic{enumi}_s$)}}
\setcounter{enumi}{4}
\item for each $f\in{\sf R}^m\setminus\{0\}$ and any strong normal form
$h$ of $f$ w.r.t. $C(F)$, we have
$$f\in{\sf M} \iff h = 0;$$
\renewcommand\theenumi{{\rm (\arabic{enumi})}}
\item for each $f\in{\sf R}^m\setminus\{0\}, f - \Can(f,{\sf M})$ has a strong 
 Gr\"obner representation in terms of $C(F)$;
\item each $B(i,j)\in\Frak{R}$  has a weak  Gr\"obner
representation in terms of $F$;
\item  for each element $\sigma$ of a Gebauer--M\"oller set for $F$,
the  S-polynomial ${\Frak S}_L(\sigma)$ has a left quasi-Gr\"obner representation in terms of $F$.
\end{enumerate}
\end{Theorem}

\begin{proof} \

\begin{description}
\item[$(1) \iff (1_s)$] is Proposition~\ref{48xP1};
\item[$(1) \iff (2)$] is trivial;
\item[$(1) \iff (5) \iff (3)$] is Proposition~\ref{57P1};
\item[$(1_s) \iff (4) \iff (5_s)$] is Proposition~\ref{57P1};
\item[$(1) \then (6)$] is the content of Section~\ref{c46S11E};
\item[$(6) \then (4)$] because for each $f\in{\sf M}, \Can(f,{\sf M}) = 0$;
 \item[$(1) \iff (7)$] is Corollary~\ref{48Cz1};
\item[$(1) \iff (3) \iff (8)$] is Theorem~\ref{SAASSAAS}.
\end{description}\end{proof} 

\section{Spear's Theorem}

For  Gr\"obner bases in a ring ${\Cal A}$ 
given as   quotient
$$\Pi : {\Cal Q} :={\Bbb Z}\langle Y_1,\ldots,Y_n \rangle \twoheadrightarrow {\Cal A} \cong {\Cal Q}/{\Cal I},
\quad {\Cal I} := \ker(\Pi)$$
 of a free assocative algebra,
a general approach is to directly apply  Spear's Theorem  \cite{Sp}  \cite[Proposition~24.7.3]{SPES}  \cite{Bath}, which, while not  a tool for  computation,  can be helpful in order to understand 
the structure  of ${\Cal A}$.

For the present setting, denoting
\begin{itemize}
\renewcommand\labelitemi{\bf --}
\item $f_{ij} := Y_jY_i - \alpha_j(Y_i)Y_j-\delta_j(Y_i), 1\leq i < j \leq n$,
\item $C := \{f_{ij} : 1\leq i < j \leq n\}$;
\item ${\Cal I} := {\Bbb I}_2(C)$
\end{itemize}
and for each $m\in{\Bbb N}$,
\begin{itemize}
\renewcommand\labelitemi{\bf --}
\item $\{{\bf e}_1,\ldots,{\bf e}_m\}$ the  canonical basis of ${\sf R}^m$,
\item $C^{(m)} := \{f_{ij\iota}  : = f_{ij} {\bf e}_\iota : 1\leq i < j \leq n, 1\leq \iota \leq m\}$,
\end{itemize}
we have the presentation
$${\sf R}= {\Cal Q}/{\Cal I},  {\Cal I} := \ker(\Pi),
\Pi : {\Cal Q} := R\langle Y_1,\ldots,Y_n \rangle \twoheadrightarrow {\sf R}$$
and, for each free  ${\sf R}$-module ${\sf R}^m, m\in{\Bbb N}$,  
the projection $\Pi$  extends to the canonical projections, still denoted $\Pi$,
$$\Pi : {\Cal Q}^m \twoheadrightarrow {\sf R}^m, \ker(\Pi) = {\Cal I}^m
={\Bbb I}_2(C^{(m)}).$$

Thus denoting
\begin{itemize}
\renewcommand\labelitemi{\bf --}
\item $F := \{g_1,\ldots,g_u\}\subset {\sf R}^m, 
g_i = {\bf M}(g_i)-p_i =: c_i \tau_i {\bf e}_{l_i} - p_i,$
\item  ${\sf M}\subset  {\sf R}^m$  the module ${\sf M} := {\Bbb I}_2(F)$,
\item ${\sf M}' : =\Pi^{-1}({\sf M}) = {\sf M}+{\Cal I}^m\subset {\Cal Q}^m$,
\end{itemize}
we can reformulate Spear's result as
\begin{Lemma}\label{Le} \cite[Lemma 12]{Bath} Assume $F\subset{\sf M}'$ is a Gr\"obner basis of ${\sf M}'$ and
denote
$$\bar{F} := \{\Can(g,{\Cal I}^m) : g\in F,  {\bf T}(g)\in{\Cal T}^{(m)}\}\subset R[Y_1,\ldots,Y_n]^m$$ where  $\Can(g,{\Cal I}^m)$ denotes the canonical form of  $g\in {\Cal Q}^m$ w.r.t.
$C^{(m)}$ so that in particular $g=\Pi(g)$ for each $g\in\bar{F}$.

Then  $\bar{F}\sqcup C^{(m)}$ is a Gr\"obner basis of ${\sf M}'$.
\end{Lemma}

 \begin{Theorem}[Spear] \cite[Theorem 13]{Bath}
\index{Spear's Theorem}
With the present notation, the following holds:
\begin{enumerate}
\item if $F$ is a reduced Gr\"obner basis of ${\sf M}'$, then 
\begin{eqnarray*}
\{g\in F : g = \Pi(g)\} &=&
\{\Pi(g) : g\in F, {\bf T}(g)\in{\Cal T}^{(m)}\} 
\\ &=&
 F\cap R[Y_1,\ldots,Y_n]^m
 \end{eqnarray*}
is a reduced Gr\"obner basis of ${\sf M}$;
\item if $F\subset R[Y_1,\ldots,Y_n]^m$, so that in particular $\Pi(f)=f$ for each $f\in F$,
 is  the Gr\"obner basis of ${\sf M}$, then $F\sqcup C^{(m)}$ is a   Gr\"obner basis of ${\sf M}'$.
\item Assume each  $m'\in{\sf M}'$ has a Gr\"obner representation in terms of $F\subset{\sf M}'$.

Set
$$\bar{F} := \{\Can(g,{\Cal I}^m) : g\in F, g\notin {\Cal I}^m\}\subset R[Y_1,\ldots,Y_n]^m$$ where
 $\Can(g,{\Cal I}^m)\in R[Y_1,\ldots,Y_n]^m$ denotes the canonical form of  $g\in {\Cal Q}^m$ w.r.t.
$C^{(m)}$ so that in particular $g=\Pi(g)$ for each $g\in\bar{F}$.

Then  each
$m\in{\sf M}$
 has a Gr\"obner representation in terms of $\bar{F}$.
\item if $F\subset R[Y_1,\ldots,Y_n]^m$, so that in particular $\Pi(f)=f$ for each $f\in F$,
is such that each
$m\in{\sf M}$
 has a Gr\"obner representation in terms of $F$, then
each $m'\in{\sf M}'$
 has a Gr\"obner representation in terms of  $F\sqcup C^{(m)}$.
\end{enumerate}
\end{Theorem}

\begin{Corollary} \cite[Corollary 14]{Bath}\label{SpeCo} With the present notation and considering
\begin{itemize}
\renewcommand\labelitemi{\bf --}
\item the bilateral ${\sf R}$-module  $\left({\sf R}\otimes_{\hat{R}}{\sf R}^{\op}\right)^u$ with canonical basis
$\{e_1,\ldots,e_u\}$
\item the bilateral ${\Cal Q}$-module  $\left({\Cal Q}\otimes_{\hat{R}}{\Cal Q}^{\op}\right)^{\vert F\vert+m\vert G\vert}$ with canonical basis
$$\{e_1,\ldots,e_u\}\sqcup\{{\sf e}_{ij\iota} : 1\leq i < j \leq n, 1\leq \iota \leq m\},$$
\item the projections ${\Frak S}_2:  \left({\sf R}\otimes_{\hat{R}}{\sf R}^{\op}\right)^{\vert F\vert}\to{\sf R}^m : 
{\Frak S}_2(e_i) = g_i, 1\leq i \leq u$,  and
\item $\hat{\Frak S}_2 :  \left({\Cal Q}\otimes_{\hat{R}}{\Cal Q}^{\op}\right)^{\vert F\vert+m\vert C\vert} \to {\Cal Q}^m : $
$$\hat{\Frak S}_2(e_i) = g_i, 1\leq i \leq u,\quad
\hat{\Frak S}_2({\sf e}_{ij\iota}) = f_{ij\iota}, 1\leq i < j \leq n, 1\leq \iota \leq m,$$
\item  the  map  $$\bar\Pi :  
\left({\Cal Q}\otimes_{\hat{R}}{\Cal Q}^{\op}\right)^{\vert F\vert+m\vert C\vert} \to \left({\sf R}\otimes_{\hat{R}}{\sf R}^{\op}\right)^{\vert F\vert}$$
(where each
$\lambda,\rho\in R\langle{Y_1,..., Y_n}\rangle,
a,b\in R\setminus\{0\}$)
$$\bar\Pi\left(\sum_i a_i \lambda_i e_i  b_i \rho_i+
\sum_{ij\iota} a_{ij\iota} \lambda_{ij\iota} {\sf e}_{ij\iota}  b_{ij\iota} \rho_{ij\iota}
\right)
=
 \sum_{i} a_i \Pi(\lambda_i) e_i b_i \Pi(\rho_i),
$$
\end{itemize}
if $\Sigma\subset\left({\Cal Q}\otimes_{\hat{R}}{\Cal Q}^{\op}\right)^{\vert F\vert+m\vert C\vert}$
is a bilateral standard basis of $\ker(\hat{\Frak S}_2)$, then
$\bar\Pi(\Sigma)$ is a bilateral standard basis of $\ker({\Frak S}_2)$.
\end{Corollary}

\protect\markright{\thesection\hskip5pt
Lazard Structural Theorem for  Ore Extensions over PIDs}
\section{Lazard Structural Theorem for  Ore Extensions over a Principal Ideal Domain}
\protect\markright{\thesection\hskip5pt
Lazard Structural Theorem for  Ore Extensions over PIDs}

\index{Lazard, D.!Theorem|(}
 Let ${\Bbb D}$ be a commutative principal ideal domain, ${\sf R}:= {\Bbb D}[Y;\alpha,\delta]$ be an Ore extension  and  
${\sf I} \subset {\sf R}$ be a bilateral 
ideal.

Let $F := \{f_0,f_1,\ldots,f_k\}$ be a reduced minimal strong bilateral Gr\"obner basis of ${\sf I}$ ordered so that 
$$\deg(f_0) \leq \deg(f_1) \leq \cdots  \leq \deg(f_k)$$ and let us denote, for each
$i$, 
$c_i := \lc(f_i)$,
$r_i\in {\Bbb D}\setminus\{0\}$ and $p_i\in{\sf R}$ the content\footnote{Defined here as the greatest  common  divisor of the coefficients of $f_i$ in the principal ideal domain ${\Bbb D}$ .}  and the primitive of $f_i$ so that $f_i=r_ip_i;$
denoting $P := p_0$ the primitive part of $f_0$ 
and $G_{k+1} := r_k\in  {\Bbb D}\setminus\{0\}$ the content of $f_k$ we have
 
\begin{Theorem} With the present notation, for each $i, 0\leq   i < k$, there is $H_{i+1}\in {\sf R}, d(i) := \deg(H_i)$ and $G_i\in  {\Bbb D}\setminus\{0\}$ such that
\begin{itemize}
\item $f_0 = G_1\cdots G_{k+1}P,$
\item $f_j =  G_{j+1}\cdots G_{k+1}H_jP, 1 \leq j \leq k,$
\end{itemize}
and
\begin{enumerate}
\item $0 < d(1) < d(2) < \cdots < d(k)$;
\item $G_i\in  {\Bbb D}, 1\leq i \leq k+1$, is such that $c_{i-1}=G_{i}c_{i}$
\item $P=p_0$  (the primitive part of $f_0\in {\sf R})$;
\item $H_i\in {\sf R}$ is a monic polynomial of degree $d(i)$, for each $i$;
\item {\small $H_{i+1} \in (G_1\cdots G_{i},  G_{2}\cdots G_{i}H_1, \ldots, G_{j+1}\cdots G_iH_j,  \ldots, H_{i-1} G_i, H_i)$ for all $i$.}
\item $r_i = G_{i+1}\cdots G_{k}$.
\end{enumerate}
\end{Theorem}

\begin{proof}
Let $P$ and $G_{k+1}$ be, resp., 
the greatest  common right divisor of $\{p_0,\ldots,p_k\}$  in ${\sf R}$
and the greatest  common divisor of $\{r_0,\ldots,r_k\}$ in ${\Bbb D}$;
since a set  $\{g_0,\ldots,g_k\}$ is a minimal strong  Gr\"obner basis   if and only if the
same is true for $\{rg_0g,\ldots,rg_kg\}$, $r\in{\Bbb D},g\in{\sf R}$, we can left divide by $G_{k+1}$ and right
divide by $P$ and assume wlog
that $P = G_{k+1} = 1$ and  that both the greatest  common right divisor of $\{p_0,\ldots,p_k\}$ 
and the greatest  common   divisor of $\{r_0,\ldots,r_k\}$ are 1.

Setting  ${\sf d}(i) := \deg(f_i)$ and 
$\nu(i):={\sf d}(i+1)-{\sf d}(i)$  for each $i$, by assumption we have
${\sf d}(i) \leq {\sf d}(i+1)$.

If ${\sf d}(i) = {\sf d}(i+1)$, let us define
$$h := b_if_i+b_{i+1}f_{i+1}\in{\sf I}$$
where $c,b_i,b_{i+1}\in {\Bbb D}$ are such that
$b_ic_i+b_{i+1}c_{i+1} = c$, $c$ being the greatest common divisor of $c_i$ and $c_{i+1}$,
so that $cY^{{\sf d}(i+1)} = {\bf M}(h)\in {\bf M}({\sf I})$;
this implies the existence of $j$ such that ${\bf M}(f_j) \mid   {\bf M}(h) \mid   {\bf M}(f_{i+1})$
contradicting minimality; thus ${\sf d}(i) < {\sf d}(i+1)$  and this, in turn, implies (1) since
$d(i)={\sf d}(i)-\deg(P).$


Both $f_iY^{\nu(i)}$ and $f_{i+1}$ are
in the ideal and have degree ${\sf d}(i+1)$.

Therefore, for
 $c,b_i,b_{i+1}\in {\Bbb D}$ such that
$b_ic_i+b_{i+1}c_{i+1} = c$, $c$ being the greatest common  divisor of $c_i$ and $c_{i+1}$,
$h := b_if_iY^{{\sf d}(i+1)-{\sf d}(i)}+b_{i+1}f_{i+1}\in{\sf I},$
so that
$cY^{{\sf d}(i+1)} = {\bf M}(h)\in {\bf M}({\sf I})$ 
and ${\bf M}(f_j) \mid   {\bf M}(h)$ for some $j$. If
$c_{i+1} \neq c$, necessarily $\deg(f_j)<\deg(f_{i+1})$ whence $j < i+1$ and 
${\bf M}(f_j)\mid  {\bf M}(h)\mid {\bf M}(f_{i+1})$ getting a contradiction.
As a conclusion $c_i=G_{i+1}c_{i+1}$, for some $G_{i+1}\in  {\Bbb D}$ and (2).

Since $G_{i+1} f_{i+1} -f_i Y^{\nu(i)}$ is a polynomial of degree less than
${\sf d}(i+1)$  which reduces to zero by the  Gr\"obner basis, it follows that
 $$G_{i+1}f_{i+1} \in {\Bbb I}(f_0,  \ldots,f_i) \Forall i, 0\leq   i < k;$$
thus, inductively we obtain
$$p_0 \mid_R f_j \Forall j \leq i \then
p_0 \mid_R f_j \Forall j \leq i+1.$$
Also
 \begin{eqnarray*}
c_i \mid_L f_j \Forall j \leq i &\then& 
 G_{i+1}c_{i+1}= c_i\mid_L
G_{i+1}f_{i+1}
\\&\then&
c_{i+1} \mid_L f_j \Forall j \leq i+1.
\end{eqnarray*}

Therefore, the assumptions that  the greatest  common right divisor of $\{p_0,\ldots,p_k\}$ 
and the greatest  common   
divisor of $\{r_0,\ldots,r_k\}$ are 1
 imply that 
$ p_0 =  c_k=1$  proving (3); thus in particular $f_0=c_0$ so that $c_0\mid  f_0$ and this is sufficient to deduce, by the inductive argument, that  each $c_i$ left-divides
$f_i$ and therefore coincides with $r_i$.

Inductively we obtain
$$ r_i\lc(P) = c_i = G_{i+1} c_{i+1} =  
 G_{i+1} r_{i+1}\lc(P) =
 G_{i+1}\cdots G_{k}\lc(P)
$$
thus  proving (6); 
defining $H_i$ the polynomial s.t. $c_iH_iP = f_i$ for all $i$ we have $\lc(H_i) = 1$ (proving (4)), $d(i) + \deg(P) = \deg(f_i)$ and
$G_{i+1} f_{i+1}\in(f_0,\ldots,f_i)$  which proves  (5) dividing  out $G_{i+1}\cdots G_{k}$.
\end{proof}   
\appendix
\section{The PIR case}

While an understandable {\em timor} restrained us to violate Ore's {\em tabu} requiring degree preservation of product, it is well-known that Zacharias--M\"oller results are naturally stated for polynomials over PIRs and the restriction to PIDs is unnatural; we therefore sketch here the few modifications to the theory which are required  in order to adapt it to Ore extensions ${\sf R}$ over a PIR $R$.

The first delicate adaptation is required by formula (\ref{c46Eq3}); the natural solution is due to Gateva \cite{Ga1,Ga2,Ga3} which considered valuation over the semigroup with zero  ${\Cal T}\cup\{o\}$ instead of ${\Cal T}$ setting
$$ \circ :  {\bf N} \times  {\bf N}\to {\bf N}\cup\{o\} :
u\circ v = \begin{cases}{\bf T}(u\star v)&u\star v\neq 0\cr
o&u\star v=0.\cr\end{cases}$$
Her theory however apply only to domains.

Thus in order to extend Corollary~\ref{c46CoX2X2X2} we need to reformulate it as

\setcounter{Theorem}{14}
\begin{Corollary}  
If $\prec$ is a term  ordering on ${\Cal T}$ 
and $<$ is a $\prec$-compatible
term  ordering on ${\Cal T}^{(m)}$,
then, for each $l,r\in{\sf R}$ and $f\in{\sf R}^{(m)}$,
\begin{enumerate}
\item ${\bf M}(l\star  f)= {\bf M}( {\bf M}(l)\star{\bf M}(f))$  provided $\lc(l)\lc(f) \neq 0$;
\item ${\bf M}(f\star r)=  {\bf M}({\bf M}(f)\star {\bf M}(r))$  provided $\lc(f)\lc(r) \neq 0$;
\item ${\bf M}(l\star  f\star r)={\bf M}({\bf M}(l)\star{\bf M}(f)\star {\bf M}(r))$ provided $\lc(l)\lc(f)\lc(r) \neq 0$.
\item ${\bf T}(l\star   f)\leq {\bf T}(l)\circ{\bf T}(f)$  equality holding provided that $\lc(l)\lc(f) \neq 0$;
\item ${\bf T}(f\star r)\leq  {\bf T}(f)\circ{\bf T}(r)$  equality holding provided that $\lc(f)\lc(r) \neq 0$;
\item ${\bf T}(l \star  f\star r)\leq  {\bf T}(l)\circ{\bf T}(f)\circ{\bf T}(r)$  equality holding provided that $\lc(l)\lc(f)\lc(r) \neq 0$.
\end{enumerate}
If, moreover, $R$ is a domain, then
\begin{enumerate}
\setcounter{enumi}{6}
\item ${\bf T}(l\star    f)= {\bf T}(l)\circ{\bf T}(f)$;
\item ${\bf T}(f\star  r)= {\bf T}(f)\circ{\bf T}(r)$;
\item ${\bf T}(l\star  f\star  r)= {\bf T}(l)\circ{\bf T}(f)\circ{\bf T}(r)$.
\end{enumerate}
\end{Corollary}

As regards Gr\"obner basis computation we remark that  the first and the third algorithms (Section~\ref{1Alg} and~\ref{3Alg} )  apply {\em verbatim} also in the PIR case; in the algorithm in fact we have $\{i\}\in {\Frak H}(F)$ for each $i, 1\leq i \leq u$ and  thus each ${\sf m} := {\bf T}(g_i)\in{\sf H}$ is treated by the algorithm which (if the basis is minimal) produces also the annihilitator syzygy
$$(d_1,\ldots,d_u))\in\Syz_L(v({\sf m})_1,\ldots,v({\sf m})_u),  d_j := \begin{cases}a_j&\mbox{if $j=i$}\cr0&\mbox{otherwise}\cr\end{cases}$$
where we  denote, for each $i\leq u$,  $a_i\in{R}$ the annihilator of ${\Bbb I}(c_i)$.

In the second algorithm (Section~\ref{2Alg})  the inductive seed  becomes
$${\Cal S}_{1}=\{b\in R : b\lc(g_1) = 0\}={\Bbb I}(a_1)\subset R,$$ and, for  each $s, 1< s\leq u$,
$\{s\}$  is basic for  ${\bf T}(g_s)$  provided the basis is minimal.

Therefore 
\setcounter{Theorem}{48}

\begin{Corollary} Assuming that the Zacharias ring $R$ is  principal  and denoting, for each $i,j, 1 \leq i < j \leq u$, ${\bf e}_{l_i} ={\bf e}_{l_j}$,
\begin{eqnarray*}
b(i,j) &:=& 
\frac{\lcm\left(
\alpha_{\{i,j\},i}(c_i),
\alpha_{\{i,j\},j}(c_j)\right)}
{\alpha_{\{i,j\},j}(c_j)} \frac{\lcm(\tau_i,\tau_j)}{\tau_j} e_j
\\ &-&
\frac{\lcm\left(
\alpha_{\{i,j\},i}(c_i),
\alpha_{\{i,j\},j}(c_j)\right)}{\alpha_{\{i,j\},i}(c_i)} \frac{\lcm(\tau_i,\tau_j)}{\tau_i}e_i,
\\B(i,j) &:=&  
\frac{\lcm\left(
\alpha_{\{i,j\},i}(c_i),
\alpha_{\{i,j\},j}(c_j)\right)}
{\alpha_{\{i,j\},j}(c_j)} \frac{\lcm(\tau_i,\tau_j)}{\tau_j}\star g_j 
\\ &-&
\frac{\lcm\left(
\alpha_{\{i,j\},i}(c_i),
\alpha_{\{i,j\},j}(c_j)\right)}{\alpha_{\{i,j\},i}(c_i)}
\frac{\lcm(\tau_i,\tau_j)}{\tau_i}\star g_i
\\ a(i) &:=& a_i e_i\, 
\\ A(i) &:=& a_i\star g_i,
\end{eqnarray*}
we have that 
$$\{b(i,j) : 1 \leq i < j \leq u, {\bf e}_{l_i} ={\bf e}_{l_j}\}\cup\{a(i), i\leq u\}$$ is a Gebauer--M\"oller set for
$F$,
so that 
 $F$ is a   Gr\"obner basis of ${\sf M}$, iff
each $B(i,j)$, $1 \leq i < j \leq u, {\bf e}_{l_i} ={\bf e}_{l_j},$
 and each $A(i), i\leq u$,
  has a weak  Gr\"obner
representation in terms of $F$.
\qed\end{Corollary}

\setcounter{Theorem}{50}
\begin{Corollary} Assuming that the Zacharias ring $R$ is  principal and 
that each $\alpha_i$  is an automorphism
denoting, for each $i,j, 1 \leq i < j \leq u$, ${\bf e}_{l_i} ={\bf e}_{l_j}$,
\begin{eqnarray*}
b(i,j) &:=& 
e_j\alpha^{-1}_{\tau_j}\left(\frac{\lcm(c_i,c_j)}{c_j}\right) \frac{\lcm(\tau_i,\tau_j)}{\tau_j}  -
e_i\alpha^{-1}_{\tau_i}\left(\frac{\lcm(c_i,c_j)}{c_i}\right) \frac{\lcm(\tau_i,\tau_j)}{\tau_i} 
\\B(i,j) &:=&  
g_j\star\alpha^{-1}_{\tau_j}\left(\frac{\lcm(c_i,c_j)}{c_j}\right) \frac{\lcm(\tau_i,\tau_j)}{\tau_j}  -
g_i\star\alpha^{-1}_{\tau_i}\left(\frac{\lcm(c_i,c_j)}{c_i}\right) \frac{\lcm(\tau_i,\tau_j)}{\tau_i} 
\\ a(i) &:=& e_i\alpha^{-1}_{\tau_i}(a_i), 
\\ A(i) &:=& g_i\star\alpha^{-1}_{\tau_i}(a_i),
\end{eqnarray*}
we have that 
$$\{b(i,j) : 1 \leq i < j \leq u, {\bf e}_{l_i} ={\bf e}_{l_j}\}\cup\{a(i), i\leq u\}$$ 
is a Gebauer--M\"oller set for $F$,
so that 
 $F$ is a  right Gr\"obner basis of ${\sf M}$, iff
each $B(i,j)$, $1 \leq i < j \leq u, {\bf e}_{l_i} ={\bf e}_{l_j},$ and each $A(i), i\leq u$,  has a right weak  Gr\"obner
representation in terms of $F$.
\qed\end{Corollary}

\setcounter{Theorem}{57}

\begin{Corollary}  
With the present notation, under the assumption that
 $R$ is a principal ideal ring,  
  $F$ is a    Gr\"obner basis of ${\sf I}$ iff, denoting
$\GM$  a useful S-pair set  for $F$,
each S-polynomial $B(i,j), \{i,j\}\in\GM$, and each $A(i), 1\leq i\leq u$, 
 has
 a   Gr\"obner representation in terms of $F$.
\qed\end{Corollary}

\setcounter{Theorem}{60}

\begin{Corollary}   With the present notation, under the assumption that
 $R$ is a principal ideal domain,  
  $F$ is a    Gr\"obner basis of ${\sf M}$ iff 
each S-polynomial $B(i,j), \{i,j\}\in\Frak{R}$, and each $A(i), 1\leq i\leq u$, 
has
 a   Gr\"obner representation in terms of $F$.
\qed\end{Corollary}

\begin{Corollary}   With the present notation, under the assumption that
 $R$ is a principal ideal ring,  
  $F$ is a    Gr\"obner basis of ${\sf M}$ iff 
 \begin{enumerate}
\item \item each S-polynomial $B_W(i,j), \{i,j\}\in\Frak{R}$,  and each $A(i), 1\leq i\leq u$, 
has
 a restricted Gr\"obner representation in terms of $F$;
\item for each 
$g_j\in F$, both $Y_i\star g_j, 1\leq i \leq n$  and  each 
$g_j\star  \alpha_{\upsilon}(v),  v\in{\Bcc v}, \upsilon\in{\Cal T},\upsilon<{\bf T}(g_j)$, have a restricted representation in terms of $F$ w.r.t. $<$.
\end{enumerate}
\qed\end{Corollary}
Finally we remark that a 
Lazard Structural Theorem for  Ore Extensions over a Principal Ideal Domain can be easily obtained by adapting the result given by Norton--S\u al\u agean \cite{NS1,NS2}, \cite[\S~33.3]{SPES} for polynomial rings.

\end{document}